\documentclass{amsart}
\usepackage{amssymb,amsmath}
\usepackage{stmaryrd}
\usepackage{yhmath}
\usepackage{latexsym}
\usepackage{amsthm}
\usepackage{amscd}
\usepackage{mathrsfs}
\usepackage[all]{xy}
\usepackage{mathtools}
\usepackage{arydshln}
\usepackage{bbm}
\usepackage{bm}
\usepackage{url}
\usepackage{color}
\usepackage{threeparttable}
\usepackage{booktabs}
\usepackage{braket}

\newcommand{\migi}{\mathrm{right}}

\newcommand{\nr}{\mathrm{nr}}

\newcommand{\red}{\mathrm{red}}
\newcommand{\der}{\mathrm{der}}

\newcommand{\w}{\mathrm{w}}

\newcommand{\Frob}{\mathrm{Frob}}
\newcommand{\lrFrob}{\langle\mathrm{Frob}\rangle}

\newcommand{\mcA}{\mathcal{A}}
\newcommand{\mcB}{\mathcal{B}}
\newcommand{\mcC}{\mathcal{C}}
\newcommand{\mcF}{\mathcal{F}}
\newcommand{\mcH}{\mathcal{H}}
\newcommand{\mcI}{\mathcal{I}}
\newcommand{\mcM}{\mathcal{M}}
\newcommand{\mcO}{\mathcal{O}}
\newcommand{\mcP}{\mathcal{P}}

\newcommand{\mcR}{\mathcal{R}}
\newcommand{\mcS}{\mathcal{S}}
\newcommand{\mcT}{\mathcal{T}}

\newcommand{\Z}{\mathbb{Z}}

\newcommand{\R}{\mathbb{R}}
\newcommand{\Q}{\mathbb{Q}}
\newcommand{\C}{\mathbb{C}}
\newcommand{\Gm}{\mathbb{G}_{\mathrm{m}}}
\newcommand{\Ga}{\mathbb{G}_{\mathrm{a}}}
\newcommand{\J}{\mathbf{J}}

\newcommand{\bfA}{\mathbf{A}}
\newcommand{\bfB}{\mathbf{B}}
\newcommand{\bfG}{\mathbf{G}}

\newcommand{\bfL}{\mathbf{L}}
\newcommand{\bfM}{\mathbf{M}}
\newcommand{\bfN}{\mathbf{N}}
\newcommand{\bfP}{\mathbf{P}}
\newcommand{\bfS}{\mathbf{S}}
\newcommand{\bfT}{\mathbf{T}}
\newcommand{\bfU}{\mathbf{U}}
\newcommand{\bfo}{\mathbf{o}}
\newcommand{\bfq}{\mathbf{q}}

\newcommand{\G}{\mathbf{G}}

\newcommand{\mfg}{\mathfrak{g}}

\newcommand{\mfp}{\mathfrak{p}}

\newcommand{\ul}{\underline}
\newcommand{\ol}{\overline}

\DeclareMathOperator{\Res}{Res}
\DeclareMathOperator{\adj}{adjoint}
\DeclareMathOperator{\Std}{Std}

\DeclareMathOperator{\val}{val}

\DeclareMathOperator{\half}{half}
\DeclareMathOperator{\spin}{spin}

\DeclareMathOperator{\vol}{vol}

\DeclareMathOperator{\GL}{GL}

\DeclareMathOperator{\SO}{SO}
\DeclareMathOperator{\SL}{SL}

\DeclareMathOperator{\GSp}{GSp}
\DeclareMathOperator{\GSO}{GSO}

\DeclareMathOperator{\Spin}{Spin}
\DeclareMathOperator{\GSpin}{GSpin}

\DeclareMathOperator{\Ind}{Ind}
\DeclareMathOperator{\nInd}{n-Ind}
\DeclareMathOperator{\Hom}{Hom}
\DeclareMathOperator{\End}{End}
\DeclareMathOperator{\Ker}{Ker}
\DeclareMathOperator{\Gal}{Gal}

\DeclareMathOperator{\diag}{diag}

\DeclareMathOperator{\As}{As}

\theoremstyle{plain}
\newtheorem{thm}{Theorem}[section]
\newtheorem*{thm*}{Theorem}
\newtheorem{prop}[thm]{Proposition}
\newtheorem{lem}[thm]{Lemma}
\newtheorem{cor}[thm]{Corollary}

\theoremstyle{definition}
\newtheorem{defn}[thm]{Definition}

\theoremstyle{remark}
\newtheorem{rem}[thm]{Remark}
\newtheorem*{claim*}{Claim}

\SelectTips{cm}{11}

\title{Iwahori--Hecke algebra and unramified local $L$-functions}

\author{Masao OI}
\address{Department of Mathematics (Hakubi center), Kyoto University, Kitashirakawa, Oiwake-cho, Sakyo-ku, Kyoto 606-8502, Japan.}
\email{masaooi@math.kyoto-u.ac.jp}

\author{Ryotaro SAKAMOTO}
\address{RIKEN Center for Advanced Intelligence Project (AIP), 1-4-1 Nihonbashi, Chuo-ku, Tokyo 103-0027, Japan.} 
\email{ryotaro.sakamoto@riken.jp}

\author{Hiroyoshi TAMORI}
\address{Department of Mathematics, Faculty of Science, Hokkaido University, Kita 10, Nishi 8, Kita-Ku, Sapporo, Hokkaido, 060-0810, Japan.}
\email{tamori@math.sci.hokudai.ac.jp}

\begin{document}

\begin{abstract}
In this paper, we compute the Hecke action of a certain test function on the space of an unramified principal series of a connected reductive group over a non-archimedean local field by using the theory of Iwahori--Hecke algebra.
As an application, we obtain a new expression of the local $L$-functions of unramified representations.
\end{abstract}

\maketitle

\tableofcontents

\section{Introduction}\label{sec:Intro}
Let $\G$ be an unramified connected reductive group over a non-archimedean local field $F$ (i.e., $\G$ is quasi-split and splits over an unramified extension of $F$).
The unramified representations of $\G(F)$ are one of the most fundamental classes in representation theory of the group $\G(F)$.
Their importance can be explained in relation to the global theory, that is, almost all local components of automorphic representations are unramified. 
Hence unramified representations play a basic role in the theory of automorphic representations.
They have been investigated from the early days, and a lot of results have been obtained so far.

One fundamental result on unramified representations is the construction of the local $L$-functions.
Let ${}^{L}\G$ be the $L$-group of $\G$, which is given by the semi-direct product $\hat{\G}\rtimes W_{F}$ of the Langlands dual group $\hat{\G}$ and the Weil group $W_{F}$ of $F$.
According to the conjectural local Langlands correspondence, it is expected that the local $L$-function $L(s,\pi,r)$ is defined for any irreducible smooth representation $\pi$ of $\G(F)$ and a finite-dimensional continuous complex representation $r$ of the $L$-group ${}^{L}\G$.
When the representation $\pi$ is unramified, we can define the Satake parameter of $\pi$, which is a semisimple conjugacy class of ${}^{L}\G$.
Then we can attach the local $L$-function $L(s,\pi,r)$ to any pair of an unramified representation $\pi$ of $\G(F)$ and a finite-dimensional continuous representation $r$ of ${}^{L}\G$.

The aim of this paper is to give a new formula describing the local $L$-functions for unramified representations.
Before we explain the main result of this paper, let us introduce some motivating examples.

The first example is the case of the standard $L$-function of $\GL_{2}$. 
Let $\pi$ be an irreducible unramified representation of $\GL_{2}(\Q_{p})$.
We can take an unramified character $\chi$ of the diagonal maximal torus of $\GL_{2}$ such that $\pi$ is realized as a subquotient of the principal series (normalized parabolic induction) $(I_{\chi},V_{\chi})$ of $\chi$.
Consider the standard representation $\mathrm{Std}$ of the Langlands dual group $\GL_{2}(\C)$ of $\GL_{2}$.
Then, by an easy computation, we can check the following equality:
\[
L(s,\pi,\mathrm{Std})
=
\det\bigl(1-p^{-(s+1/2)} I_{\chi}(U_{J}) \,\big\vert\, V_{\chi}^{J}\bigr)^{-1}.
\]
Here $J$ is the open compact subgroup of $\GL_{2}(\Q_{p})$ defined by
\[
J:=
\biggl\{
\begin{pmatrix}a&b\\c&d\end{pmatrix}\in\GL_{2}(\Z_{p})
\,\bigg\vert\,
 c\in p \Z_{p} 
\biggr\}
\]
and $U_{J}$ is the characteristic function of the open compact subset $J\diag(p,1)J$ normalized so that $U_{J}(\diag(p,1))=\vol(J)^{-1}$.

The second example is the case of the spin $L$-function of $\GSp_{4}$.
We put
 \[
\GSp_{4}
:=
\biggl\{
g\in\GL_{4}
\,\bigg\vert\, 
{}^{t}\!g\begin{pmatrix}&-J_2\\J_2&\end{pmatrix}g=x\begin{pmatrix}&-J_2\\J_2&\end{pmatrix}\text{ for some $x\in \Gm$}
\biggr\},
\]
where $J_2$ denotes the anti-diagonal matrix whose anti-diagonal entries are one. 
We consider the spin representation $\mathrm{Spin}$ of the Langlands dual group $\GSpin_{5}(\C)$ of $\GSp_{4}$.
Let $(\pi, V)$ be an irreducible unramified principal series representation of $\GSp_{4}(\Q_{p})$. 
Then, in \cite[Section~2.4]{MR2636500} (see also \cite[Section~3.4.2]{LSZ}), Taylor established a similar identity to above for the spin $L$-function $L(s,\pi,\mathrm{Spin})$ in his study of $p$-adic family of Siegel modular forms.
More precisely, by using the Siegel parahoric subgroup $J$ of $\GSp_{4}(\Q_{p})$, which is defined by
\[
J
:=
\biggl\{
\begin{pmatrix}A&B\\C&D\end{pmatrix}\in\GSp_{4}(\Q_{p})
\,\bigg\vert\,
A,D\in\GL_{2}(\Z_{p}), B\in M_{2}(\Z_{p}), C\in M_{2}(p\Z_{p})
\biggr\}, 
\]
Taylor proved that
\[
L(s,\pi,\mathrm{Spin}) = \det\bigl(1 - p^{-(s+3/2)} \pi(U_{J}) \,\big\vert\, V^{J}\bigr)^{-1},
\]
where $U_{J}$ is the characteristic function of the open compact subset $J\diag(p,p,1,1)J$ normalized so that $U_{J}(\diag(p,p,1,1))=\vol(J)^{-1}$.

These formulas are, in addition to their original importance in a study of modular forms, also interesting from the purely representation-theoretic viewpoint as follows.
In the definition of the local $L$-functions for unramified representations, we utilize the Satake parameters determined by the Satake isomorphism.
This amounts to looking at the action of the spherical Hecke algebra on the subspace of spherical vectors, which is $1$-dimensional.
For example, in the case of $\GL_{2}$ mentioned above, we consider the action of all elements of $C_{c}^{\infty}(\GL_{2}(\Z_{p})\backslash \GL_{2}(\Q_{p})/\GL_{2}(\Z_{p}))$ (bi-$\GL_{2}(\Z_{p})$-invariant test functions on $\GL_{2}(\Q_{p})$) on the $1$-dimensional subspace $V_{\chi}^{\GL_{2}(\Z_{p})}$ of $\GL_{2}(\Z_{p})$-fixed vectors.
On the other hand, in the above formulas, the local $L$-function is expressed by the characteristic polynomial of the action of only one test function on the subspace whose dimension is the same as the degree of the local $L$-function.
For instance, in the case of $\GL_{2}$, the local $L$-function $L(s,\pi,\Std)$ is described by the action of a single test function $U_{J}$ on the subspace $V_{\chi}^{J}$, which is $2$-dimensional.

In this paper, we establish these kind of formulas for connected reductive groups and general finite-dimensional representations of the Langlands dual groups. 
For simplicity, we assume that $\G$ is split in the rest of this introduction.
Let $\bfT$ be a split maximal torus of $\G$ defined over $F$.
By fixing a Borel subgroup $\bfB$ containing $\bfT$, a dominance is determined on the characters and cocharacters of $\bfT$.
Then, to each dominant cocharacter $\mu$ of $\bfT$, we can associate an open compact subgroup $J_{\mu}$ of $\G(F)$ (see Section \ref{subsec:parahoric}) and a normalized characteristic function $\mathbbm{1}_{\mu}$ of a certain $J_{\mu}$-double coset (see Sections \ref{subsec:Hecke-Iwahori} and \ref{subsec:Hecke-parahoric}).
For a finite-dimensional representation $r$ of the Langlands dual group $\hat{\G}$, we put $\mcP^{+}(r)$ to be the set of dominant weights in $r$.
Note that each element $\mu$ of $\mcP^{+}(r)$ can be regarded as a dominant cocharacter of $\bfT$ through the duality between $\G$ and $\hat{\G}$.
For each $\mu \in\mcP^{+}(r)$, we write $m_{\mu}$ for the multiplicity of $\mu$ in $r$.
The following is the main result of this paper.  

\begin{thm}[Theorem \ref{thm:L} and Remark \ref{rem:main-the-split}]\label{thm:L_{intro}}
Let $\pi$ be an irreducible unramified representation of $\G(F)$.
We take an unramified character $\chi$ of $\bfT(F)$ such that $\pi$ is realized as a subquotient of the normalized parabolic induction $(I_{\chi},V_{\chi})$ of $\chi$.
Then we have an equality 
\[
L(s,\pi,r)
= \prod_{\mu \in\mcP^{+}(r)}
\det\bigl(1-q^{-(s+\langle\rho_{\bfB}, \mu\rangle)}I_{\chi}(\mathbbm{1}_{\mu}) \,\big\vert\, V_{\chi}^{J_{\mu}}\bigr)^{-m_{\mu}},
\]
where $\rho_{\bfB}$ is the half sum of the positive roots of $\bfT$ in $\G$.
\end{thm}

Note that if $(\G,r)$ is $(\GL_{2},\mathrm{Std})$ or $(\GSp_{4},\Spin)$, then the set $\mathcal{P}^{+}(r)$ is a singleton and the formula in Theorem~\ref{thm:L_{intro}} is nothing but the identity in the above examples (see Sections~\ref{subsec:GL} and \ref{subsec:GSp}). 
More generally, when $r$ is a quasi-minuscule representation (see Definition~\ref{def:minuscule}), 
we get a similar formula to the above examples (see Corollary \ref{cor:L2}). 
See Remark~\ref{rem:q-min} and Table~\ref{Table} for a list of $(\G,r)$ such that $\G$ is simple and $r$ is quasi-minuscule. 

We also remark that Theorem \ref{thm:L_{intro}} (Theorem \ref{thm:L}) is proved in a slightly more general setting where $\G$ might not be split and $\pi$ is a parahoric-spherical representation of $\G(F)$ (i.e., an irreducible smooth representation having a nonzero vector fixed by a parahoric subgroup, see Definition \ref{defn:spherical}).
When $\pi$ is not unramified but spherical for some parahoric subgroup, we consider the semisimple $L$-function (see Definition \ref{def:L-function}) instead of the usual $L$-function.

We explain the outline of the proof of Theorem \ref{thm:L_{intro}}.
The key in our proof is that the action of $I_{\chi}(\mathbbm{1}_{\mu})$ on the space $V_{\chi}^{J_{\mu}}$ can be triangulated with respect to an ordered basis of $V_{\chi}^{J_{\mu}}$.
To explain this, we assume that $\mu$ is strictly dominant for simplicity.
In this case, $J_{\mu}$ is an Iwahori subgroup, hence let us simply write $I$ for $J_{\mu}$.
Then we can find an explicit basis $\{v^{\vee}_{w}\}_{w\in W}$ of the subspace $V_{\chi}^{I}$ of $I$-fixed vectors in $V_{\chi}$, which is labelled by the elements of the Weyl group $W$ of $\bfT$ in $\G$.
With respect to this ordered basis of $V_{\chi}^{I}$, we have the following:

\begin{prop}[Proposition \ref{prop:triangle-Iwahori}]\label{prop:triangle-Iwahori_{Intro}}
For any $w\in W$, there exists a family $\{c_{w'}\}_{w'\in W, w'\geq w}$ of complex numbers satisfying
\[
I_{\chi}(\mathbbm{1}_{\mu})\cdot v^{\vee}_{w}
=
c_{w}\cdot v^{\vee}_{w}+\sum_{\begin{subarray}{c}w'\in W \\ w'> w \end{subarray}} c_{w'}\cdot v^{\vee}_{w'}.
\]
Moreover, the number $c_{w}$ can be determined explicitly.
\end{prop}

Once this proposition is proved, we immediately get a description of the characteristic polynomial of the action of $\mathbbm{1}_{\mu}$ on $V_{\chi}^{I}$.
Then we obtain Theorem \ref{thm:L_{intro}} by tracking the construction of the Satake parameter and rewriting the local $L$-function $L(s,\pi,r)$ in terms of the weights of the representation $r$.

Originally, we proved this proposition by making full use of the Chevalley basis by assuming that our group $\G$ is split.
By utilizing various relations of the Chevalley basis, we carried out the induction on the length of $w\in W$; then the problem is essentially reduced to the case of $\SL_{2}$.
Although the basic idea of our original proof is fairly simple in this way, we had to show a lot of technical statements about group-theoretic properties of parahoric subgroups to justify the induction step (cf.\ the older version of this paper; \cite{OST}).

However, after we released the first version of this paper, Thomas Haines told the authors that the above triangularity result can be proved in a more sophisticated way if we appeal to the theory of the Iwahori--Hecke algebra.
Furthermore, he also explained that his approach naturally enables us to prove Proposition \ref{prop:triangle-Iwahori_{Intro}} for any general (i.e., possibly non-split) connected reductive group $\G$.
Hence we decided to follow his idea and present his simplified version of the proof in this paper.

The outline of the proof of Proposition \ref{prop:triangle-Iwahori_{Intro}} is as follows.
We continue to assume that $\G$ is split in the following for simplicity.
We write $\bfN$ for the unipotent radical of $\bfB$ and put $\mcM$ to be the space $C_{c}^{\infty}(\bfT(\mcO_{F})\bfN(F)\backslash \G(F)/I)$, where $I$ denotes an Iwahori subgroup of $\G(F)$.
Then the space $\mcM$ has commuting actions of two kinds of $\C$-algebras; one is the group algebra $\mcR$ of the cocharacter group of $\bfT$, and the other one is the Iwahori--Hecke algebra $\mcH_{I}:=C_{c}^{\infty}(I\backslash\G(F)/I)$.
(See Section \ref{subsec:univ} for the details.)

This space $\mcM$ can be understood as the space of $I$-fixed vectors in the universal unramified principal series.
More precisely, any unramified character $\chi$ of $\bfT(F)$ defines a $\C$-algebra homomorphism from $\mcR$ to $\C$ (let us again write $\chi$).
Then, by specializing the $\mcR$-module $\mcM$ to a $\C$-module via $\chi$, we obtain the space $(\nInd_{\bfB(F)}^{\G(F)}\chi^{-1})^{I}$ of $I$-fixed vectors in the unramified principal series of $\chi^{-1}$, i.e., we have $\C\otimes_{\mcR,\chi}\mcM\cong(\nInd_{\bfB(F)}^{\G(F)}\chi^{-1})^{I}$.
Also, we can find an $\mcR$-basis $\{v_{w}\}_{w\in W}$ of $\mcM$ labelled by the elements of $W$.
With this language, Proposition \ref{prop:triangle-Iwahori_{Intro}} is rephrased as follows:

\begin{prop}[Proposition \ref{prop:Haines}]\label{prop:Haines_{Intro}}
For any $w\in W$, there exists a family $\{a_{w'}\}_{w'\in W, w'\leq w}$ of elements of $\mcR$ satisfying
\[
v_{w}\ast\Theta_{\mu}
=
a_{w}\cdot v_{w}+\sum_{\begin{subarray}{c}w'\in W \\ w'<w \end{subarray}} a_{w'}\cdot v_{w'}.
\]
Moreover, $a_{w}$ can be explicitly determined. 
Here $\Theta_{\mu}$ is an element of the Iwahori--Hecke algebra which is a constant multiple of $\mathbbm{1}_{\mu}$ (see Section \ref{subsec:Bernstein}). 
\end{prop}

The point here is that the ring structure of $\mcH_{I}$ and its action on $\mcR$ are well-investigated, especially in the works of Haines--Kottwitz--Prasad (split case, \cite{MR2642451}) and Rostami (general case, \cite{MR3355113}).
By using several basic relations of the Iwahori--Hecke algebra (e.g., the Bernstein relation, see Proposition \ref{prop:Bernstein}), we can prove Proposition \ref{prop:Haines_{Intro}} by an induction argument on the length of $w\in W$.

It seems that our computations in the previous version of the proof are essentially encoded in the various identities in the theory of the Iwahori--Hecke algebra.
In this sense, the core of the new proof presented in this paper is not totally different to our original proof.
Nevertheless, we would like to emphasize that most of the arguments are drastically simplified and our main result is far more generalized by following the formulation suggested by Haines.

\subsection*{Acknowledgments}
The authors are grateful to Miyu Suzuki for encouragement and constructive advice on a draft of this paper.
The authors also thank to Hiraku Atobe, Yoichi Mieda, and Lei Zhang for their helpful comments.
Finally, the authors express their sincere gratitude to Thomas Haines for his detailed explanation about how to prove our result via Iwahori--Hecke algebras.
He also kindly answered a lot of questions by the authors and encouraged them.

This work was supported by the Program for Leading Graduate Schools, MEXT, Japan and JSPS KAKENHI Grant Number 17J05451 and 20K14287 (Oi), 17J02456 (Sakamoto), and 17J01075 and 20J00024 (Tamori).
R.S. was also supported by RIKEN Center for Advanced Intelligence Project (AIP).

\subsection*{Notations and conventions}

Let $F$ be a non-archimedean local field.
We let $\mcO$, $\mfp$, and $k$ denote the ring of integers, its maximal ideal, and its residue field of $F$, respectively.
Let $q$ be the order of $k$.
We write $W_{F}$ and $I_{F}$ for the Weil group of $F$ and the inertia subgroup, respectively.
We fix a lift $\Frob$ of the geometric Frobenius in $\Gal(\overline{k}/k)$ (i.e., $x\mapsto x^{q^{-1}}$) to $W_{F}$.

For an algebraic variety $\mathbf{J}$ over $F$ (written by the bold letter), we let $J := \mathbf{J}(F)$ (written by the usual italic letter) denote the set of its $F$-valued points.
For an algebraic group $\bfT$, we write $X^{\ast}(\bfT)$ (resp.\ $X_{\ast}(\bfT)$) for the groups of characters $\Hom(\bfT,\Gm)$ (resp.\ cocharacters $\Hom(\Gm,\bfT)$) of $\bfT$.
When an algebraic group $\bfT$ is defined over $F$, we write $X^{\ast}(\bfT)_{F}$ and $X_{\ast}(\bfT)_{F}$ for the groups of $F$-rational characters and cocharacters of $\bfT$, respectively.

For an abelian group $M$, we write $M_{\R}$ for $M\otimes_{\Z}\R$.

\section{Iwahori subgroup and Iwahori--Hecke algebra}\label{sec:Iwahori}

In this section, we review the fundamental properties of the Iwahori--Hecke algebra needed for us.
The content of this section is based on the paper \cite{MR2642451} of Haines--Kottwitz--Prasad and also the paper \cite{MR3355113} of Rostami, which generalizes the results of \cite{MR2642451} from the split case to the non-split case.

\subsection{Iwahori subgroup and Kottwitz homomorphism}\label{subsec:Iwahori}
Let $\G$ be a connected reductive group over $F$.
We write $\mcB(\G,F)$ (resp.\ $\mcB_{\red}(\G,F)$) for the Bruhat--Tits building (resp.\ reduced Bruhat--Tits building) of $\G$ over $F$.
We fix a point $\bfo\in\mcB(\G,F)$ whose image in $\mcB_{\red}(\G,F)$ is a special vertex.
Let $K$ denote the special maximal parahoric subgroup of $G$ associated with $\bfo$.
We fix a maximal $F$-split torus $\bfA$ of $\G$ whose apartment $\mcA(\bfA,F)$ contains the point $\bfo$.
Note that, by using the fixed special point $\bfo$, the apartment $\mcA(\bfA,F)$ is identified with $X_{\ast}(\bfA)_{\R}$:
\[
X_{\ast}(\bfA)_{\R}\cong\mcA(\bfA,F)\colon \mu \mapsto \bfo+\mu.
\]

We furthermore fix an Iwahori subgroup $I$ contained in $K$.
Then $I$ determines an alcove $\mathcal{C}$ of the apartment $\mcA(\bfA,F)$ whose closure contains the special point $\bfo$.
Let $\Phi:=\Phi(\G,\bfA)$ be the set of roots of $\bfA$ in $\G$.
Then the alcove $\mathcal{C}$ determines a system $\Phi^{+}$ (resp.\ $\Phi^{-}$) of positive (resp.\ negative) roots in $\Phi$.
We put $\Phi_{\red}$ to be the set of reduced roots in $\Phi$ and put $\Phi_{\red}^{\pm}:=\Phi^{\pm}\cap\Phi_{\red}$.
We write $\Delta$ for the set of simple roots.

Let $\bfM$ be the centralizer of the fixed maximal $F$-split torus $\bfA$ in $\G$, which a minimal $F$-rational Levi subgroup of $\G$.
Let $\bfP$ be the minimal parabolic subgroup with Levi factor $\bfM$ such that the corresponding set of positive roots is given by $\Phi^{+}$.
We write $\kappa_{M}$ for the Kottwitz homomorphism for $M$ (see \cite[Section 7.7]{MR1485921}):
\[
\kappa_{M}\colon M\twoheadrightarrow X^{\ast}(Z(\hat{\bfM})^{I_{F}})^{\Frob},
\]
where 
\begin{itemize}
\item
$\hat{\bfM}$ is the Langlands dual group of $\bfM$,
\item
$(-)_{I_{F}}$ denotes the group of $I_{F}$-coinvariants, and
\item
$(-)^{\Frob}$ denotes the group of Frobenius invariants.
\end{itemize}
In the following, we simply write $\Lambda_{M}$ for $X^{\ast}(Z(\hat{\bfM})^{I_{F}})^{\Frob}$.
We put 
\[
M_{1}:=\Ker(\kappa_{M}\colon M\twoheadrightarrow \Lambda_{M}).
\]
Thus we have an identification $M/M_{1}\cong\Lambda_{M}$.
For an element $\mu\in\Lambda_{M}$, we write $\ul{\mu}$ for the inverse image $\kappa_{M}^{-1}(\mu)$ of $\mu$ in $M/M_{1}$ (we often loosely regard $\ul{\mu}\in M/M_{1}$ as an element of $M$ as long as it does not cause any confusion).

According to \cite[Section 5.2]{MR3355113}, we introduce a dominance on $\Lambda_{M}$ as follows.
We put $\nu_{M}\colon M\rightarrow\Hom(X^{\ast}(\bfM)_{F},\Z)$ to be the homomorphism defined by
\[
\nu_{M}(m):=[\chi\mapsto\val_{F}(\chi(m))].
\]
Then there exists a homomorphism $q_{M}\colon \Lambda_{M}\rightarrow\Hom(X^{\ast}(\bfM)_{F},\Z)$ such that $q_{M}\circ\kappa_{M}=\nu_{M}$.
By tensoring $\R$ over $\Z$ and composing with a natural isomorphism $\Hom(X^{\ast}(\bfM)_{F},\R)\cong X_{\ast}(\bfA)_{\R}$, we get an identification $\Lambda_{M,\R}\xrightarrow{\cong}X_{\ast}(\bfA)_{\R}$:
\[
\xymatrix{
M \ar^-{\kappa_{M}}[r] \ar_-{\nu_{M}}[rd] & \Lambda_{M} \ar^-{q_{M}}[d] \ar[rr] & & \Lambda_{M,\R} \ar^-{\cong}[d]\\
&\Hom(X^{\ast}(\bfM)_{F},\Z) \ar[r]& \Hom(X^{\ast}(\bfM)_{F},\R) \ar^-{\cong}[r]& X_{\ast}(\bfA)_{\R}
}
\]
(see \cite[Sections 2.5--2.7]{MR3355113} for details).
Hence we can transport a dominance on $X_{\ast}(\bfA)_{\R}\,(\cong\mcA(\bfA,F))$, which is determined by the alcove $\mathcal{C}$, to $\Lambda_{M,\R}$.
We say that an element $\mu$ of $\Lambda_{M}$ is dominant if its image in $\Lambda_{M,\R}$ is dominant.

For any $\mu\in\Lambda_{M}$ and $\alpha\in X^{\ast}(\bfA)_{\R}$, 
we often simply write $\langle\alpha,\mu\rangle$ for $\langle \alpha,q_{M}(\mu)\rangle$, which is the value at $(\alpha,q_{M}(\mu))$ of the natural pairing $\langle-,-\rangle$ on $X^{\ast}(\bfA)_{\R}\times X_{\ast}(\bfA)_{\R}$.

\begin{rem}\label{rem:q}
In \cite{MR3355113}, $\kappa_{M}$ and $\nu_{M}$ are defined to be $-\kappa_{M}$ and $-\nu_{M}$, respectively (see \cite[Section 2.7, 519 page]{MR3355113}).
Since $q_{M}$ is not affected by the difference of these normalizations (the sign differences cancel out), the identification between $\Lambda_{M,\R}$ and $X_{\ast}(\bfA)_{\R}$ in this paper is the same as that in \cite{MR3355113}.
\end{rem}

\subsection{Iwahori--Weyl group}\label{subsec:Iwahori-Weyl}

Let $\tilde{W}$ denote the Iwahori--Weyl group defined by
\[
\tilde{W}:=N_{\G}(\bfA)(F)/M_{1},
\]
where $N_{\G}(\bfA)$ is the normalizer group of $\bfA$ in $\G$.
We write $W:=W_{\G}(\bfA)(F)=(N_{\G}(\bfA)/\bfM)(F)$ for the relative Weyl group of the relative root system $\Phi$.
Then we have a short exact sequence (see \cite[Lemma 3.1.1]{MR3355113})
\[
1
\rightarrow \Lambda_{M}
\rightarrow \tilde{W}
\rightarrow W
\rightarrow1.
\]
Let $W_{\bfo}$ be the subgroup of $\tilde{W}$ generated by the reflections with respect to the walls of the fixed alcove $\mathcal{C}$ containing the point $\bfo$.
By \cite[Lemma 5.0.1]{MR2602034}, the natural map
\[
W_{\bfo}\subset\tilde{W}\twoheadrightarrow W
\]
is bijective since $\bfo$ is a special point.
Accordingly, we can express $\tilde{W}$ as a semi-direct product (i.e., $W$ is regarded as a subgroup of $\tilde{W}$ through the splitting $W\xrightarrow{1:1}W_{\bfo}$):
\[
\tilde{W} \cong \Lambda_{M}\rtimes W.
\]
See \cite[Sections 2.8 and 2.9]{MR3355113} and also \cite{MR3481263} for the details.

\subsection{Parahoric subgroups}\label{subsec:parahoric}
For any facet $\mcF$ of the apartment $\mcA(\bfA,F)$, we let $J_{\mcF}$ denote the parahoric subgroup associated with $\mcF$.
Note that then, with this notation, we have $K=J_{\bfo}$ and $I=J_{\mcC}$.

The fixed special point $\bfo$ defines ``a valuation of root datum'' of $\G$, which consists of group-theoretic data satisfying several axiomatic properties (see \cite[Section 6.1]{MR0327923} for the definition of a valuation of root datum).
In particular, for each $\alpha\in\Phi$, the root subgroup $U_{\alpha}=\bfU_{\alpha}(F)$ of $\G$ has a descending filtration $\{U_{\alpha,r}\}_{r\in\R}$.

\begin{rem}\label{rem:Chevalley}
When $\G$ is split, the choice of a special point $\bfo$ of the Bruhat--Tits building $\mcB(\G,F)$, or equivalently, its associated valuation of root data can be made explicitly in terms of a Chevalley basis.
More precisely, a Chevalley basis of $\G$ consists of homomorphisms $x_{\alpha}\colon \Ga\to \bfU_{\alpha}\subset\G$ for each $\alpha\in\Phi$ satisfying several axioms, where $\bfU_{\alpha}$ denotes the root subgroup of $\alpha$ in $\G$ (cf.\ \cite[page 21, Corollary 1]{MR3616493}).
Then, for $\alpha\in\Phi$, the filtration $\{U_{\alpha,r}\}_{r\in\R}$ of $U_{\alpha}=\bfU_{\alpha}(F)$ is given by $U_{\alpha,r}=x_{\alpha}(\{a\in F \mid \val_{F}(a)\geq r\})$.
\end{rem}


For a dominant element $\mu\in\Lambda_{M}$, we define an open compact subgroup $J_{\mu}$ such that $I\subset J_{\mu}\subset K$ by
\[
J_{\mu}:=\langle M_{1},U_{\alpha,f_{\mu}(\alpha)} \mid \alpha\in\Phi_{\red} \rangle,
\]
where $f_{\mu}\colon\Phi_{\rm red}\rightarrow\R$ is a function given by
\[
f_{\mu}
:=\begin{cases}
0&\text{if $\langle\alpha,\mu\rangle\geq0$,}\\
0+&\text{if $\langle\alpha,\mu\rangle<0$}
\end{cases}
\]
($0+$ denotes any sufficiently small positive number).
This group $J_{\mu}$ is nothing but the parahoric subgroup $J_{\mcF}$ associated with the facet $\mcF$ such that 
\begin{itemize}
\item
$\mcF$ is contained in the closure $\ol{\mcC}$ of the fixed alcove $\mcC$,
\item
$\ol{\mcF}$ contains $\bfo$, and
\item
$\mcF$ contains $\bfo+\varepsilon\mu$ for any sufficiently small $\varepsilon>0$.
\end{itemize}
If we put $W_{\mcF}$ to be the subgroup of $\tilde{W}$ generated by the reflections with respect to the walls containing the facet $\mcF$ (note that  $W_{\mcF}$ is automatically contained in $W_{\bfo}$), 
then we have
\[
J_{\mu}\, (=J_{\mcF})=IW_{\mcF}I.
\]
This follows from that the Iwahori subgroup $I$ and the Iwahori--Weyl group $\tilde{W}$ form a Tits system and that a parahoric subgroup is a parabolic subgroup in the sense of a Tits system (see \cite[Proposition 5.2.12]{MR756316} and \cite[Section 1.5]{MR0327923}).
See also an expository of Yu \cite[Section 7.3]{MR3525846}.

Through the isomorphism $W_{\bfo}\cong W$ mentioned in Section \ref{subsec:Iwahori-Weyl}, the subgroup $W_{\mcF}$ of $W_{\bfo}$ is identified with the subgroup $W_{\mu}$ of $W$ given by
\[
W_{\mu}
:=\langle s_{\alpha} \mid \alpha\in\Phi, s_{\alpha}(\mu)=\mu\rangle
=\langle s_{\alpha} \mid \alpha\in\Phi, \langle\alpha,\mu\rangle=0\rangle,
\]
where $s_{\alpha}$ denotes the reflection with respect to a root $\alpha\in\Phi$.

\subsection{Some lemmas on Iwahori subgroups}\label{subsec:lemmas}

In terms of the valuation of root datum associated with $\bfo$, the Iwahori subgroup $I$ is explicitly described as follows:
\[
I=\langle M_{1}, U_{\alpha,0}, U_{\beta,0+} \mid \alpha\in\Phi_{\red}^{+},  \beta\in\Phi_{\red}^{-}\rangle.
\]
Furthermore, the Iwahori subgroup $I$ has the following uniqueness of the product expression (see \cite[Section 3.1.1]{MR546588}).
\begin{prop}\label{prop:uniqueness}
The natural multiplication map
\[
\prod_{\alpha\in\Phi_{\red}^{+}} U_{\alpha,0}
\times M_{1}
\times \prod_{\beta\in\Phi_{\red}^{-}} U_{\beta,0+}
\rightarrow I
\]
is bijective with any orders on $\Phi_{\red}^{+}$ and $\Phi_{\red}^{-}$ (also, the products over $\Phi_{\red}^{+}$ and $\Phi_{\red}^{-}$ can be swapped).
\end{prop}

For an $F$-rational standard parabolic subgroup $\mathbf{Q}$ of $\G$ with Levi decomposition $\mathbf{Q}=\mathbf{LU}$, we introduce the following notation ($\ol{\bfU}$ denotes the opposite to $\bfU$): 
\begin{itemize}
\item
We put $\Phi_{\red}^{+}(\bfU):=\{\alpha\in\Phi^{+}_{\red}\mid \bfU_{\alpha}\subset\bfU\}$ and define
\[
I_{U}:=\prod_{\alpha\in \Phi_{\red}^{+}(\bfU)} U_{\alpha,0} \subset G.
\]
\item
We put $\Phi_{\red}^{\pm}(\bfL):=\{\alpha\in\Phi^{\pm}_{\red}\mid \bfU_{\alpha}\subset\bfL\}$ and define
\[
I_{L}:=\prod_{\alpha\in \Phi_{\red}^{+}(\bfL)} U_{\alpha,0}\times M_{1}\times \prod_{\alpha\in \Phi_{\red}^{-}(\bfL)} U_{\alpha,0+} \subset G,
\]
\item
We put $\Phi_{\red}^{-}(\ol{\bfU}):=\{\alpha\in\Phi^{-}_{\red}\mid \bfU_{\alpha}\subset\ol{\bfU}\}$ and define
\[
I_{\ol{U}}:=\prod_{\alpha\in \Phi_{\red}^{-}(\ol{\bfU})} U_{\alpha,0+} \subset G.
\]
\end{itemize}
Note that the definitions of $I_{U}$, $I_{L}$, $I_{\ol{U}}$ are independent of the choice of orders on the sets of roots and that these sets are subgroups of $G$.

\begin{lem}\label{lem:Iwahori}
For any $F$-rational standard parabolic subgroup $\mathbf{Q}$ of $\G$ with Levi decomposition $\mathbf{Q}=\mathbf{LU}$, the following hold.
\begin{enumerate}
\item
We have $I=I_{U}I_{L}I_{\ol{U}}=I_{\ol{U}}I_{L}I_{U}$.
\item
For any $w\in W$, we have $wI_{\ol{U}}w^{-1}\subset I$.
\item
For any dominant $\mu\in \Lambda_{M}$, we have $\mu I_{U}\mu^{-1}\subset I_{U}$ and $\mu^{-1} I_{\ol{U}}\mu\subset I_{\ol{U}}$.
\end{enumerate}
\end{lem}

\begin{proof}
\begin{enumerate}
\item
This is clear from Proposition \ref{prop:uniqueness} and the definitions of $I_{U}$, $I_{L}$, and $I_{\ol{U}}$.

\item
Since we regard $w\in W$ as an element of $\tilde{W}$ through the isomorphism $W_{\bfo}\cong W$, the action of $w$ on the apartment $\mcA(\bfA,F)$ stabilizes the special point $\bfo$.
Hence $w$ stabilizes the valuation of root datum associated with $\bfo$.
In particular, we have $wU_{\alpha,r}w^{-1}=U_{w(\alpha),r}$  for any $\alpha\in\Phi$ and $r\in\R$.
Thus we get $wU_{\alpha,0+}w^{-1}\subset I$ for any $\alpha\in\Phi_{\red}^{-}$, which implies that we have $wI_{\ol{U}}w^{-1}\subset I$.

\item
Since $\{U_{\alpha,r}\}_{r\in\R}$ consists of a part of the valuation of root datum, we have $\ul{\mu} U_{\alpha,0}\ul{\mu}^{-1}=U_{\alpha,\langle\alpha,\nu_{M}(\ul{\mu})\rangle}$ (see\cite[Proposition 6.2.10]{MR0327923}).
The fact that $q_{M}\circ\kappa_{M}=\nu_{M}$ shows that 
\[
\langle\alpha,\nu_{M}(\ul{\mu})\rangle
=\langle\alpha,\nu_{M}(\kappa_{M}^{-1}(\mu))\rangle
=\langle\alpha,q_{M}(\mu)\rangle.
\]
Since the dominance on $\Lambda_{M}$ is introduced through the homomorphism $q_{M}$ (see Section \ref{subsec:Iwahori}), we have $\langle \alpha,q_{M}(\mu)\rangle\geq0$ for any $\alpha\in\Phi^{+}$.
Thus we have $\ul{\mu}U_{\alpha,0}\ul{\mu}^{-1}\subset U_{\alpha,0}$, hence get $\ul{\mu}I_{U}\ul{\mu}^{-1}\subset I_{U}$.

We can check that $\ul{\mu}^{-1} U_{\alpha,0+}\ul{\mu}\subset U_{\alpha,0+}$ for any $\alpha\in\Phi_{\red}^{-}$ (hence $\ul{\mu}^{-1}I_{\ol{U}}\ul{\mu}\subset I_{\ol{U}}$) in a similar way.
\end{enumerate}
\end{proof}

\begin{lem}\label{lem:HKP-pre}
\begin{enumerate}
\item
For any $w\in W$, we have $wIw^{-1}I\cap NI=I$.
\item
For any dominant element $\mu\in \Lambda_{M}$, we have $\ul{\mu} I\ul{\mu}^{-1}I\cap NI=I$.
\end{enumerate}
\end{lem}

\begin{proof}
Let us show (1).
Since the inclusion $wIw^{-1}I\cap NI\supset I$ is obvious, we only need to prove the converse inclusion $wIw^{-1}I\cap NI\subset I$.
To see this, it suffices to check that $wIw^{-1}\cap NI\subset I$.
As we have $I=I_{N}I_{M}I_{\ol{N}}$ by Lemma \ref{lem:Iwahori} (1), we have $NI=NI_{M}I_{\ol{N}}$.
Hence the multiplication map
\[
N\times M\times \ol{N}\rightarrow G,
\]
which is injective (\cite[Th\'eor\`eme 2.2.3]{MR756316}), induces a bijection $N\times I_{M}\times I_{\ol{N}}\xrightarrow{1:1} NI$.
If we define a function $f_{w}\colon\Phi_{\red}\rightarrow\R$ by
\[
f_{w}(\alpha)=
\begin{cases}
0&\text{if $w^{-1}(\alpha)\in\Phi_{\red}^{+}$,}\\
0+&\text{if $w^{-1}(\alpha)\in\Phi_{\red}^{-}$}
\end{cases}
\]
($0+$ denotes any sufficiently small positive number), then we have
\[
wIw^{-1}=\langle M_{1}, U_{\alpha,f_{w}(\alpha)} \mid \alpha\in \Phi_{\red}\rangle.
\]
We put $I'_{N}:=\prod_{\alpha\in\Phi_{\red}^{+}}U_{\alpha,f_{w}(\alpha)}$  
and $I'_{\ol{N}}:=\prod_{\alpha\in\Phi_{\red}^{-}}U_{\alpha,f_{w}(\alpha)}$. 
Note that $wI_{M}w^{-1} = I_{M}=M_1$. 
Then, similarly to Proposition \ref{prop:uniqueness}, we see that the multiplication map $N\times M\times \ol{N}\rightarrow G$ induces a bijection
\[
I'_{N}\times I_{M}\times I'_{\ol{N}}\rightarrow wIw^{-1}.
\]
Since $I'_{N}\subset I_{N}$ by Lemma \ref{lem:Iwahori} (2), we obtain  
\[
wIw^{-1} \cap NI \subset  I'_{N} I_{M} (I'_{\ol{N}} \cap I_{\ol{N}}) \subset I_{N}I_{M} I_{\ol{N}} = I. 
\]

The same argument works for (2) by using Lemma \ref{lem:Iwahori} (3) instead of Lemma \ref{lem:Iwahori} (2).
\end{proof}

\subsection{Iwahori--Hecke algebra}\label{subsec:Iwahori-Hecke}

Let $\mcH_{I}:=C_{c}^{\infty}(I\backslash G/I)$ be the Iwahori Hecke algebra, which has a structure of a $\C$-algebra via convolution product denoted by $\ast$.
Here we use the Haar measure $dg$ on $G$ normalized so that $dg(I)=1$ in the definition of the convolution product.
Recall that we have the Iwahori decomposition (see \cite[Lemma 4.57]{MR3444233}):
\[
G=\bigsqcup_{w\in\tilde{W}} IwI.
\]
Thus, if we put $T_{w}$ to be the characteristic function $\mathbbm{1}_{IwI}$ of the double coset $IwI$ for $w\in\tilde{W}$, then the set $\{T_{w}\}_{w\in\tilde{W}}$ forms a $\C$-basis of $\mcH_{I}$.

According to \cite[Definition 5.3.1]{MR3355113}, we normalize $T_{w}$ for $w\in\tilde{W}$ by
\[
\overline{T}_{w}:=\bfq(w)^{-\frac{1}{2}}T_{w}.
\]
Here, we define a function $\bfq\colon\tilde{W}\rightarrow\Z_{>0}$ by
\[
\bfq(w):=[IwI:I].
\]
This quantity can be expressed in a root-theoretic way as follows (see \cite[Section 1.4]{MR3481263} for the details).
We let $\tilde{W}^{\nr}$ denote the Iwahori--Weyl group over the completion $\breve{F}$ of the maximal unramified extension of $F$.
Then, by \cite[Proposition 1.11]{MR3481263}, $\tilde{W}$ is contained in $\tilde{W}^{\nr}$ and we have
\[
\bfq(w)=q^{\ell^{\nr}(w)}
\]
for any $w\in \tilde{W}$, where $\ell^{\nr}$ denotes the length function on $\tilde{W}^{\nr}$.

\begin{rem}\label{rem:rho}
For any dominant element $\lambda\in\Lambda_{M}$, we can compute $\ell^{\nr}(w)$ by using the result of Lusztig \cite{MR991016} on affine Weyl groups as follows.
Let $\bfS$ be a maximal $\breve{F}$-split torus of $\G$ which is defined over $F$ and contains $\bfA$.
Let $\Sigma$ be the scaled root system associated with $\Phi(\G,\bfS)$, i.e., the unique reduced root system in $X^{\ast}(\bfS)_{\R}$ such that hyperplanes determined by the affine functions $\Sigma+\Z$ on the apartment $\mcA(\bfS,\breve{F})$ coincide with those determined by the affine roots with respect to $\Phi(\G,\bfS)$ (see \cite[Section 2.3]{MR3355113} for details).
Then $\Lambda_{M}$ can be regarded as a subgroup of the affine Weyl group associated with the reduced root system $\Sigma$ (see \cite[Section 3.3]{MR3355113}).
By putting $\rho^{\nr}$ to be the half sum of all positive roots in $\Sigma$, we have $\frac{1}{2}\ell^{\nr}(\lambda)=\langle\rho^{\nr},\lambda\rangle$ for any dominant element $\lambda\in\Lambda_{M}$ by \cite[Section 1.4 (f)]{MR991016}.
Here, we consider the positivity on $\Sigma$ determined by the alcove of $\mcA(\bfS,\breve{F})$ whose Frobenius fixed part coincides with our fixed alcove of $\mcA(\bfA,F)$ (see \cite[Section 1.2]{MR3481263}).
\end{rem}

\subsection{Universal unramified principal series}\label{subsec:univ}

Recall that we fixed a minimal $F$-rational parabolic subgroup $\bfP$ of $\bfG$ with Levi factor $\bfM$.
We let $\bfN$ denote the unipotent radical of $\bfP$.
Hence we have a Levi decomposition $\bfP=\bfM\bfN$.

We put
\[
\mcM:=C_{c}^{\infty}(M_{1}N\backslash G/I).
\]
For $w\in\tilde{W}$, we put $v_{w}:=\mathbbm{1}_{M_{1}NwI}$.
Since we have
\[
G=\bigsqcup_{w\in\tilde{W}}IwI=\bigsqcup_{w\in\tilde{W}}M_{1}NwI,
\]
(see \cite[Lemma 4.61]{MR3444233}), the set $\{v_{w}\}_{w\in\tilde{W}}$ forms a $\C$-basis of  $\mcM$.

Let $\mcR$ be the group algebra $\C[\Lambda_{M}]$ of $\Lambda_{M}$, which is isomorphic to $C_{c}^{\infty}(M/M_{1})$.
For $\mu\in\Lambda_{M}$, we let $R_{\mu}$ denote the element of the group algebra $\C[\Lambda_{M}]$ corresponding to $\mu$.
Then $\{R_{\mu}\}_{\mu\in \Lambda_{M}}$ forms a $\C$-basis of $\mcR$.
We make $\mcM$ into a left $\mcR$-module by
\[
(r\cdot f)(g)
:= \int_{M} r(y)\delta_{P}^{\frac{1}{2}}(y)f(y^{-1}g)\,dy
\]
for any $r\in \mcR$ and $f\in\mcM$, where $\delta_{P}$ denotes the modulus character of $P$ and the Haar measure $dy$ on $M$ is normalized so that $dy(M_{1})=1$.

We will next make $\mcM$ into a right $\mcH_{I}$-module.
For this, we consider the set $C_{c}^{\infty}(M_{1}N\backslash G)$ of compactly supported left-$M_{1}N$-invariant smooth functions.
(We call this space the universal unramified principal series.)
Then we have a right action of the full Hecke algebra $\mcH:=C_{c}^{\infty}(G)$ on $C_{c}^{\infty}(M_{1}N\backslash G)$ given by $f\mapsto f\ast h$ for any $f\in C_{c}^{\infty}(M_{1}N\backslash G)$ and $h\in\mcH$.
This action naturally induces a right action of the Iwahori--Hecke algebra $\mcH_{I}$ on $\mcM=C_{c}^{\infty}(M_{1}N\backslash G)^{I}$.

In summary, with respect to these actions, $\mcM$ has a structure of an $(\mcR,\mcH_{I})$-bimodule.

\begin{rem}\label{rem:left-right}
Since $C_{c}^{\infty}(M_{1}N\backslash G)$ is a smooth representation of $G$ via right translation (let $\rho_{\migi}$ denote this representation), we may also consider the left action of $\mcH$ on $C_{c}^{\infty}(M_{1}N\backslash G)$ given by
\[
\rho_{\migi}(h)(f):=\int_{G} h(g)\cdot\rho_{\migi}(g)(f)\,dg
\]
for $h\in\mcH$ and $f\in C_{c}^{\infty}(M_{1}N\backslash G)$.
The relationship between the right action $(-)\ast h$ and the left action $\rho_{\migi}(h)(-)$ is described as follows.
Let $\iota\colon\mcH\rightarrow\mcH$ be the anti-involution given by $\iota(h)(g):=h(g^{-1})$ (i.e., $\iota$ is a $\C$-linear automorphism of $\mcH$ satisfying $\iota(h_{1}\ast h_{2})=\iota(h_{2})\ast \iota(h_{1})$ for any $h_{1},h_{2}\in\mcH$).
Then we have
\[
(-)\ast\iota(h)=\rho_{\migi}(h)(-).
\]
Indeed, for any $h\in\mcH$, $f\in C_{c}^{\infty}(M_{1}N\backslash G)$, and $x\in G$, we have
\begin{align*}
\bigl(f\ast\iota(h)\bigr)(x)
&=\int_{G}f(g)\cdot\iota(h)(g^{-1}x)\,dg\\
&=\int_{G}f(g)\cdot h(x^{-1}g)\,dg\\
&=\int_{G}f(xg)\cdot h(g)\,dg\\
&=\int_{G}h(g)\cdot \bigl(\rho_{\migi}(g)(f)\bigr)(x)\cdot \,dg
=\bigl(\rho_{\migi}(h)(f)\bigr)(x).
\end{align*}
\end{rem}

We put 
\[
X^{\w}(M):=\Hom(M/M_{1},\C^{\times})
\]
and call an element of $X^{\w}(M)$ an weakly unramified character of $M$.
Each element $\chi\in X^{\w}(M)$ defines a $\C$-algebra homomorphism
\[
\mcR=\C[\Lambda_{M}]\twoheadrightarrow\C\colon R_{\mu}\mapsto\chi(\ul{\mu}).
\]
If we again write $\chi$ for this homomorphism, then we have an isomorphism
\[
\C\otimes_{\mcR,\chi}\mcM
\cong(\nInd_{P}^{G}\chi^{-1})^{I}
\]
as right $\mcH_{I}$-modules (\cite[Lemma 4.63 (a)]{MR3444233}), where $\nInd_{P}^{G}$ denotes the normalized parabolic induction.

\subsection{Several basic identities on $\mcR$, $\mcH_{I}$, and $\mcM$}\label{subsec:Bernstein}

Recall that
\begin{itemize}
\item
we put $T_{w}:=\mathbbm{1}_{IwI}\in\mcH_{I}$ for $w\in\tilde{W}$ (hence $\{T_{w}\}_{w\in\tilde{W}}$ is a $\C$-basis of $\mcH_{I}$),
\item
we put $\{R_{\mu}\}_{\mu\in \Lambda_{M}}$ to be the natural $\C$-basis of the group algebra $\mcR=\C[\Lambda_{M}]$, and
\item
we put $v_{w}:=\mathbbm{1}_{M_{1}NwI}\in\mcM$ for $w\in\tilde{W}$ (hence $\{v_{w}\}_{w\in\tilde{W}}$ is a $\C$-basis of $\mcM$).
\end{itemize}

Let $\rho_{\bfP}\in X^{\ast}(\bfA)_{\R}$ be the element satisfying $\delta^{\frac{1}{2}}_{P}(\ul{\mu})=q^{-\langle\rho_{\bfP},\mu\rangle}$.
Note that this is explicitly given by
\[
\rho_{\bfP}
=\frac{1}{2}\sum_{\alpha\in\Phi_{\red}^{+}} \bigl(\dim_{F}(\mfg_{\alpha})\cdot\alpha+\dim_{F}(\mfg_{2\alpha})\cdot2\alpha\bigr),
\]
where $\mfg_{\alpha}$ and $\mfg_{2\alpha}$ denote the root subspaces of $\mfg$ associated with the roots $\alpha$ and $2\alpha$, respectively (we simply put $\mfg_{2\alpha}:=0$ when $2\alpha$ is not a root).

\begin{lem}\label{lem:Rv}
For any $\mu\in\Lambda_{M}$, we have $R_{\mu}\cdot v_{1}=q^{-\langle\rho_{\bfP},\mu\rangle}\cdot v_{\mu}$.
\end{lem}

\begin{proof}
By the definition of the left $\mcR$-module structure of $\mcM$, we have 
\begin{align*}
(R_{\mu}\cdot v_{1})(g)
&=\int_{M} R_{\mu}(y)\delta_{P}^{\frac{1}{2}}(y)\mathbbm{1}_{M_{1}NI}(y^{-1}g)\,dy\\
&=\int_{\ul{\mu}M_{1}} \delta_{P}^{\frac{1}{2}}(y)\mathbbm{1}_{M_{1}NI}(y^{-1}g)\,dy\\
&=\delta_{P}^{\frac{1}{2}}(\ul{\mu}) \int_{M_{1}} \mathbbm{1}_{M_{1}NI}(y^{-1}\ul{\mu}^{-1}g)\,dy
\end{align*}
for any $g\in G$.
This is not zero if only if $y^{-1}\ul{\mu}^{-1}g$ belongs to $M_{1}NI$ for some $y\in M_{1}$, which is equivalent to that $g$ belongs to $\ul{\mu}M_{1}NI=M_{1}N\ul{\mu}I$.
In other words, $R_{\mu}\cdot v_{1}$ is supported on $M_{1}N\ul{\mu}I$.
When $g$ belongs to $M_{1}N\ul{\mu}I$, we have
\[
R_{\mu}\cdot v_{1}(g)
=
\delta_{P}^{\frac{1}{2}}(\ul{\mu}) \int_{M_{1}} \mathbbm{1}_{M_{1}NI}(y^{-1}\ul{\mu}^{-1}g)\,dy
=
\delta_{P}^{\frac{1}{2}}(\ul{\mu}) dy(M_{1})= q^{-\langle\rho_{\bfP},\mu\rangle}.
\]
Thus we have $R_{\mu}\cdot v_{1}=q^{-\langle\rho_{\bfP},\mu\rangle}\cdot v_{\mu}$.
\end{proof}

The following proposition in the split case can be found in \cite{MR2642451} ((1): \cite[(1.6.1)]{MR2642451}; (2): \cite[(1.6.3)]{MR2642451}).
\begin{prop}\label{prop:HKP}
\begin{enumerate}
\item
For any $w\in W\subset\tilde{W}$, we have $v_{1}\ast T_{w}=v_{w}$.
\item
For any dominant element $\mu\in \Lambda_{M}$, we have $v_{1}\ast T_{\mu}=v_{\mu}$.
\end{enumerate}
\end{prop}

\begin{proof}
\begin{enumerate}
\item
By the definitions of $v_{1}$ and $T_{w}$, we have
\begin{align*}
(v_{1}\ast T_{w})(x)
&=\int_{G}\mathbbm{1}_{M_{1}NI}(g)\cdot \mathbbm{1}_{IwI}(g^{-1}x)\,dg\\
&=\int_{M_{1}NI}\mathbbm{1}_{IwI}(g^{-1}x)\,dg.
\end{align*}
Let $g\in M_{1}NI$.
If the integrand $\mathbbm{1}_{IwI}(g^{-1}x)$ is not zero, then $x$ must belong to $gIwI$.
By Lemma \ref{lem:Iwahori} (1), we have $M_{1}NI=M_{1}NI_{N}I_{M}I_{\ol{N}}=M_{1}NI_{M}I_{\ol{N}}$.
Since $M$ normalizes $N$ and $I_{M}\subset M_{1}$, we have $M_{1}NI_{M}I_{\ol{N}}=M_{1}NI_{\ol{N}}$.
Hence $gIwI$ is contained in $M_{1}NI_{\ol{N}}wI$, which is equal to $M_{1}NwI$ by Lemma \ref{lem:Iwahori} (2).
Thus the function $v_{1}\ast T_{w}$ is supported on $M_{1}NwI$.

Let $x$ be an element of $M_{1}NwI$.
Let us write $x=mnwy$ with $m\in M_{1}$, $n\in N$, $y\in I$.
Then $g^{-1}x$ belongs to $IwI$ if and only if $g$ belongs to $mnwyIw^{-1}I=mnwIw^{-1}I$.
Hence we get
\begin{align*}
(v_{1}\ast T_{w})(x)
&=dg(mnwIw^{-1}I\cap M_{1}NI)\\
&=dg(wIw^{-1}I\cap M_{1}NI).
\end{align*}
By Lemma \ref{lem:HKP-pre} (1), we have $dg(wIw^{-1}I\cap M_{1}NI)=dg(I)=1$.
Thus we conclude that $v_{1}\ast T_{w}$ is equal to $\mathbbm{1}_{M_{1}NwI}$, which equals $v_{w}$ by definition.

\item
The proof is similar to that of claim (1) (the same argument works by using Lemmas \ref{lem:Iwahori} (3) and  \ref{lem:HKP-pre} (2) instead of Lemmas \ref{lem:Iwahori} (2) and  \ref{lem:HKP-pre} (1), respectively).
\end{enumerate}
\end{proof}

By \cite[Lemma 1.6.1]{MR2642451} (split case) and \cite[Lemma 4.63 (b)]{MR3444233} (non-split case), $\mcM$ is free of rank $1$ with generator $v_{1}$ as an $\mcH_{I}$-module.
In particular, we have an isomorphism of $\C$-algebras
\[
\mcH_{I}\cong\End_{\mcH_{I}}(\mcM)\colon h' \mapsto [v_{1}\ast h\mapsto v_{1}\ast h'\ast h].
\] 
Accordingly, the left $\mcR$-action on $\mcM$ induces an injective $\C$-algebra homomorphism
\[
\mcR\hookrightarrow\End_{\mcH_{I}}(\mcM)\cong\mcH_{I}.
\]

\begin{defn}[{\cite[Definition 5.3.1]{MR3355113}}]\label{defn:Theta}
For any element $\mu\in\Lambda_{M}$, we put
\[
\Theta_{\mu}:=\overline{T}_{\lambda_{1}}\ast\overline{T}_{\lambda_{2}}^{-1}
\]
by taking dominant elements $\lambda_{1},\lambda_{2}\in\Lambda_{M}$ satisfying $\mu=\lambda_{1}-\lambda_{2}$.
(See \cite[Definition 5.3.1]{MR3355113} for the well-definedness of this definition.)
\end{defn}

\begin{rem}\label{rem:Theta}
Note that, for $\mu\in\Lambda_{M}$, the quantity $\bfq(\mu)$ and the element $T_{\mu}$ are defined by regarding $\mu$ as an element of the Iwahori--Weyl group $\tilde{W}$ through the Kottwitz homomorphism $\kappa_{M} \colon M\twoheadrightarrow \Lambda_{M}$.
As we mentioned in Remark \ref{rem:q}, in \cite{MR3355113}, the symbol $\kappa_{M}$ denotes the $(-1)$-multiple of the usual Kottwitz homomorphism $\kappa_{M}$.
Accordingly, our $\Theta_{\mu}$ is equal to Rostami's $\Theta_{-\mu}$.
\end{rem}

\begin{prop}\label{prop:Theta}
The image of $R_{\mu}$ under the above homomorphism $\mcR\hookrightarrow\mcH_{I}$ is given by $q^{\langle\rho^{\nr}-\rho_{\bfP},\mu\rangle}\cdot \Theta_{\mu}$.
In other words, we have
\[
q^{\langle\rho^{\nr}-\rho_{\bfP},\mu\rangle}\cdot v_{1}\ast\Theta_{\mu}
=
R_{\mu}\cdot v_{1}.
\]
\end{prop}

\begin{proof}
Let $\mu$ be an element of $\Lambda_{M}$.
Note that, for any dominant element $\lambda\in\Lambda_{M}$, we have
\[
R_{\lambda}\cdot v_{1} = q^{-\langle\rho_{\bfP},\lambda\rangle}\cdot v_{\lambda}= q^{-\langle\rho_{\bfP},\lambda\rangle}\cdot v_{1} \ast T_{\lambda},
\]
or, equivalently,
\[
R_{\lambda}^{-1}\cdot v_{1}= q^{\langle\rho_{\bfP},\lambda\rangle}\cdot v_{1} \ast T_{\lambda}^{-1}
\]
by Lemma \ref{lem:Rv} and Proposition \ref{prop:HKP} (2).
Hence, by taking dominant elements $\lambda_{1}$ and $\lambda_{2}$ of $\Lambda_{M}$ such that $\mu=\lambda_{1}-\lambda_{2}$ and applying this identity to $\lambda_{1}$ and $\lambda_{2}$, we get
\begin{align*}
R_{\mu}\cdot v_{1}
&=R_{\lambda_{1}}\cdot R_{\lambda_{2}}^{-1}\cdot v_{1}\\
&= q^{\langle\rho_{\bfP},\lambda_{2}\rangle}\cdot R_{\lambda_{1}}\cdot v_{1} \ast T_{\lambda_{2}}^{-1}\\
&= q^{\langle\rho_{\bfP},\lambda_{2}-\lambda_{1}\rangle}\cdot v_{1} \ast T_{\lambda_{1}}\ast T_{\lambda_{2}}^{-1}
= q^{-\langle\rho_{\bfP},\mu\rangle}\cdot v_{1}\ast T_{\lambda_{1}}\ast T_{\lambda_{2}}^{-1}.
\end{align*}
Since $\Theta_{\mu}$ is defined by
\[
\Theta_{\mu}
=\overline{T}_{\lambda_{1}}\ast\overline{T}_{\lambda_{2}}^{-1}
=\bfq(\lambda_{1})^{-\frac{1}{2}}\cdot \bfq(\lambda_{2})^{\frac{1}{2}}\cdot T_{\lambda_{1}}\ast T_{\lambda_{2}}^{-1},
\]
we get
\[
R_{\mu}\cdot v_{1}
= q^{-\langle\rho_{\bfP},\mu\rangle}\cdot \bfq(\lambda_{1})^{\frac{1}{2}}\cdot \bfq(\lambda_{2})^{-\frac{1}{2}} \cdot v_{1}\ast\Theta_{\mu}.
\]
Since we have $\bfq(\lambda_{i})^{\frac{1}{2}}=q^{\frac{1}{2}\ell^{\nr}(\lambda_{i})}$ (see Section \ref{subsec:Iwahori-Weyl}) and $\frac{1}{2}\ell^{\nr}(\lambda_{i})=\langle\rho^{\nr},\lambda_{i}\rangle$ (see Remark \ref{rem:rho}), we get
\[
q^{-\langle\rho_{\bfP},\mu\rangle}\cdot \bfq(\lambda_{1})^{\frac{1}{2}}\cdot \bfq(\lambda_{2})^{-\frac{1}{2}}
=q^{\langle\rho^{\nr}-\rho_{\bfP},\mu\rangle}.
\]
\end{proof}

\begin{cor}\label{cor:Theta-H}
For any dominant element $\mu\in \Lambda_{M}$, we have
\[
\Theta_{\mu}
=
q^{-\langle\rho^{\nr},\mu\rangle}\cdot T_{\mu}.
\]
\end{cor}

\begin{proof}
Since $\mcM$ is a free $\mcH_{I}$-module of rank $1$ with generator $v_{1}$, it suffices to check that 
\[
v_{1}\ast\Theta_{\mu}=q^{-\langle\rho^{\nr},\mu\rangle}\cdot v_{1}\ast T_{\mu}.
\]
By Proposition \ref{prop:Theta}, we have $v_{1}\ast\Theta_{\mu}=q^{\langle-\rho^{\nr}+\rho_{\bfP},\mu\rangle}\cdot R_{\mu}\cdot v_{1}$.
Since we have $R_{\mu}\cdot v_{1}=q^{-\langle\rho_{\bfP},\mu\rangle}\cdot v_{\mu}=q^{-\langle\rho_{\bfP},\mu\rangle}\cdot v_{1}\ast T_{\mu}$ by Lemma \ref{lem:Rv} and Proposition \ref{prop:HKP} (2), we get $q^{\langle-\rho^{\nr}+\rho_{\bfP},\mu\rangle}\cdot R_{\mu}\cdot v_{1}=q^{-\langle\rho^{\nr},\mu\rangle}\cdot v_{1}\ast T_{\mu}$.
\end{proof}

\begin{rem}\label{rem:q-ur}
Assume that $\G$ is split over $\breve{F}$. 
In this case, the set of affine roots for the apartment $\mcA(\bfS,\breve{F})$ is given by $\Phi(\G,\bfS)+\Z$ under the identification $\mcA(\bfS,\breve{F})\cong X_{\ast}(\bfS)_{\R}$ given by the Chevalley special point (Remark \ref{rem:Chevalley}) since $\val_F\circ x_{\alpha}^{-1}(U_{\alpha})=\Z$ for any $\alpha\in\Phi(\G,\bfS)$. 
Hence the scaled root system $\Sigma$ (see Remark \ref{rem:rho}) equals $\Phi(\G,\bfS)$ as $\Phi(\G,\bfS)$ is reduced. 
By the definition of the positive system of $\Sigma=\Phi(\G,\bfS)$, any positive root in $\Phi(\G,\bfS)$ restricts to a positive root in $\Phi(\G,\bfA)$ or zero. 
Therefore $\rho^{\nr}$ maps to $\rho_{\bfP}$ under the restriction from $X^{\ast}(\bfS)_{\R}$ to $X^{\ast}(\bfA)_{\R}$, and $q^{\langle\rho^{\nr}-\rho_{\bfP},\mu\rangle}=1$ for any $\mu\in\Lambda_{M}$.
In particular, the definition of $\Theta_{\mu}$ given in this paper coincides with that by \cite[Section 1.7]{MR2642451} when $\G$ is split.
\end{rem}

Finally, we introduce the Bernstein relation, which will play an important role in the induction step of the proof of Proposition \ref{prop:Haines}.
\begin{prop}[Bernstein relation, {\cite[Proposition 5.4.2]{MR3355113}}]\label{prop:Bernstein}
Let $\alpha\in\Delta$ be a simple root with simple reflection $s_{\alpha}\in W$.
Then, for any $\mu\in\Lambda_{M}$, there exist a family of complex numbers $\{\bfq_{\bar{j}}(s_{\alpha})\}_{j=0,\ldots,N-1}$ satisfying 
\[
T_{s_{\alpha}}\ast\Theta_{\mu}
=\Theta_{s_{\alpha}(\mu)}\ast T_{s_{\alpha}}+\sum_{j=0}^{N-1}\bfq_{\bar{j}}(s_{\alpha})\Theta_{\mu-j\alpha^{\vee}},
\]
where $\alpha^{\vee}$ denotes the coroot corresponding to $\alpha$.
\end{prop}

See \cite[Section 5.4]{MR3355113} for the details of the notations used in the statement of Proposition \ref{prop:Bernstein}.

\section{Hecke action on the unramified principal series}\label{sec:Hecke}

\subsection{Triangularity of the action of $\Theta_{\mu}$}\label{subsec:triangle}

The space $\mcM$ is free as $\mcR$-module with a basis $\{v_{w}\}_{w\in W}$ (see \cite[Lemma 4.63 (c)]{MR3444233}).
Our aim in this section is to compute the action of $\Theta_{\mu}$ on $\mcM$ in terms of the basis $\{v_{w}\}_{w\in W}$.
For this, we recall basics on the Bruhat order on $W$.

For $\alpha\in\Phi$, we write $s_{\alpha}$ for the reflection with respect to $\alpha$.
For each $w\in W$, we put $\ell(w)$ to be the length of $w$, which is defined by
\[
\ell(w):= \# \{\alpha\in\Phi^{+}_{\red} \mid w(\alpha)\in\Phi^{-}\}.
\]
For $w,w'\in W$, write $w'\rightarrow w$ if $\ell(w')<\ell(w)$ and $w=w's_{\alpha}$ for some $\alpha\in\Phi$.
Then we define $w'\leq w$ if there is a sequence $w'=w_0\rightarrow w_1\rightarrow \cdots \rightarrow w_m=w$ for some nonnegative integer $m$ and $w_0,\ldots, w_m\in W$. 
The relation is a partial order on $W$ and is called the Bruhat order. 
It is immediate that we have $\ell(w')< \ell(w)$ if $w'< w$. 
\begin{lem}[{\cite[Lemma 1.6]{MR1066460}}]\label{lem:order}
For every $w\in W$ and $\alpha\in\Delta$, we have
\[
\begin{cases}
w < ws_{\alpha}&\text{if $w(\alpha)\in\Phi^{+}$,}\\
w > ws_{\alpha}&\text{if $w(\alpha)\in\Phi^{-}$.}
\end{cases}
\]
\end{lem}
The bijection $w\mapsto w^{-1}$ of $W$ is an automorphism as an ordered set (see \cite[Corollary 2.2.5]{MR2133266}).  
From this fact and \cite[Proposition 2.2.7]{MR2133266}, we obtain the following 
\begin{lem}\label{lem:lifting}
Let $w, w''\in W$ and $\alpha\in\Delta$. 
If $w'=ws_{\alpha}<w, w''<w$ and $w''<w''s_{\alpha}$, then we have $w''s_{\alpha}<w$ and $w''<w'$.
\end{lem}

The following is the key to prove our main theorem of this paper.

\begin{prop}\label{prop:Haines}
For any $w\in W$ and $\mu\in \Lambda_{M}$, there exists a family $\{a_{w'}\}_{w'\in W, w'<w}$ of elements of $\mcR$ satisfying
\begin{align}\label{eq:Haines}
v_{w}\ast\Theta_{\mu}
=
q^{\langle\rho_{\bfP}-\rho^{\nr},w(\mu)\rangle}\cdot R_{w(\mu)}\cdot v_{w}+\sum_{\begin{subarray}{c}w'\in W \\ w'<w \end{subarray}} a_{w'}\cdot v_{w'}.
\end{align}
\end{prop}

\begin{proof}
We prove the assertion by the induction on the length $\ell(w)$ of $w\in W$. 
When $\ell(w)=0$, i.e., $w=1$, the equality $\eqref{eq:Haines}$ is nothing but Proposition \ref{prop:Theta}.

We consider the case where $\ell(w)=1$, i.e., $w=s_{\alpha}$ for some simple root $\alpha\in\Phi$.
Since any element $w'\in W$ satisfying $w'<w$ is necessarily equal to $1$, our task in this case is to find an element $a_{1}$ of $\mcR$ satisfying
\[
v_{s_{\alpha}}\ast\Theta_{\mu}
=
q^{\langle\rho_{\bfP}-\rho^{\nr},s_{\alpha}(\mu)\rangle}\cdot R_{s_{\alpha}(\mu)}\cdot v_{s_{\alpha}}+a_{1}\cdot v_{1}.
\]
By using Propositions \ref{prop:HKP} (1),  \ref{prop:Bernstein}, and  \ref{prop:Theta} in this order, we get
\begin{align*}
v_{s_{\alpha}}\ast\Theta_{\mu}
&\overset{\ref{prop:HKP}}{=}v_{1}\ast T_{s_{\alpha}}\ast\Theta_{\mu}\\
&\overset{\ref{prop:Bernstein}}{=} v_{1}\ast\biggl(\Theta_{s_{\alpha}(\mu)}\ast T_{s_{\alpha}} +\sum_{j=0}^{N-1}\bfq_{\bar{j}}(s_{\alpha})\Theta_{\mu-j\alpha^{\vee}}\biggr)\\
&\overset{\ref{prop:Theta}}{=} q^{\langle\rho_{\bfP}-\rho^{\nr},s_{\alpha}(\mu)\rangle}\cdot R_{s_{\alpha}(\mu)}\cdot v_{1}\ast T_{s_{\alpha}}+\sum_{j=0}^{N-1}a_{s_{\alpha},j}\cdot R_{\mu-j\alpha^{\vee}}\cdot v_{1},
\end{align*}
where $a_{s_{\alpha},j}$ is given by $q^{\langle\rho_{\bfP}-\rho^{\nr},\mu-j\alpha^{\vee}\rangle}\cdot \bfq_{\bar{j}}(s_{\alpha})$.
Thus it suffices to put
\[
a_{1}:=\sum_{j=0}^{N-1}a_{s_{\alpha},j}\cdot R_{\mu-j\alpha^{\vee}}\in \mcR.
\]

Next, we  consider the case where $\ell(w)>1$.
In this case, there exists a simple root $\alpha\in\Delta$ such that $w(\alpha)\in\Phi^-$. By Lemma \ref{lem:order}, $w':=ws_{\alpha}$ satisfies $w'<w$.
Then we have $T_{w}=T_{w'}\ast T_{s_{\alpha}}$ by the Iwahori--Matsumoto relation (see \cite[Section 7.2]{MR2642451} (split case) and \cite[Proposition 4.1.1 (ii)]{MR3355113} (non-split case)).
Thus, by using Proposition \ref{prop:HKP} (1) and this relation, we have
\[
v_{w}\ast\Theta_{\mu}
=v_{1}\ast T_{w}\ast\Theta_{\mu}
=v_{1}\ast T_{w'}\ast T_{s_{\alpha}}\ast\Theta_{\mu}.
\]
By using Propositions \ref{prop:Bernstein} and  \ref{prop:HKP} (1) in this order, we get
\begin{align*}
v_{1}\ast T_{w'}\ast T_{s_{\alpha}}\ast\Theta_{\mu}
&\overset{\ref{prop:Bernstein}}{=}v_{1}\ast T_{w'}\ast\biggl(\Theta_{s_{\alpha}(\mu)}\ast T_{s_{\alpha}} +\sum_{j=0}^{N-1}\bfq_{\bar{j}}(s_{\alpha})\Theta_{\mu-j\alpha^{\vee}}\biggr)\\
&\overset{\ref{prop:HKP}}{=}v_{w'}\ast\biggl(\Theta_{s_{\alpha}(\mu)}\ast T_{s_{\alpha}}+\sum_{j=0}^{N-1}\bfq_{\bar{j}}(s_{\alpha})\Theta_{\mu-j\alpha^{\vee}}\biggr).
\end{align*}
By the induction hypothesis, the second term $\sum_{j=0}^{N-1}\bfq_{\bar{j}}(s_{\alpha})v_{w'}\ast\Theta_{\mu-j\alpha^{\vee}}$ can be written as the $\mcR$-linear sum of $v_{w''}$'s for $w''\in W$ satisfying $w''\leq w'$ (hence, in particular, $w''<w$). 
Let us consider the first term $v_{w'}\ast\Theta_{s_{\alpha}(\mu)}\ast T_{s_{\alpha}}$.
By the induction hypothesis, there exists a family $\{a_{w''}\}_{w''\in W, w''<w'}$ of elements of $\mcR$ such that
\[
v_{w'}\ast\Theta_{s_{\alpha}(\mu)}
=
q^{\langle\rho_{\bfP}-\rho^{\nr},w'(s_{\alpha}(\mu))\rangle}\cdot R_{w'(s_{\alpha}(\mu))}\cdot v_{w'}+\sum_{\begin{subarray}{c}w''\in W \\ w''<w' \end{subarray}} a_{w''}\cdot v_{w''}.
\]
Hence, by noting that $w'(s_{\alpha}(\mu))=w(\mu)$, we get
\[
v_{w'}\ast\Theta_{s_{\alpha}(\mu)}\ast T_{s_{\alpha}}
=
q^{\langle\rho_{\bfP}-\rho^{\nr},w(\mu)\rangle}\cdot R_{w(\mu)}\cdot v_{w'}\ast T_{s_{\alpha}}+\sum_{\begin{subarray}{c}w''\in W \\ w''<w' \end{subarray}} a_{w''}\cdot v_{w''}\ast T_{s_{\alpha}}.
\]
The first term of the right hand side equals $q^{\langle\rho_{\bfP}-\rho^{\nr},w(\mu)\rangle}\cdot R_{w(\mu)}\cdot v_{w}$ since 
\[v_{w'}\ast T_{s_{\alpha}}=v_{1}\ast T_{w'}\ast T_{s_{\alpha}}=v_{1}\ast T_{w}=v_{w}\] 
by Proposition \ref{prop:HKP} (1) and the Iwahori--Matsumoto relation. 
Therefore it suffices to show that for any $w''\in W$ with $w''<w'$, the element $v_{w''}\ast T_{s_{\alpha}}=v_{1}\ast T_{w''}\ast T_{s_{\alpha}}$ is expressed as a $\C$-linear combination of elements in $\{v_{w'''}\mid w'''<w\}$. 

If $w''<w''s_{\alpha}$, we see 
$v_{1}\ast T_{w''}\ast T_{s_{\alpha}}=v_{1}\ast T_{w''s_{\alpha}}=v_{w''s_{\alpha}}$
by Proposition \ref{prop:HKP} (1) and the Iwahori--Matsumoto relation. 
Since $w'=ws_{\alpha}<w$, $w''<w$ and $w''<w''s_{\alpha}$, 
we have $w''s_{\alpha}<w$ by Lemma \ref{lem:lifting}. 
Hence the assertion holds when $w''<w''s_{\alpha}$. 

If $w''>w''s_{\alpha}$, Proposition \ref{prop:HKP} (1) together with the Iwahori--Matsumoto relations $T_{w''}=T_{w''s_{\alpha}}\ast T_{s_{\alpha}}$ and $T_{s_{\alpha}}\ast T_{s_{\alpha}}=(\bfq(s_{\alpha})-1)T_{s_{\alpha}}+\bfq(s_{\alpha})T_1$ (see \cite[Proposition 4.1.1 (ii), (iii)]{MR3355113}) shows that 
\begin{align*}
v_{1}\ast T_{w''}\ast T_{s_{\alpha}}
&=v_{1}\ast T_{w''s_{\alpha}}\ast T_{s_{\alpha}}\ast T_{s_{\alpha}}\\
&=v_{1}\ast T_{w''s_{\alpha}}\ast ((\bfq(s_{\alpha})-1)T_{s_{\alpha}}+\bfq(s_{\alpha})T_1)\\
&=v_{1}\ast ((\bfq(s_{\alpha})-1)T_{w''}+\bfq(s_{\alpha})T_{w''s_{\alpha}})\\
&=(\bfq(s_{\alpha})-1)v_{w''}+\bfq(s_{\alpha})v_{w''s_{\alpha}}
\end{align*}
Since $w''s_{\alpha}<w''<w$, the assertion also holds.
\end{proof}

\subsection{The case of Iwahori}\label{subsec:Hecke-Iwahori}

Let $V_{\chi}:=\nInd_{P}^{G}\chi$ be the principal series with respect to an weakly unramified character $\chi\colon M/M_{1}\rightarrow\C^{\times}$.
Recall that the space $\C\otimes_{\mcR,\chi^{-1}}\mcM$ equipped with the right $\mcH_{I}$-action  is nothing but $V_{\chi}^{I}$ as noted in Section \ref{subsec:univ}.
Hence the image of $\{v_{w}\}_{w\in W}$ in $\C\otimes_{\mcR,\chi}\mcM$ (for which we again write $\{v_{w}\}_{w\in W}$) forms a $\C$-basis of $V_{\chi^{-1}}^{I}$ for any $\chi$.

\begin{prop}\label{prop:triangle-Iwahori}
Let $\mu\in\Lambda_{M}$ be a strictly dominant element, i.e., a dominant element satisfying $\langle\alpha,\mu\rangle>0$ for any positive root $\alpha\in\Phi$.
Then there exists a $\C$-basis $\{v^{\vee}_{w}\}_{w\in W}$ of $V_{\chi}^{I}$ such that, for any $w\in W$, there exists a family $\{c_{w'}\}_{w'\in W, w'>w}$ of complex numbers satisfying
\[
I_{\chi}(\mathbbm{1}_{\mu})\cdot v^{\vee}_{w}
=
q(w,\mu) \cdot \Bigl(
\chi\circ\kappa_{M}^{-1}(w(\mu))\cdot v^{\vee}_{w}+\sum_{\begin{subarray}{c}w'\in W \\ w'>w \end{subarray}} c_{w'}\cdot v^{\vee}_{w'}
\Bigr),
\]
where $\mathbbm{1}_{\mu}:=\mathbbm{1}_{I\ul{\mu} I}$ and $q(w,\mu):=q^{\langle\rho^{\nr},\mu\rangle+\langle\rho_{\bfP}-\rho^{\nr},w(\mu)\rangle}$.
\end{prop}

\begin{proof}
Note that we have $\mathbbm{1}_{\mu}=T_{\mu}$.
Thus, by Remark \ref{rem:left-right}, the left action $I_{\chi}(\mathbbm{1}_{\mu})$ on $V_{\chi}^{I}$ coincides with the right action of 
$\iota(T_{\mu}) \in \mcH_{I}$ on $\C\otimes_{\mcR,\chi^{-1}}\mcM$. 
By Corollary \ref{cor:Theta-H}, we have
\[
\iota(T_{\mu}) = q^{\langle\rho^{\nr},\mu\rangle} \cdot \iota(\Theta_{\mu}). 
\]

Let us consider the left action of $\iota(\Theta_{\mu})$ on $\mcM$.
Note that we have an $\mcR$-valued perfect pairing 
\[
(-,-) \colon \mcM \times \mcM \longrightarrow \mcR
\] 
satisfying the following conditions (see \cite[Section 1.9]{MR2642451}):
\begin{itemize}
\item[(A)]$(r_{1}\cdot m_{1}, r_{2} \cdot m_{2}) = \iota_{\mcR}(r_{1})\cdot r_{2}\cdot (m_{1}, m_{2})$ for any $r_{1}, r_{2} \in \mcR$ and $m_{1}, m_{2} \in \mcM$, 
\item[(B)] $(m_{1}\ast h, m_{2}) = (m_{1}, m_{2}\ast\iota(h))$ for any $h \in \mcH_{I}$ and $m_{1}, m_{2} \in \mcM$.
\end{itemize}
Here $\iota_{\mcR}$ denotes the anti-involution of $\mcR$ defined by $\iota_{\mcR}(r)(x):=r(x^{-1})$ for any $r\in\mcR=\C[\Lambda_{M}]\cong C_{c}^{\infty}(M/M_{1})$.
Then the pairing $(-,-)$ induces a perfect pairing 
\[
(-,-)_{\chi} \colon V_{\chi}^{I} \times V_{\chi^{-1}}^{I} \longrightarrow \C. 
\]

Let $\{v^{\vee}_{w}\}_{w\in W}$ be the dual basis of $V_{\chi}^{I}\cong\C\otimes_{\mcR,\chi^{-1}}\mcM$ to $\{v_{w}\}_{w\in W}$ with respect to this perfect pairing, that is, each $v^{\vee}_{w}$ satisfies $(v^{\vee}_{w},v_{w'})_{\chi}=\delta_{w,w'}$.
Then, by Proposition \ref{prop:Haines}, we have
\[
(v^{\vee}_{w}\ast\iota(\Theta_{\mu}),v_{w'})_{\chi}
=(v^{\vee}_{w},v_{w'}\ast \Theta_{\mu})_{\chi}
=\begin{cases}
q^{\langle\rho_{\bfP}-\rho^{\nr},w(\mu)\rangle}\cdot\chi(R_{w(\mu)})&\text{if $w=w'$,}\\
c_{w'}&\text{if $w<w'$,}\\
0&\text{otherwise}
\end{cases}
\]
with some complex number $c_{w'}$.
By noting that $\chi(R_{w(\mu)})=\chi\circ\kappa_{M}^{-1}(w(\mu))$, we get the assertion.
\end{proof}

\begin{rem}\label{rem:q-spl}
By Remark \ref{rem:q-ur}, we simply have $q(w,\mu)=q^{\langle\rho_{\bfP},\mu\rangle}$ when $\G$ is split over $\breve{F}$.
\end{rem}

\begin{cor}\label{cor:OST4.3-Iwahori}
With the notations as in Proposition \ref{prop:triangle-Iwahori}, we have
\[
\det\bigl(1-q^{-s}\cdot I_{\chi}(\mathbbm{1}_{\mu})\,\big\vert\, V_{\chi}^{I}\bigr)
=
\prod_{w\in W} \bigl(1-q^{-s}\cdot q(w,\mu)\cdot\chi\circ\kappa_{M}^{-1}(w(\mu))\bigr).
\]
\end{cor}

\begin{proof}
For each $k\in\Z_{\geq0}$, we put 
\[
W(k):=\{w\in W \mid \ell(w)=k\}.
\]
Then obviously we have $W=W(0)\sqcup\cdots\sqcup W(h)$ for $h:=\max\{\ell(w)\mid w\in W\}$.
We choose a labeling $W=\{w_{1},\ldots,w_{\#W}\}$ so that we have
\begin{align*}
W(0)&=\{w_{1},\ldots,w_{\#W(0)}\},\\
W(1)&=\{w_{\#W(0)+1},\ldots,w_{\#W(0)+\#W(1)}\},\\
&\vdots\\
W(h)&=\{w_{\#W(0)+\cdots\#W(h-1)+1},\ldots,w_{\#W}\}.
\end{align*}

We take a $\C$-basis $\{v^{\vee}_{w}\}_{w\in W}$ of $V_{\chi}^{I}$ as in Proposition \ref{prop:triangle-Iwahori} and consider a matrix representation of $I_{\chi}(\mathbbm{1}_{\mu})$ on $V_{\chi}^{I}$ with respect to the basis $\{v^{\vee}_{w_{i}}\}_{w\in W}$ ordered according to the above labeling on $W$. 
Then, by Proposition \ref{prop:triangle-Iwahori}, we have
\[
I_{\chi}(\mathbbm{1}_{\mu})\cdot v^{\vee}_{w_{i}}
=
q(w_{i},\mu) \cdot \Bigl(
\chi\circ\kappa_{M}^{-1}(w_{i}(\mu))\cdot v^{\vee}_{w_{i}}+\sum_{\begin{subarray}{c}i'\in\{1,\ldots,\#W\}\\ w_{i'}>w_{i} \end{subarray}} c_{w_{i'}}\cdot v^{\vee}_{w_{i'}}
\Bigr).
\]
When $i'$ satisfies $w_{i'}>w_{i}$, we necessarily have $\ell(w_{i'})>\ell(w_{i})$ by the definition of the Bruhat order (see the beginning of Section \ref{subsec:triangle}).
In particular, we have $i'>i$.
This means that the action of $I_{\chi}(\mathbbm{1}_{\mu})$ on $V_{\chi}^{I}$ is triangulated with respect to the ordered basis $\{v^{\vee}_{w_{i}}\}_{i=1,\ldots,\#W}$.
As the diagonal entry corresponding to $v^{\vee}_{w_{i}}$ is given by $q(w_{i},\mu) \cdot \chi\circ\kappa_{M}^{-1}(w_{i}(\mu))$, we get the assertion.
\end{proof}

\subsection{General case}\label{subsec:Hecke-parahoric}

We next consider the general case.
Let $\mu\in\Lambda_{M}$ be a dominant element.
As explained in  Section \ref{subsec:parahoric}, $\mu$ defines the parahoric subgroup $J_{\mu}$ satisfying $I\subset J_{\mu}\subset K$. 
Recall that we have $J_{\mu}=IW_{\mu}I$, where we regard $W_{\mu}$ as a subgroup of $W_{\bfo} \subset \tilde{W}$ by using the bijection $W_{\bfo}\xrightarrow{1:1}W$. 

We let $\bfM_{\mu}\supset\bfM$ denote the Levi subgroup of $\G$ determined by $\mu$, i.e., for a root $\alpha\in\Phi$, $\bfU_{\alpha}\subset\bfM_{\mu}$ if and only if $\langle\alpha,\mu\rangle=0$.
We put $\Phi_{\red}^{+}(\bfM_{\mu}):=\{\alpha\in\Phi_{\red}^{+}\mid\langle\alpha,\mu\rangle=0\}$ and define
\[
I_{N}^{M_{\mu}}:=\prod_{\alpha\in\Phi_{\red}^{+}(\bfM_{\mu})} U_{\alpha,0}.
\]

\begin{prop}\label{prop:J-reduction}
We have
\[
IW_{\mu}I
=I_{N}^{M_{\mu}}W_{\mu}I
=\bigsqcup_{w\in W_{\mu}}I_{N}^{M_{\mu}}wI.
\]
\end{prop}

\begin{proof}
Since the second equality follows from the disjointness of the Iwahori decomposition $G=\bigsqcup_{w\in\tilde{W}}IwI$ (see Section \ref{subsec:Iwahori-Hecke}), it is enough to show the first equality.

By Lemma \ref{lem:Iwahori} (1), we have $I=I_{N}I_{M}I_{\ol{N}}$, which implies $IW_{\mu}I=I_{N}I_{M}I_{\ol{N}}W_{\mu}I$.
 Lemma \ref{lem:Iwahori} (2) shows that $I_{N}I_{M}I_{\ol{N}}W_{\mu}I=I_{N}I_{M}W_{\mu}I$.
As $W_{\mu}$ normalizes $I_{M}$, we get $IW_{\mu}I =I_{N}W_{\mu}I$.

Since we have $I_{N}=\prod_{\alpha\in\Phi_{\red}^{+}}U_{\alpha,0}$ with any order on $\Phi_{\red}^{+}$ and $wU_{\alpha,0}w^{-1} = U_{w(\alpha), 0}$, it suffices to check that $w(\alpha) \in \Phi^{+}_{\rm red}$ for any $w\in W_{\mu}$ and any $\alpha\in\Phi_{\red}^{+}$ satisfying $\langle\alpha,\mu\rangle\neq0$.
Let $w\in W_{\mu}$.
By the definition of $W_{\mu}$, we can write $w=s_{\beta_{1}}\cdots s_{\beta_{r}}$ with $\beta_{i}\in\Phi$ such that $\langle\beta_{i},\mu\rangle=0$.
If $\alpha\in\Phi_{\red}^{+}$ is a root satisfying $\langle\alpha,\mu\rangle\neq0$, then we have  $\langle\alpha,\mu\rangle>0$ as $\mu$ is dominant.
Hence we have
\begin{align*}
\langle s_{\beta_{r}}(\alpha),\mu\rangle
&=\langle \alpha-\langle\alpha,\beta_{r}^{\vee}\rangle\beta_{r},\mu\rangle\\
&=\langle\alpha,\mu\rangle-\langle\alpha,\beta_{r}^{\vee}\rangle\langle\beta_{r},\mu\rangle
=\langle\alpha,\mu\rangle>0.
\end{align*}
Thus the dominance of $\mu$ implies that $s_{\beta_{r}}(\alpha)$ is positive.
By applying the same argument to $s_{\beta_{r}}(\alpha)$, we know that $s_{\beta_{r-1}}(s_{\beta_{r}}(\alpha))$ satisfies $\langle s_{\beta_{r-1}}(s_{\beta_{r}}(\alpha)),\mu\rangle>0$ and is positive.
Repeating this procedure, we get $w(\alpha)\in\Phi_{\red}^{+}$.
\end{proof}

Recall that an order on the quotient $W/W_{\mu}$ induced by the Bruhat order on $W$ as follows. 
Define 
\[
W^{\mu}:=\{w\in W\mid \ell(w)\le\ell(ws_{\alpha})\text{ for all $\alpha\in \Delta$ with $\langle \alpha,\mu\rangle=0$}\}.
\] 
Then it follows from \cite[Proposition~1.10~(c)]{MR1066460} that the canonical quotient $W^{\mu} \to W/W_{\mu}$ is bijective. 
Since the set $W^{\mu}$ has a partial order induced from the Bruhat order of $W$, we can transport it to $W/W_{\mu}$ via the bijection $W^{\mu}\cong W/W_{\mu}$. 

\begin{lem}[{\cite[Proposition 2.5.1]{MR2133266}}]\label{lem:order3}
The quotient map $W \twoheadrightarrow W/W_{\mu}$ preserves the orders, namely, $wW_{\mu}\leq w'W_{\mu}$ if $w\leq w'$ in $W$.
\end{lem}

\begin{rem}\label{rem:Weyl}
For any $w\in W$ and $\alpha\in\Delta$, we have $\ell(w)\leq\ell(ws_{\alpha})$ if and only if $w(\alpha)\in\Phi^{+}$ by Lemma \ref{lem:order}.
Thus we have
\[
W^{\mu}=\{w\in W\mid w(\alpha)\in\Phi^{+}\text{ for all $\alpha\in \Delta$ with $\langle \alpha,\mu\rangle=0$}\}.
\]
Since $\mu\in\Lambda_M$ is dominant, any positive root $\alpha$ satisfying $\langle\alpha,\mu\rangle=0$ can be written as the sum of simple roots $\alpha_{i}$'s satisfying $\langle\alpha_{i},\mu\rangle=0$ with non-negative integer coefficients. 
Hence $W^{\mu}$ furthermore equals
\[
\{w\in W\mid w(\alpha)\in\Phi^{+}\text{ for all $\alpha\in \Phi^{+}$ with $\langle \alpha,\mu\rangle=0$}\}.
\]
\end{rem}

We let $e_{J_{\mu}}\in\mcH_{I}$ denote the idempotent corresponding to $J_{\mu}$, which is given explicitly by $dg(J_{\mu})^{-1}\mathbbm{1}_{J_{\mu}}$.
We put
\[
\mathbbm{1}_{\mu}=dg(J_{\mu})^{-1}\mathbbm{1}_{J_{\mu}\ul{\mu}J_{\mu}}.
\]

\begin{lem}\label{lem:proj}
We have a decomposition
\[
J_{\mu}/I = \bigsqcup_{w\in W_{\mu}}I_{N}^{M_{\mu}}wI/I.
\]
Moreover, for each $w\in W_{\mu}$, we have a bijection
\[
I_{N}^{M_{\mu}}/I_{N}^{M_{\mu}}[w]
\xrightarrow{\cong}
I_{N}^{M_{\mu}}wI/I
\colon x\mapsto xwI,
\]
where
\[
I_{N}^{M_{\mu}}[w]
:=\prod_{\alpha\in\Phi_{\red}^{+}(\bfM_{\mu})} U_{\alpha,r_{\alpha}},\quad
r_{\alpha}
:=
\begin{cases}
0& w^{-1}(\alpha)>0,\\
0+& w^{-1}(\alpha)<0.
\end{cases}
\]
\end{lem}

\begin{proof}
The first statement is an immediate consequence of Proposition \ref{prop:J-reduction}.

To show the second statement, let us take two elements $x,y\in I_{N}^{M_{\mu}}$ such that $xwI=ywI$.
Then we have $y^{-1}x\in wIw^{-1}$, hence $y^{-1}x\in I_{N}^{M_{\mu}}\cap wIw^{-1}$.
By a similar argument to the proof of Lemma \ref{lem:HKP-pre}, we can check that
\[
I_{N}^{M_{\mu}}\cap wIw^{-1}
=
\prod_{\alpha\in\Phi_{\red}^{+}(\bfM_{\mu})} U_{\alpha,r_{\alpha}},
\]
where $r_{\alpha}$ is as in the statement.
\end{proof}

\begin{prop}\label{prop:proj}
We have
\[
\mathbbm{1}_{\mu}
=e_{J_{\mu}}\ast T_{\mu} \ast e_{J_{\mu}},
\]
\end{prop}

\begin{proof}
Recall that $T_{\mu}=\mathbbm{1}_{I\ul{\mu}I}$.
Thus our task is to show that
\[
dg(J_{\mu})\cdot\mathbbm{1}_{J_{\mu}\ul{\mu}J_{\mu}}
=\mathbbm{1}_{J_{\mu}}\ast \mathbbm{1}_{I\ul{\mu}I} \ast \mathbbm{1}_{J_{\mu}}.
\]
Let us compute $\mathbbm{1}_{J_{\mu}}\ast \mathbbm{1}_{I\ul{\mu}I} \ast \mathbbm{1}_{J_{\mu}}$.
In general, for any $f_{1},f_{2},f_{3}\in C_{c}^{\infty}(G)$, we have
\begin{align*}
f_{1}\ast f_{2}\ast f_{3}(x)
&=\int_{G}f_{1}(g) (f_{2}\ast f_{3})(g^{-1}x)\,dg\\
&=\int_{G}f_{1}(g) \biggl(\int_{G} f_{2}(h) f_{3}(h^{-1}g^{-1}x)\,dh\biggr)\,dg\\
&=\int_{G}f_{1}(g) \biggl(\int_{G} f_{2}(g^{-1}xh) f_{3}(h^{-1})\,dh\biggr)\,dg.
\end{align*}
(In the last equality, we replaced $h$ with $g^{-1}xh$ by noting that $dh$ is a Haar measure on $G$.)
Hence we have
\begin{align*}
\mathbbm{1}_{J_{\mu}}\ast \mathbbm{1}_{I\ul{\mu}I} \ast \mathbbm{1}_{J_{\mu}}(x)
&=\int_{G}\mathbbm{1}_{J_{\mu}}(g) \biggl(\int_{G} \mathbbm{1}_{I\ul{\mu}I}(g^{-1}xh) \mathbbm{1}_{J_{\mu}}(h^{-1})\,dh\biggr)\,dg\\
&=\int_{J_{\mu}}\int_{J_{\mu}}\mathbbm{1}_{I\ul{\mu}I}(g^{-1}xh)\,dg\,dh.
\end{align*}
The integrand of the right-hand side is not zero if and only if $x$ belongs to $J_{\mu}\ul{\mu}J_{\mu}$.
Furthermore, we see that $\mathbbm{1}_{J_{\mu}}\ast \mathbbm{1}_{I\ul{\mu}I} \ast \mathbbm{1}_{J_{\mu}}(x)$ is constant for any $x\in J_{\mu}\ul{\mu}J_{\mu}$ again by noting that $dg$ and $dh$ are Haar measures on $G$ (hence of $J_{\mu}$).

Thus now it is enough to check that $\mathbbm{1}_{J_{\mu}}\ast \mathbbm{1}_{I\ul{\mu}I} \ast \mathbbm{1}_{J_{\mu}}(\ul{\mu})$ is given by $dg(J_{\mu})$.
By Lemma \ref{lem:proj}, we have 
\[
J_{\mu}/I = \bigsqcup_{w\in W_{\mu}}I_{N}^{M_{\mu}}wI/I,
\]
and, for each $w\in W_{\mu}$, we have a bijection
\[
I_{N}^{M_{\mu}}/I_{N}^{M_{\mu}}[w]
\xrightarrow{\cong}
I_{N}^{M_{\mu}}wI/I
\colon x\mapsto xwI.
\]
Thus we can take a complete set of representatives $\{g_{i}\}_{i=1}^{\#J_{\mu}/I}$ of the quotient $J_{\mu}/I$ so that each $g_{i}$ is given by $x_{i}w$ with some $x_{i}\in I_{N}^{M_{\mu}}$.
We note that $\ul{\mu}$-conjugation preserves $I_{N}^{M_{\mu}}$ and $I_{N}^{M_{\mu}}[w]$.
This fact follows from that $\langle\alpha,\mu\rangle=0$ for any $\alpha\in\Phi^{+}_{\red}(\bfM_{\mu})$ by a similar argument to the proof of Lemma \ref{lem:Iwahori} (3).
Hence, for each $x_{i}\in I_{N}^{M_{\mu}}$, there exists a unique $x_{i'}\in I_{N}^{M_{\mu}}$ satisfying 
\begin{align}\label{eq:J-I-1}
\ul{\mu}x_{i}\ul{\mu}^{-1}I_{N}^{M_{\mu}}[w]=x_{i'}I_{N}^{M_{\mu}}[w].
\end{align}
On the other hand, as $w$ commutes with $\mu\in\Lambda_{M}$ (as elements of $\tilde{W}$), we have 
\begin{align}\label{eq:J-I-2}
\ul{\mu}w\ul{\mu}^{-1}I_{M}=wI_{M}.
\end{align}
By combining equalities (\ref{eq:J-I-1}) and (\ref{eq:J-I-2}), we can check that
\[
\ul{\mu}x_{i}w\ul{\mu}^{-1}I=x_{i'}wI,
\]
or equivalently, $\ul{\mu}g_{i}\ul{\mu}^{-1}I=g_{i'}I$.
By taking the inverse, we get $I\ul{\mu}g_{i}^{-1}=Ig_{i'}^{-1}\ul{\mu}$.
This implies that, for any $g\in g_{i'}I$ and $h\in g_{j}I$, we have
\[
\mathbbm{1}_{I\ul{\mu}I}(g^{-1}\ul{\mu}h)
=\mathbbm{1}_{I\ul{\mu}I}(g_{i'}^{-1}\ul{\mu}g_{j})
=\mathbbm{1}_{I\ul{\mu}I}(\ul{\mu}g_{i}^{-1}g_{j}).
\]
Thus, by noting that the association $[g_{i}\mapsto g_{i'}]$ gives a bijection from $\{g_{i}\}_{i}$ to itself, we get
\[
\mathbbm{1}_{J_{\mu}}\ast \mathbbm{1}_{I\ul{\mu}I} \ast \mathbbm{1}_{J_{\mu}}(\ul{\mu})
=\int_{J_{\mu}}\int_{J_{\mu}}\mathbbm{1}_{I\ul{\mu}I}(g^{-1}\ul{\mu}h)\,dg\,dh
=\sum_{i=1}^{\#J_{\mu}/I}\sum_{j=1}^{\#J_{\mu}/I}\mathbbm{1}_{I\ul{\mu}I}(\ul{\mu}g_{i}^{-1}g_{j}).
\]
Now our task is to show that $\mathbbm{1}_{I\ul{\mu}I}(\ul{\mu}g_{i}^{-1}g_{j})\neq0$ if and only if $i=j$.

The ``if'' part is obviously true, so let us consider the ``only if'' part.
We suppose that $\mathbbm{1}_{I\ul{\mu}I}(\ul{\mu}g_{i}^{-1}g_{j})\neq0$, namely, $g_{i}^{-1}g_{j}\in{\ul{\mu}}^{-1}I\ul{\mu}I$.
Let $N_{\mu}$ be the unipotent radical of the standard parabolic subgroup with Levi subgroup $M_{\mu}$, and let $\overline{N}_{\mu}$ be its opposite.
If we put
\[
I_{\overline{N}_{\mu}}:=I\cap\overline{N}_{\mu},\quad
I_{M_{\mu}}:=I\cap M_{\mu},\quad
I_{N_{\mu}}:=I\cap N_{\mu},
\]
then we have $I=I_{N_{\mu}}I_{M_{\mu}}I_{\overline{N}_{\mu}}$ and $I=I_{\overline{N}_{\mu}}I_{M_{\mu}}I_{N_{\mu}}$(Lemma \ref{lem:Iwahori} (1)).

By Lemma \ref{lem:Iwahori} (3), we have ${\ul{\mu}}^{-1}I_{\overline{N}_{\mu}}\ul{\mu}\subset I_{\overline{N}_{\mu}}$ and ${\ul{\mu}}^{-1}I_{N_{\mu}}\ul{\mu}\supset I_{N_{\mu}}$ by the dominance of $\mu$.
Moreover, by a similar argument to the proof of Lemma \ref{lem:Iwahori} (3), we can show that ${\ul{\mu}}^{-1}U_{\alpha,r}{\ul{\mu}}=U_{\alpha,r}$ for any $r\in\R$ and any $\alpha$ whose root subgroup $U_{\alpha}$ is contained in $\bfM_{\mu}$.
Accordingly, we have ${\ul{\mu}}^{-1}I_{M_{\mu}}\ul{\mu}=I_{M_{\mu}}$.
Thus we have
\begin{align*}
{\ul{\mu}}^{-1}I\ul{\mu}I
&={\ul{\mu}}^{-1}(I_{N_{\mu}}I_{M_{\mu}}I_{\overline{N}_{\mu}})\ul{\mu}I\\
&={\ul{\mu}}^{-1}I_{N_{\mu}}\ul{\mu}I\\
&={\ul{\mu}}^{-1}I_{N_{\mu}}\ul{\mu}(I_{N_{\mu}}I_{M_{\mu}}I_{\overline{N}_{\mu}})\\
&={\ul{\mu}}^{-1}I_{N_{\mu}}\ul{\mu}I_{M_{\mu}}I_{\overline{N}_{\mu}}.
\end{align*}
Recall that the multiplication map
\[
N_{\mu}\times M_{\mu}\times \overline{N}_{\mu} \rightarrow G
\]
is injective (see \cite[Th\'eor\`eme 2.2.3]{MR756316}) and note that ${\ul{\mu}}^{-1}I_{N_{\mu}}\ul{\mu}$, $I_{M_{\mu}}$, and $I_{\overline{N}_{\mu}}$ are contained in $N_{\mu}$, $M_{\mu}$, and $\overline{N}_{\mu}$, respectively.
Hence, as $g_{i}^{-1}g_{j}$ lies in $M_{\mu}$, the assumption that $g_{i}^{-1}g_{j}\in{\ul{\mu}}^{-1}I\ul{\mu}I$ implies that $g_{i}^{-1}g_{j}$ belongs to $I_{M_{\mu}}$.
This means that $g_{i}I$ and $g_{j}I$ are the same $I$-coset, thus we have $g_{i}=g_{j}$.
\end{proof}

\begin{lem}\label{lem:parah-Iwah}
For any $w\in W$, we have
\[
M_{1}NwJ_{\mu}
=\bigsqcup_{w'\in W_{\mu}}M_{1}Nww'I.
\]
\end{lem}

\begin{proof}
Since $W^{\mu}\xrightarrow{1:1}W/W_{\mu}$ and $J_{\mu}$ contains $W_{\mu}$, it is enough to treat only the case where $w\in W^{\mu}$. 

By Proposition \ref{prop:J-reduction}, we have $J_{\mu}=\bigsqcup_{w'\in W_{\mu}}I_{N}^{M_{\mu}}w'I$.
Hence we have
\[
M_{1}NwJ_{\mu}
=\bigcup_{w'\in W_{\mu}}M_{1}NwI_{N}^{M_{\mu}}w'I.
\]
By the definition of $W^{\mu}$ and Remark \ref{rem:Weyl}, we have $w(\alpha)\in\Phi^{+}$ for  any $\alpha\in\Phi^{+}$ satisfying $\langle\alpha,\mu\rangle=0$. 
This fact shows that $wI_{N}^{M_{\mu}}w^{-1}\subset N$, and hence we get 
\[
M_{1}NwJ_{\mu}
=\bigcup_{w'\in W_{\mu}}M_{1}Nww'I.
\]
Since the decomposition $G=\bigcup_{w'\in \tilde{W}}M_{1}Nw'I$ is disjoint (see Section \ref{subsec:univ}), this decomposition is disjoint.
\end{proof}

For $w\in W$, we put $v_{w}^{J_{\mu}}:=\sum_{w'\in W_{\mu}}v_{ww'}$.

\begin{lem}\label{lem:averaging}
For any $w\in W$, we have 
\[
v_{w}\ast e_{J_{\mu}}
=
\#W_{\mu}^{-1}\cdot v_{w}^{J_{\mu}}.
\]
\end{lem}

\begin{proof}
Recall that $v_{w}=\mathbbm{1}_{M_{1}NwI}$ and $e_{J_{\mu}}=dg(J_{\mu})^{-1}\cdot\mathbbm{1}_{J_{\mu}}$.
We have
\[
\mathbbm{1}_{M_{1}NwI}\ast\mathbbm{1}_{J_{\mu}}(x)
=\int_{G}\mathbbm{1}_{M_{1}NwI}(g)\mathbbm{1}_{J_{\mu}}(g^{-1}x)\,dg
=dg(M_{1}NwI\cap xJ_{\mu}).
\]
Thus, since $I \subset J_{\mu}$, we have $\mathrm{Supp}(\mathbbm{1}_{M_{1}NwI}\ast\mathbbm{1}_{J_{\mu}}) = M_{1}NwJ_{\mu}$. 
Suppose that $x\in M_{1}NwJ_{\mu}$ and write $x=mnwj$ with $m\in M_{1}$, $n\in N$, and $j\in J_{\mu}$.
Then, by noting that $dg$ is a Haar measure on $G$ and that $M$ normalizes $N$, we have
\[
dg(M_{1}NwI\cap xJ_{\mu})
=dg(M_{1}NwI\cap mnwJ_{\mu})
=dg(M_{1}NwI\cap wJ_{\mu}).
\]
This fact implies that  $\mathbbm{1}_{M_{1}NwI}\ast\mathbbm{1}_{J_{\mu}}$ is equal to constant multiple of $\mathbbm{1}_{M_{1}NwJ_{\mu}}$.
Since we have 
\[
M_{1}NwJ_{\mu}
=\bigsqcup_{w'\in W_{\mu}}M_{1}Nww'I
\]
by Lemma \ref{lem:parah-Iwah}, there is a constant $C\in\C$ such that $v_{w}\ast e_{J_{\mu}}=C\cdot v_{w}^{J_{\mu}}$.
Since $e_{J_{\mu}}$ is an idempotent, we have
\[
v_{w}^{J_{\mu}}\ast e_{J_{\mu}}
=(C^{-1}\cdot v_{w}\ast e_{J_{\mu}})\ast e_{J_{\mu}}
=C^{-1}\cdot v_{w}\ast e_{J_{\mu}}
=v_{w}^{J_{\mu}}.
\]
On the other hand, we have
\[
v_{w}^{J_{\mu}}\ast e_{J_{\mu}}
=\Bigl(\sum_{w'\in W_{\mu}}v_{ww'}\Bigr)\ast e_{J_{\mu}}
=\sum_{w'\in W_{\mu}}C\cdot v_{w}^{J_{\mu}}
=\#W_{\mu}\cdot C\cdot v_{w}^{J_{\mu}}.
\]
Thus we get $C=\#W_{\mu}^{-1}$.
\end{proof}

Since $(-)\ast e_{J_{\mu}}$ gives a projector from $\mcM$ onto $\mcM^{J_{\mu}}$, Lemma \ref{lem:averaging} implies that $\{v_{w}^{J_{\mu}}\}_{w\in W/W_{\mu}}$ forms an $\mcR$-basis of $\mcM^{J_{\mu}}$.

\begin{prop}\label{prop:Haines-parahoric}
For any $w\in W/W_{\mu}$, there exists a family $\{a_{w'}\}_{w'\in W/W_{\mu}, w'<w}$ of elements of $\mcR$ satisfying
\[
v^{J_{\mu}}_{w}\ast(e_{J_{\mu}}\ast\Theta_{\mu}\ast e_{J_{\mu}})
=
q^{\langle\rho_{\bfP}-\rho^{\nr},w(\mu)\rangle}\cdot R_{w(\mu)}\cdot v^{J_{\mu}}_{w}+\sum_{\begin{subarray}{c}w'\in W/W_{\mu} \\ w'<w \end{subarray}} a_{w'}\cdot v^{J_{\mu}}_{w'}.
\]
\end{prop}

\begin{proof}
We have
\[
v^{J_{\mu}}_{w}\ast(e_{J_{\mu}}\ast\Theta_{\mu})
=v^{J_{\mu}}_{w}\ast\Theta_{\mu}
=\sum_{w'\in W_{\mu}}v_{ww'}\ast\Theta_{\mu}.
\]
By applying Proposition \ref{prop:Haines} to each $v_{ww'}\ast\Theta_{\mu}$, we have
\[
\sum_{w'\in W_{\mu}}v_{ww'}\ast\Theta_{\mu}
=
\sum_{w'\in W_{\mu}}\Bigl(q^{\langle\rho_{\bfP}-\rho^{\nr},ww'(\mu)\rangle}\cdot R_{ww'(\mu)}\cdot v_{ww'}+\sum_{\begin{subarray}{c}w''\in W \\ w''<ww' \end{subarray}} a^{(w')}_{w''}\cdot v_{w''} \Bigr),
\]
where $a^{(w')}_{w''}\in\mcR$ is an element determined by $w'$ and $w''$.
By noting that $ww'(\mu)=w(\mu)$ for any $w'\in W_{\mu}$, we get
\begin{align*}
\sum_{w'\in W_{\mu}}q^{\langle\rho_{\bfP}-\rho^{\nr},ww'(\mu)\rangle}\cdot R_{ww'(\mu)}\cdot v_{ww'}
&= q^{\langle\rho_{\bfP}-\rho^{\nr},w(\mu)\rangle}\cdot R_{w(\mu)}\cdot\sum_{w'\in W_{\mu}} v_{ww'}\\
&= q^{\langle\rho_{\bfP}-\rho^{\nr},w(\mu)\rangle}\cdot R_{w(\mu)}\cdot v_{w}^{J_{\mu}}.
\end{align*}
On the other hand, by Lemma \ref{lem:averaging}, we have
\[
\Bigl(\sum_{w'\in W_{\mu}}\sum_{\begin{subarray}{c}w''\in W \\ w''<ww' \end{subarray}} a^{(w')}_{w''}\cdot v_{w''} \Bigr)\ast e_{J_{\mu}}
=
\#W_{\mu}^{-1}\cdot\sum_{w'\in W_{\mu}}\sum_{\begin{subarray}{c}w''\in W \\ w''<ww' \end{subarray}} a^{(w')}_{w''}\cdot v_{w''}^{J_{\mu}}.
\]
By Lemma \ref{lem:order3}, for any $w'\in W_{\mu}$ and $w''\in W$ satisfying $w''<ww'$, we have $w''W_{\mu}<wW_{\mu}$ .
This implies that we have
\[
\#W_{\mu}^{-1}\cdot\sum_{w'\in W_{\mu}}\sum_{\begin{subarray}{c}w''\in W \\ w''<ww' \end{subarray}} a^{(w')}_{w''}\cdot v_{w''}^{J_{\mu}}
=
\sum_{\begin{subarray}{c}w'\in W/W_{\mu} \\ w'<w\end{subarray}} a'_{w'}\ast v_{w'}^{J_{\mu}}
\]
by choosing $a'_{w'}$ for each $w'\in W/W_{\mu}$ satisfying $w'<w$ appropriately.
\end{proof}

\begin{prop}\label{prop:triangle-parahoric}
There exists a $\C$-basis $\{v^{J_{\mu},\vee}_{w}\}_{w\in W/W_{\mu}}$ of $V_{\chi}^{J_{\mu}}$ such that, for any $w\in W/W_{\mu}$, there exists a family $\{c_{w'}\}_{w'\in W/W_{\mu}, w'>w}$ of complex numbers satisfying
\[
I_{\chi}(\mathbbm{1}_{\mu})\cdot v^{J_{\mu},\vee}_{w}
=
q(w,\mu) \cdot \Bigl(
\chi\circ\kappa_{M}^{-1}(w(\mu))\cdot v^{J_{\mu},\vee}_{w}+\sum_{\begin{subarray}{c}w'\in W/W_{\mu} \\ w'>w \end{subarray}} c_{w'}\cdot v^{J_{\mu},\vee}_{w'}
\Bigr).
\]
\end{prop}

\begin{proof}
Note that, by Proposition \ref{prop:proj}, we have $\mathbbm{1}_{\mu}=e_{J_{\mu}}\ast T_{\mu} \ast e_{J_{\mu}}$.
Thus, by Remark \ref{rem:left-right}, the left action $I_{\chi}(\mathbbm{1}_{\mu})$ on $V_{\chi}^{J_{\mu}}$ coincides with the right action of 
$\iota(e_{J_{\mu}}\ast T_{\mu}\ast e_{J_{\mu}}) \in \mcH_{I}$ on $\C\otimes_{\mcR,\chi^{-1}}\mcM^{J_{\mu}}$.
By Lemma \ref{cor:Theta-H}, we have
\[
\iota(e_{J_{\mu}}\ast T_{\mu}\ast e_{J_{\mu}})
=q^{\langle\rho^{\nr},\mu\rangle} \cdot \iota(e_{J_{\mu}}\ast\Theta_{\mu}\ast e_{J_{\mu}}).
\]

Since the perfect pairing $(-,-)$ introduced in the proof of Proposition \ref{prop:triangle-Iwahori} is anti-invariant with respect to the action of the Iwahori--Hecke algebra (the property (B) in the proof of Proposition \ref{prop:triangle-Iwahori}), it canonically induces a perfect pairing
\[
(-,-)_{\chi} \colon V_{\chi}^{J_{\mu}} \times V_{\chi^{-1}}^{J_{\mu}} \longrightarrow \C. 
\]
Thus, by choosing a $\C$-basis of $\{v^{J_{\mu},\vee}_{w}\}_{w\in W/W_{\mu}}$ of $V_{\chi}^{J_{\mu}}$ to be the dual to $\{v^{J_{\mu}}_{w}\}_{w\in W/W_{\mu}}$ with respect to this pairing, the same argument as in the proof of Proposition \ref{prop:triangle-Iwahori} works using Proposition \ref{prop:Haines-parahoric} instead of Proposition \ref{prop:Haines}.
\end{proof}

With notations as in Proposition \ref{prop:triangle-parahoric}, we introduce a diagonalizable operator $A_{\mu}$ on $V_{\chi}^{J_{\mu}}$ given by $A_{\mu}(v_{w}^{J_{\mu},\vee})=q(w,\mu)^{-1}\cdot v_{w}^{J_{\mu},\vee}$.

\begin{cor}\label{cor:OST4.3}
 We have
\[
\det\bigl(1-q^{-s}\cdot c\cdot A_{\mu}\circ I_{\chi}(\mathbbm{1}_{\mu})\,\big\vert\, V_{\chi}^{J_{\mu}}\bigr)
=
\prod_{w\in W/W_{\mu}} \bigl(1-q^{-s}\cdot c\cdot\chi\circ\kappa_{M}^{-1}(w(\mu))\bigr)
\]
for any $c\in\C$.
\end{cor}

\begin{proof}
Recall that there exists a complete set $W^{\mu}$ of representatives of the quotient $W/W_{\mu}$ and that the order on $W/W_{\mu}$ is nothing but the order transported from the Bruhat order on $W^{\mu}\subset W$.
By noting this, we can carry out the same argument as in the proof of Corollary \ref{cor:OST4.3-Iwahori}.
To be more precise, we put 
\[
W^{\mu}(k):=\{w\in W^{\mu} \mid \ell(w)=k\}
\]
for each $k\in\Z_{\geq0}$ and define a total order on $W^{\mu}$ such that 
\begin{align*}
W^{\mu}(0)&=\{w_{1},\ldots,w_{\#W^{\mu}(0)}\},\\
W^{\mu}(1)&=\{w_{\#W^{\mu}(0)+1},\ldots,w_{\#W^{\mu}(0)+\#W^{\mu}(1)}\},\\
&\vdots\\
W^{\mu}(h)&=\{w_{\#W^{\mu}(0)+\cdots+\#W^{\mu}(h-1)+1},\ldots,w_{\#W^{\mu}}\}.
\end{align*}
Then, if we order the $\C$-basis $\{v^{J_{\mu},\vee}_{w}\}_{w\in W/W_{\mu}}$ of $V_{\chi}^{J_{\mu}}$ as in Proposition \ref{prop:triangle-parahoric} according to this total order, Proposition \ref{prop:triangle-parahoric} shows that the action of $I_{\chi}(\mathbbm{1}_{\mu})$ on $V_{\chi}^{J_{\mu}}$ is triangulated with respect to the ordered basis $\{v^{J_{\mu},\vee}_{w_{i}}\}_{i=1,\ldots,\#W^{\mu}}$.
As the diagonal entry corresponding to $v^{J_{\mu},\vee}_{w_{i}}$ is given by $q(w_{i},\mu) \cdot \chi\circ\kappa_{M}^{-1}(w_{i}(\mu))$, we get the assertion.
\end{proof}

\section{Relation to the local $L$-functions}\label{sec:L}

\subsection{Representations with parahoric fixed vectors}\label{subsec:parahoric-fixed}

We recall basic a fact about irreducible smooth representations of $G$ having a non-zero fixed vector by a parahoric subgroup following \cite{MR3444233}.

Let $J\subset G$ be a parahoric subgroup of $G$.
\begin{defn}[$J$-spherical representation]\label{defn:spherical}
We say that an irreducible smooth representation $\pi$ of $G$ is \textit{$J$-spherical} if $\pi$ has a nonzero vector fixed by $J$.
\end{defn}

In the following, we assume that $J$ contains the fixed Iwahori subgroup $I$.
Note that then any $J$-spherical representation is $I$-spherical.
We also remark that this assumption is always satisfied up to conjugacy since
\begin{itemize}
\item
any parahoric subgroup contains an Iwahori subgroup, and
\item
any Iwahori subgroups are conjugate.
\end{itemize}

\begin{prop}[{\cite[Section 11.5]{MR3444233}}]\label{prop:cusp-supp}
Let $\pi$ be an $I$-spherical irreducible smooth representation of $G$.
Then there exists an weakly unramified character $\chi\in X^{\w}(M)$ of M such that $\pi$ is a subquotient of the normalized parabolic induction $\nInd_{P}^{G}\chi$.
Moreover, such an weakly unramified character $\chi$ is unique up to the action of the Weyl group $W=W(\G,\bfA)$.
\end{prop}

\begin{rem}\label{rem:unram-rep}
When $\G$ is unramified (i.e., quasi-split and splits over an unramified extension of $F$) and $J$ is a hyperspecial maximal open compact subgroup of $G$, the above result is nothing but the well-known classification of unramified representations via the Satake isomorphism (e.g., see \cite[Section~4]{MR546593} for the details).
\end{rem}

\subsection{Satake parameters of parahoric-spherical representations}\label{subsec:Satake}

We review the construction of the Satake parameters of parahoric-spherical representations according to Haines \cite{MR3455870,MR3671514}.

\subsubsection{Quasi-split case}\label{subsubsec:q-spl}
We first consider the case where $\G$ is quasi-split (see \cite[Sections 6 and 7]{MR3455870} for the details of the content of this section).
In this case, the centralizer $\bfM$ of the maximal $F$-split torus $\bfA$ in $\G$ is a maximal torus, so we write $\bfT$ for $\bfM$.
As the minimal parabolic subgroup $\bfP$ is Borel, let us write $\bfB$ for $\bfP$.
From the tuple $(\G,\bfB,\bfT)$, we get the corresponding root datum
\[
\Psi(\G)=\bigl(X^{\ast}(\bfT), \Delta_{\bfB}, X_{\ast}(\bfT), \Delta_{\bfB}^{\vee}\bigr),
\]
where $\Delta_{\bfB}$ (resp.\ $\Delta_{\bfB}^{\vee}$) is the set of simple roots (resp.\ coroots) of $\bfT$ determined by $\bfB$.
By taking the dual of this root datum, we get the Langlands dual group $\hat{\G}$ of $\G$. 
To be more precise, $\hat{\G}$ is a connected reductive group over $\C$ with the following fixed data: 
\begin{itemize}
\item a maximal torus $\mcT$ of $\hat{\G}$, 
\item  a Borel subgroup  $\mcB$ of $\hat{\G}$ containing $\mcT$, 
\item an isomorphism $\bm{\iota}$ between the root datum $\Psi(\hat{\G})=(X^{\ast}(\mcT), \Delta_{\mcB}, X_{\ast}(\mcT), \Delta_{\mcB}^{\vee})$ of $\hat{\G}$ and the dual root datum $\Psi(\G)^{\vee}=(X_{\ast}(\bfT), \Delta_{\bfB}^{\vee}, X^{\ast}(\bfT), \Delta_{\bfB})$ of $\G$.
\end{itemize}

Recall that the Kottwitz homomorphism gives an isomorphism
\[
\kappa_{T}\colon T/T_{1}\xrightarrow{\cong}X^{\ast}(\hat{\bfT}^{I_{F}})^{\Frob}
\]
(see Section \ref{subsec:Iwahori}, note that now we have $Z(\hat{\bfT})=\hat{\bfT}$).
This induces an isomorphism
\[
X^{w}(M)=\Hom(T/T_{1},\C^{\times})\cong (\hat{\bfT}^{I_{F}})_{\Frob}\colon\chi\mapsto\hat{\chi},
\]
which is characterized by the identity
\[
\chi(\kappa_{T}^{-1}(\lambda))
=\lambda(\hat{\chi})
\]
for any $\lambda\in X^{\ast}(\hat{\bfT}^{I_{F}})^{\Frob}$.

We consider a map
\[
\hat{\bfT}^{I_{F}}\hookrightarrow(\hat{\G}^{I_{F}}\rtimes\Frob)_{\mathrm{ss}}\colon t\mapsto t\rtimes\Frob,
\]
where $(\hat{\G}^{I_{F}}\rtimes\Frob)_{\mathrm{ss}}$ denotes the semisimple locus in $\hat{\G}^{I_{F}}\rtimes\Frob$.
Here $\hat{\bfT}$ is regarded as a subgroup of $\hat{\G}$ via the isomorphism $\hat{\bfT}\cong\mcT$ induced by the fixed isomorphism $\bm{\iota}$.
Then, according to \cite[Proposition 6.1]{MR3455870}, this map induces a bijection
\[
(\hat{\bfT}^{I_{F}})_{\Frob}/W\xrightarrow{\cong}(\hat{\G}^{I_{F}}\rtimes\Frob)_{\mathrm{ss}}/\hat{\G}^{I_{F}}.
\]

Let $\pi$ be an Iwahori-spherical irreducible smooth representation of $G$.
Then, by Proposition \ref{prop:cusp-supp}, an element $\chi$ of $X^{\w}(T)$ is determined by $\pi$ uniquely up to $W$-conjugation.
We define the \textit{Satake parameter $s(\pi)$} of $\pi$ to be the image of $\hat{\chi}\rtimes\Frob\in(\hat{\G}^{I_{F}}\rtimes\Frob)_{\mathrm{ss}}$ in $(\hat{\G}^{I_{F}}\rtimes\Frob)_{\mathrm{ss}}/\hat{\G}^{I_{F}}$.

\subsubsection{Non-quasi-split case}\label{subsubsec:non-q-spl}
We next consider the case where $\G$ is not quasi-split (see \cite[Sections 8 and 9]{MR3455870} for the details of the content of this section).
In this case, we take the quasi-split inner form $\G^{\ast}$ of $\G$ over $F$ with an inner twist $\psi\colon\G\rightarrow\G^{\ast}$.
We fix a maximal $F$-split torus $\bfA^{\ast}$ of $\G^{\ast}$ and put $\bfT^{\ast}$ to be the centralizer of $\bfA^{\ast}$ in $\G^{\ast}$.
We also fix a Borel subgroup $\bfB^{\ast}$ of $\G^{\ast}$ containing $\bfT^{\ast}$.
For the $F$-rational parabolic subgroup $\bfP$ of $\G$ with minimal Levi subgroup $\bfM$ of $\G$, by replacing $\psi$ if necessary, there exists a parabolic subgroup $\bfP^{\ast}$ of $\G^{\ast}$ such that $\psi(\bfP)=\bfP^{\ast}$, $\psi(\bfM)=\bfM^{\ast}$ and $\bfP^{\ast}\supset \bfB^{\ast}$.
Then we get a Galois-equivariant isomorphism
\[
\hat{\psi}\colon Z(\hat{\bfM})\xrightarrow{\cong}Z(\hat{\bfM}^{\ast}).
\]

Since the Langlands dual group $\hat{\bfM}^{\ast}$ of $\bfM^{\ast}$ is realized as a Levi subgroup of $\hat{\bfG}^{\ast}$ containing the maximal torus $\hat{\bfT}^{\ast}$, we have an inclusion $Z(\hat{\bfM}^{\ast})\hookrightarrow\hat{\bfT}^{\ast}$.
Thus we get a Galois-equivariant homomorphism $\hat{\psi}_{0}\colon Z(\hat{\bfM})\cong Z(\hat{\bfM}^{\ast})\hookrightarrow\hat{\bfT}^{\ast}$.
We define a map $\tilde{t}_{\bfA^{\ast},\bfA}$ from $(Z(\hat{\bfM})^{I_{F}})_{\Frob}$ to $(\hat{\bfT}^{\ast I_{F}})_{\Frob}$ by
\begin{align*}
\tilde{t}_{\bfA^{\ast},\bfA}\colon(Z(\hat{\bfM})^{I_{F}})_{\Frob}&\rightarrow(\hat{\bfT}^{\ast I_{F}})_{\Frob}\\
\hat{\chi}&\mapsto \delta_{B^{\ast}}^{-\frac{1}{2}}\cdot\hat{\psi}_{0}(\delta_{P}^{\frac{1}{2}}\hat{\chi}).
\end{align*}
Here, $\delta_{B^{\ast}}^{\frac{1}{2}}$ is an weakly unramified character of $T^{\ast}$, hence can be regarded as an element of $(\hat{\bfT}^{\ast I_{F}})_{\Frob}$ through the isomorphism $X^{\w}(T^{\ast})\cong (\hat{\bfT}^{\ast I_{F}})_{\Frob}$ induced from the Kottwitz homomorphism.
Similarly, $\delta_{P}^{\frac{1}{2}}$ is an weakly unramified character of $M$ and regarded as an element of $Z(\hat{\bfM}^{I_{F}})_{\Frob}$ through the isomorphism $X^{\w}(M)\cong Z(\hat{\bfM}^{I_{F}})_{\Frob}$ induced from the Kottwitz homomorphism.
The map $\tilde{t}_{\bfA^{\ast},\bfA}$ induces a map
\[
(Z(\hat{\bfM})^{I_{F}})_{\Frob}/W(\G,\bfA)\rightarrow(\hat{\bfT}^{\ast I_{F}})_{\Frob}/W(\G^{\ast},\bfA^{\ast})
\]
(see \cite[Lemma 8.1]{MR3455870}), for which we again write $\tilde{t}_{\bfA^{\ast},\bfA}$.

On the other hand, as explained in the quasi-split case, we have
\[
(\hat{\bfT}^{\ast I_{F}})_{\Frob}/W(\G^{\ast},\bfA^{\ast})\xrightarrow{\cong}(\hat{\G}^{\ast I_{F}}\rtimes\Frob)_{\mathrm{ss}}/\hat{\G}^{\ast I_{F}}.
\]
Since the Langlands dual groups $\hat{\G}$ and $\hat{\G}^{\ast}$ are isomorphic Galois-equivariantly, we have 
\[
(\hat{\G}^{\ast I_{F}}\rtimes\Frob)_{\mathrm{ss}}/\hat{\G}^{\ast I_{F}}
\cong
(\hat{\G}^{I_{F}}\rtimes\Frob)_{\mathrm{ss}}/\hat{\G}^{I_{F}}.
\]
Therefore, by putting all of these maps together, we get a map
\begin{align*}
(Z(\hat{\bfM})^{I_{F}})_{\Frob}/W(\G,\bfA)
&\rightarrow(\hat{\G}^{I_{F}}\rtimes\Frob)_{\mathrm{ss}}/\hat{\G}^{I_{F}}\\
\hat{\chi}&\mapsto \tilde{t}_{\bfA^{\ast},\bfA}(\hat{\chi})\rtimes\Frob.
\end{align*}

Let $\pi$ be an Iwahori-spherical irreducible smooth representation of $G$.
Then, by Proposition \ref{prop:cusp-supp}, an element $\chi$ of $X^{\w}(M)$ is determined by $\pi$ uniquely up to $W$-conjugation.
We define the \textit{Satake parameter $s(\pi)$} of $\pi$ to be the image of $\tilde{t}_{\bfA^{\ast},\bfA}(\hat{\chi})\rtimes\Frob\in(\hat{\G}^{I_{F}}\rtimes\Frob)_{\mathrm{ss}}$ in $(\hat{\G}^{I_{F}}\rtimes\Frob)_{\mathrm{ss}}/\hat{\G}^{I_{F}}$.

For our convenience, for any $\chi\in X^{\w}(M)$, we let $\chi^{\ast}\in X^{\w}(T^{\ast})$ denote the image of $\tilde{t}_{\bfA^{\ast},\bfA}(\hat{\chi})\in (\hat{\bfT}^{\ast I_{F}})_{\Frob}$ under the map $X^{\w}(T^{\ast})\cong (\hat{\bfT}^{\ast I_{F}})_{\Frob}$:
\[
\xymatrix{
X^{\w}(M)\ar^-{\cong}[r]\ar_-{(-)^{\ast}}[d]& (Z(\hat{\bfM})^{I_{F}})_{\Frob}\ar^-{\tilde{t}_{\bfA^{\ast},\bfA}}[d] & \chi\ar@{|->}[r]&\hat{\chi}\ar@{|->}[d]\\
X^{\w}(T^{\ast})\ar^-{\cong}[r]& (\hat{\bfT}^{\ast I_{F}})_{\Frob} &\chi^{\ast}&\tilde{t}_{\bfA^{\ast},\bfA}(\hat{\chi})\ar@{|->}[l]
}
\]

\subsection{Local $L$-functions for parahoric-spherical representations}\label{subsec:L-factor}

According to \cite[Section 2.6]{MR546608}, we take a finite-dimensional continuous representation $(r,V)$ of ${}^{L}\G$ whose restriction to $\hat{\G}$ is an algebraic homomorphism of complex Lie groups $\hat\G\rightarrow\GL_{\C}(V)$.
Note that the continuity implies that $r$ factors through the quotient $\hat{\G}\rtimes\Gal(E/F)$ for a finite Galois extension $E$ of $F$ over which $\G$ splits.

\begin{defn}\label{def:L-function}
For an $I$-spherical irreducible smooth representation $\pi$ of $G$, we define the \textit{semi-simple local $L$-function of $\pi$ with respect to $r$} by
\[
L_{\mathrm{ss}}(s,\pi,r)
:=
\det\bigl(1- q^{-s}\cdot r(s(\pi)) \,\big\vert\, V^{I_{F}}\bigr)^{-1},
\]
where $V^{I_{F}}$ denotes the subspace of $V$ consisting of $I_{F}$-fixed vectors.
\end{defn}

\begin{rem}\label{rem:ss-L}
A meaning of the semi-simple local $L$-function can be explained as follows.
If we believe the conjectural local Langlands correspondence for $\G$, we should have an $L$-parameter $\phi_{\pi}$ of $\G$ for any irreducible smooth representation $\pi$ of $G$.
Recall that an $L$-parameter of $\G$ is a homomorphism from the product $W_{F}\times\SL_{2}(\C)$ of the Weil group and $\SL_{2}(\C)$ to the $L$-group ${}^{L}\G$ satisfying several conditions (see, for example, \cite[Section 3.2]{MR2730575} or \cite[Section 4]{MR3444233} for the precise definition).
For an $L$-parameter $\phi$ of $\G$, its local $L$-function with respect to $r$ is defined by
\[
L(s,\phi,r)
:=
\det\bigl(1- q^{-s}\cdot r(\phi(\Frob)) \,\big\vert\, V^{\phi(I_{F})}\bigr)^{-1}.
\]
On the other hand, according to \cite[Section 5.1]{MR3444233}, for an $L$-parameter $\phi$ of $\G$, its \textit{infinitesimal character} $\phi_{\mathrm{ss}}\colon W_{F}\rightarrow{}^{L}\G$ of $\phi$ is defined by $\phi_{\mathrm{ss}}:=\phi\circ\eta$, where
\[
\eta\colon W_{F}\rightarrow W_{F}\times\SL_{2}(\C);\quad
\sigma\mapsto \Biggl(\sigma,\begin{pmatrix}|\sigma|^{\frac{1}{2}}&0\\0&|\sigma|^{-\frac{1}{2}}\end{pmatrix}\Biggr).
\]
Here $|\sigma|$ denotes the absolute value of $\sigma\in W_{F}$ normalized so that $|\Frob|=q^{-1}$.
It is expected that any parahoric-spherical representation of $G$ corresponds to an $L$-parameter $\phi$ which is trivial on $I_{F}$ (i.e., $\phi(\sigma,1)=1\rtimes\sigma$ for any $\sigma\in I_{F}$) under the local Langlands correspondence.
Furthermore, it is expected that the Satake parameter $s(\pi)$ of a parahoric-spherical representation $\pi$ describes the image of the geometric Frobenius under the infinitesimal character $\phi_{\pi,\mathrm{ss}}$ of the $L$-parameter $\phi_{\pi}$ of $\pi$, i.e., $s(\pi)=\phi_{\pi,\mathrm{ss}}(\Frob)$ (see \cite[Conjecture 13.1]{MR3455870}).
Therefore, for any parahoric-spherical representation $\pi$ of $G$, we should have $L_{\mathrm{ss}}(s,\pi,r)=L(s,\phi_{\pi,\mathrm{ss}},r)$.
\end{rem}

\begin{rem}\label{rem:unram-L}
When $\G$ is unramified (i.e., $\G$ is quasi-split and splits over an unramified extension of $F$) and $\pi$ is an unramified representation (i.e., a $J$-spherical representation for a hyperspecial parahoric subgroup $J$ of $G$), the Satake parameter $s(\pi)$ is nothing but the classical Satake parameter of $\pi$ (see, for example, \cite{MR546593}).
In this case, the $L$-parameter $\phi_{\pi}$ of $\pi$ is defined just by
\begin{align*}
\phi_{\pi}&\colon W_{F}\times\SL_{2}(\C)\rightarrow\hat{\G}\rtimes W_{F};\quad\\
&\begin{cases}
(\Frob,1)\mapsto s(\pi), & \\
(\sigma,g)\mapsto 1\rtimes\sigma & \text{for any $(\sigma,g)\in I_{F}\times\SL_{2}(\C)$}.
\end{cases}
\end{align*}
Hence we have $\phi_{\pi,\mathrm{ss}}=\phi_{\pi}$ and $L_{\mathrm{ss}}(s,\pi,r)=L(s,\pi,r)=L(s,\phi_{\pi},r)$.
\end{rem}

We will rewrite the above definition of the semisimple local $L$-function in a different form by using the next 
\begin{lem}\label{lem:det}
Let $W$ be finite dimensional $\C$-vector space and $A\colon W\rightarrow W$ be a $\C$-linear automorphism.
Suppose that we have a decomposition $W=\bigoplus_{i=1}^{l}W_{i}$ such that $A$ maps $W_{i}$ to $W_{i+1}$ (we put $W_{l+1}:=W_{1}$). 
Then we have
\[
\det(1-A\mid W)
=
\det(1-A^{l}\mid W_{l}).
\]
\end{lem}

\begin{proof}
By fixing a basis of $W_{i}$ for each $i$, we let $A_{i}$ be the representation matrix of $A|_{W_{i}}\colon W_{i}\rightarrow W_{i+1}$.
Then $1-A$ is represented by the matrix
\[
\begin{pmatrix}
I_{m}&&&&-A_{l}\\
-A_{1}&I_{m}&&&\\
&-A_{2}&\ddots&&\\
&&\ddots&\ddots&\\
&&&-A_{l-1}&I_{m}
\end{pmatrix},
\]
where $m$ denotes the dimension of $W_{1}$ and $I_{m}$ denotes the identity matrix of size $m$.
By noting that
\[
\begin{pmatrix}
I_{m}&&&&-A_{l}\\
-A_{1}&I_{m}&&&\\
&-A_{2}&\ddots&&\\
&&\ddots&\ddots&\\
&&&-A_{l-1}&I_{m}
\end{pmatrix}
\begin{pmatrix}
I_{m}&&&&A_{l}\\
&I_{m}&&&\\
&&\ddots&&\\
&&&\ddots&\\
&&&&I_{m}
\end{pmatrix}
\]
\[
=
\begin{pmatrix}
I_{m}&&&&\\
-A_{1}&I_{m}&&&-A_{1}A_{l}\\
&-A_{2}&\ddots&&\\
&&\ddots&\ddots&\\
&&&-A_{l-1}&I_{m}
\end{pmatrix},
\]
we have 
\[
\begin{vmatrix}
I_{m}&&&&-A_{l}\\
-A_{1}&I_{m}&&&\\
&-A_{2}&\ddots&&\\
&&\ddots&\ddots&\\
&&&-A_{l-1}&I_{m}
\end{vmatrix}
=
\begin{vmatrix}
I_{m}&&&-A_{1}A_{l}\\
-A_{2}&\ddots&&\\
&\ddots&\ddots&\\
&&-A_{l-1}&I_{m}
\end{vmatrix}.
\]
Similarly, we have
\[
\begin{vmatrix}
I_{m}&&&-A_{1}A_{l}\\
-A_{2}&\ddots&&\\
&\ddots&\ddots&\\
&&-A_{l-1}&I_{m}
\end{vmatrix}
=
\begin{vmatrix}
I_{m}&&&-A_{2}A_{1}A_{l}\\
-A_{3}&\ddots&&\\
&\ddots&\ddots&\\
&&-A_{l-1}&I_{m}
\end{vmatrix}.
\]
Repeating this procedure, eventually we get
\[
|1-A|
=
|1-A_{l-1}\cdots A_{1}A_{l}|.
\]
Since $A_{l-1}\cdots A_{1}A_{l}$ is nothing but the restriction of $A^{l}$ to $W_{l}$, we get the conclusion.
\end{proof}

We take the quasi-split group $\bfG^{\ast}$ over $F$ equipped with an inner twist $\psi$ and use notations in Section \ref{subsubsec:non-q-spl}.
Put $W^{\ast}:=W(\G^{\ast},\bfA^{\ast})$.
Recall that the action of $W^{\ast}\times W_F$ on $X^*(\hat{\bfT}^{\ast})$ induces that of $W^{\ast}\times\langle\Frob\rangle$ on $X^*(\hat{\bfT}^{\ast I_F})$. 
Let $\mcP(r^{I_F})$ denote the $(W^{\ast}\times\langle\Frob\rangle)$-stable subset consisting of all weights in $V^{I_{F}}$ with respect to $\hat{\bfT}^{\ast I_{F}}$, i.e., 
\[
\mcP(r^{I_{F}}):=\{\mu\in X^{\ast}(\hat{\bfT}^{\ast I_{F}})\mid \text{$\mu$ appears in $V^{I_{F}}$}\}.
\]
For each $\mu\in\mcP(r^{I_F})$, we write $[\mu]$ for the image of $\mu$ under the canonical quotient map from $\mcP(r^{I_{F}})$ onto $\mcP(r^{I_F})/\langle\Frob\rangle$, and define $l_{\mu}\in\Z_{>0}$ to be the cardinality of the $\langle\Frob\rangle$-orbit $\{\Frob^i(\mu)\mid i\in\Z\}$. We also define $N(\mu):=\sum_{i=0}^{l_{\mu}-1}\Frob^i(\mu)\in\Lambda_{T^{\ast}}=X^{\ast}(\hat{\bfT}^{\ast I_{F}})^{\Frob}$. 
We remark that the maps $l$ and $N$ are $(W^{\ast}\times\langle\Frob\rangle)$-equivalent, where $W^{\ast}\times\langle\Frob\rangle$ acts on $\Z_{>0}$ trivially. 
Hence we can regard $N$ and $l$ as maps defined on $\mcP(r^{I_F})/\langle\Frob\rangle$. 
Put $\mcI$ to be the image of the $W^{\ast}$-equivalent map
\[
N\times l\colon \mcP(r^{I_F})/\langle\Frob\rangle\to\Lambda_{T^{\ast}}\times\Z_{>0};\quad [\mu]\mapsto (N([\mu]),l_{[\mu]}).
\]
Define 
\[
\mcI^+:=\mcI\cap(\text{(the set of dominant elements in $\Lambda_{T^{\ast}}$)}\times\Z_{>0}).
\]
Then the canonical map $\mcI^+ \to \mcI/W^{\ast}$ is bijective.
Indeed, at least one element of each $W^{\ast}$-orbit in $\mcI$ belongs to $\mcI^+$ since the Weyl group acts on the set of Weyl chambers transitively. 
The uniqueness follows from, for example, \cite[Lemma 10.3.B]{MR499562}.

Put $\mcP_{\lambda,l}$ to be the inverse image of $(\lambda,l)\in\mcI$ under the map $\mcP(r^{I_F})\to\mcI$, i.e.,
\[
\mcP_{\lambda,l}
:=
\{\mu\in\mcP(r^{I_F}) \mid (N([\mu]),l_{[\mu]})=(\lambda,l)\}\subset\mcP(r^{I_F}).
\]
For each $\mu\in\mcP(r^{I_{F}})$, we put $V^{I_{F}}_{\mu}$ to be the $\mu$-eigenspace in $V^{I_{F}}$. 
%
For any $(\lambda,l)\in\mcI$, a complete set $\mcS$ of representatives of $\mcP_{\lambda,l}/\langle\Frob\rangle$, and $\eta\in\mcP(r^{I_F})$, we define 
\[
V^{I_F}_{\lambda,l}:=\bigoplus_{\mu\in\mcP_{\lambda,l}}V^{I_F}_{\mu},\quad 
V^{I_F}_{\mcS}:=\bigoplus_{\mu\in\mcS}V^{I_F}_{\mu},\quad 
V^{I_F}_{[\eta]}:=\bigoplus_{\mu\in[\eta]}V^{I_F}_{\mu}.
\]
Then we have 
\[
V^{I_F}_{\lambda,l}=\bigoplus_{i=0}^{l-1}r(\Frob)^i(V^{I_F}_{\mcS}),\ 
V^{I_F}_{[\eta]}=\bigoplus_{i=0}^{l_{[\eta]}-1}r(\Frob)^i(V^{I_F}_{\eta}),\ 
V^{I_F}=\bigoplus_{[\mu]\in\mcP(r^{I_F})/\lrFrob}V^{I_F}_{[\mu]}.
\] 
We remark that $r(\Frob)^{l}$ gives an automorphism on $V^{I_F}_{\mcS}$ and $r(\Frob)^{l_{[\eta]}}$ gives one on $V^{I_F}_{\eta}$. 
Since the multiset of eigenvalues of the automorphism $r(\Frob)^{l}$ (resp.\ $r(\Frob)^{l_{[\eta]}}$) do not depend on the choice of $\mcS$ (resp.\ $\eta$), we may write $C_{\lambda,l}$ (resp.\ $C_{[\eta]}$) for it. 
Note that the cardinality of $C_{\lambda,l}$ equals the dimension of $V^{I_F}_{\mcS}$.

Recall that  $A_{\mu}$ is a diagonalizable operator on $V_{\chi^{\ast}}^{J_{\mu}}$ given by $A_{\mu}(v_{w}^{J_{\mu},\vee})=q(w,\mu)^{-1}\cdot v_{w}^{J_{\mu},\vee}$ (see the paragraph before Corollary \ref{cor:OST4.3}).

\begin{thm}\label{thm:L}
Let $\pi$ be an $I$-spherical representation of $G$.
Let $\chi\in X^{\w}(M)$ be a weakly unramified character of $M$ such that $\pi$ is a subquotient of the normalized parabolic induction of $\chi$.
With the notations as in Corollary \ref{cor:OST4.3}, we have
\begin{align}\label{eq:main}
L_{\mathrm{ss}}(s,\pi,r)
=
\prod_{(\lambda,l)\in\mcI^+}
\prod_{c\in C_{\lambda,l}}
\det\bigl(1-q^{-ls}\cdot c\cdot A_{\lambda}\circ I_{\chi^{\ast}}(\mathbbm{1}_{\lambda})\,\big\vert\, V_{\chi^{\ast}}^{J_{\lambda}}\bigr)^{-1}.
\end{align}
\end{thm}

\begin{proof}
Recall that, by definition, we have
\[
L_{\mathrm{ss}}(s,\pi,r)
=
\det\bigl(1- q^{-s}\cdot r(s(\pi)) \,\big\vert\, V^{I_{F}}\bigr)^{-1}.
\]
Since $r(s(\pi))$ preserves $V^{I_{F}}_{[\mu]}$ for each $[\mu]\in\mcP(r^{I_{F}})/\lrFrob$, we have
\[
\det\bigl(1- q^{-s}\cdot r(s(\pi)) \,\big\vert\, V^{I_{F}}\bigr)
=
\prod_{[\mu]\in\mcP(r^{I_{F}})/\lrFrob}
\det\bigl(1- q^{-s}\cdot r(s(\pi)) \,\big\vert\, V^{I_{F}}_{[\mu]}\bigr).
\]
Note that, by fixing a representative $\mu$ of $[\mu]$, we have
\[
V^{I_{F}}_{[\mu]}
=V^{I_{F}}_{\mu}\oplus V^{I_{F}}_{\Frob(\mu)}\oplus\cdots\oplus V^{I_{F}}_{\Frob^{l_{[\mu]}-1}(\mu)}
\]
and $q^{-s}\cdot r(s(\pi))$ maps $V^{I_{F}}_{\Frob^{i}(\mu)}$ to $V^{I_{F}}_{\Frob^{i+1}(\mu)}$ for each $i$.
Hence, by Lemma \ref{lem:det}, we get
\[
\det\bigl(1- q^{-s}\cdot r(s(\pi)) \,\big\vert\, V^{I_{F}}_{[\mu]}\bigr)
=
\det\bigl(1- q^{-l_{[\mu]}s}\cdot r(s(\pi))^{l_{[\mu]}} \,\big\vert\, V^{I_{F}}_{\mu}\bigr).
\]

Recall that $s(\pi)=\tilde{t}_{\bfA^{\ast},\bfA}(\hat{\chi})\rtimes\Frob$.
Thus we have 
\[
r(s(\pi))^{l_{[\mu]}}
=
r(\mathcal{N}(\tilde{t}_{\bfA^{\ast},\bfA}(\hat{\chi})))\cdot r(\Frob)^{l_{[\mu]}},
\]
where we put $\mathcal{N}(\tilde{t}_{\bfA^{\ast},\bfA}(\hat{\chi})):=\prod_{i=0}^{l_{[\mu]}-1}\Frob^{i}(\tilde{t}_{\bfA^{\ast},\bfA}(\hat{\chi}))$.
As $\mathcal{N}(\tilde{t}_{\bfA^{\ast},\bfA}(\hat{\chi}))$ belongs to $\hat{\bfT}^{\ast I_{F}}$, $r(\mathcal{N}(\tilde{t}_{\bfA^{\ast},\bfA}(\hat{\chi})))$ acts on $V^{I_{F}}_{\mu}$ by a scalar multiplication $\mu(\mathcal{N}(\tilde{t}_{\bfA^{\ast},\bfA}(\hat{\chi})))$.
Note that we have
\[
\mu(\mathcal{N}(\tilde{t}_{\bfA^{\ast},\bfA}(\hat{\chi})))
=\prod_{i=0}^{l_{[\mu]}-1}\Frob^{i}(\mu)(\tilde{t}_{\bfA^{\ast},\bfA}(\hat{\chi}))
= N([\mu])(\tilde{t}_{\bfA^{\ast},\bfA}(\hat{\chi})).
\]
By the definition of $\chi^{\ast}$ (see Section \ref{subsubsec:non-q-spl}), we have $\lambda(\tilde{t}_{\bfA^{\ast},\bfA}(\hat{\chi}))= \chi^{\ast}\circ\kappa_{T^{\ast}}^{-1}(\lambda)$ for any $\lambda\in \Lambda_{T^{\ast}}=X^{\ast}(\hat{\bfT}^{\ast I_{F}})^{\Frob}$.
Hence we get
\[
N([\mu])(\tilde{t}_{\bfA^{\ast},\bfA}(\hat{\chi}))
=\chi^{\ast}\circ\kappa_{T^{\ast}}^{-1}(N([\mu])).
\]
From the above argument, we obtain 
\begin{align}\label{eq:L_ss}
L_{\mathrm{ss}}(s,\pi,r)=
\prod_{[\mu]\in\mcP(r^{I_F})/\langle\Frob\rangle}
\prod_{c\in C_{[\mu]}}
(1-q^{-l_{[\mu]}s}\cdot c\cdot \chi^{\ast}\circ\kappa_{T^{\ast}}^{-1}(N([\mu])))^{-1}.
\end{align}

We next rewrite the index set. Recall that we have a surjection $N\times l$ from $\mcP(r^{I_F})/\langle\Frob\rangle$ onto $\mcI$. By 
\[V^{I_F}_{\lambda,l}=\bigoplus_{\mu\in\mcP_{\lambda,l}}V^{I_F}_{\mu}=\bigoplus_{[\mu]\in\mcP_{\lambda,l}/\langle\Frob\rangle}V^{I_F}_{[\mu]},\]
we have an equality of multisets 
$C_{\lambda,l}=\bigsqcup_{[\mu]\in\mcP_{\lambda,l}/\langle\Frob\rangle}C_{[\mu]}$ for any $(\lambda,l)\in\mcI$. 
Therefore \eqref{eq:L_ss} equals 
\begin{align}
&\prod_{(\lambda,l)\in\mcI}
\prod_{[\mu]\in\mcP_{\lambda,l}/\langle\Frob\rangle}
\prod_{c\in C_{[\mu]}}
(1-q^{-ls}\cdot c\cdot \chi^{\ast}\circ\kappa_{T^{\ast}}^{-1}(\lambda))^{-1}\notag\\
&=\prod_{(\lambda,l)\in\mcI}
\prod_{c\in C_{\lambda,l}}
(1-q^{-ls}\cdot c\cdot \chi^{\ast}\circ\kappa_{T^{\ast}}^{-1}(\lambda))^{-1}.\label{eq:L_ss-1}
\end{align}

Note that the map $l\colon\mcP_{\lambda,l}/\langle\Frob\rangle\to\Z_{>0}$ is $W^{\ast}$-invariant, and that we have $C_{w(\lambda),l}=C_{\lambda,l}$ as multisets for any $w\in W^{\ast}$ since the action of $w$ induces a $\langle\Frob\rangle$-equivalent isomorphism $V^{I_F}_{(\lambda,l)}\cong V^{I_F}_{(w(\lambda),l)}$. 
Moreover, we have a bijection 
\[
W^{\ast}/W^{\ast}_{\lambda} \xrightarrow{1:1}W^{\ast}\cdot\lambda ;\quad w\mapsto w(\lambda).
\]
Therefore \eqref{eq:L_ss-1} equals 
\begin{align}
&\prod_{(\lambda,l)\in\mcI^+}
\prod_{w\in W^{\ast}/W^{\ast}_{\lambda}}
\prod_{c\in C_{w(\lambda),l}}
(1-q^{-ls}\cdot c\cdot \chi^{\ast}\circ\kappa_{T^{\ast}}^{-1}(w(\lambda)))^{-1}\notag\\
&=\prod_{(\lambda,l)\in\mcI^+}
\prod_{c\in C_{\lambda,l}}
\prod_{w\in W^{\ast}/W^{\ast}_{\lambda}}
(1-q^{-ls}\cdot c\cdot \chi^{\ast}\circ\kappa_{T^{\ast}}^{-1}(w(\lambda)))^{-1}.\label{eq:L_ss-2}
\end{align}
By applying Corollary \ref{cor:OST4.3} to $(\G^{\ast},\chi^{\ast},\lambda,c)$, the right-hand side of the equation \eqref{eq:L_ss-2} is written as
\begin{align*}
\prod_{(\lambda,l)\in\mcI^+}
\prod_{c\in C_{\lambda,l}}
\det\bigl(1-q^{-ls}\cdot c\cdot A_{\lambda}\circ I_{\chi^{\ast}}(\mathbbm{1}_{\lambda})\,\big\vert\, V_{\chi^{\ast}}^{J_{\lambda}}\bigr)^{-1}.
\end{align*}
Hence we get the assertion.
\end{proof}


\subsection{The case of induced representations}
In this section, we consider the case where $\G$ is unramified and the representation $r$ of ${}^{L}\G$ is induced from the one of $\hat{\G}$.
In this case, the expression of Theorem \ref{thm:L} can be slightly simplified as we see in the following.

Assume that $\G$ is unramified, i.e., $\G$ is quasi-split and splits over an unramified extension of $F$.
As we have $\G=\G^{\ast}$, we use the notation as in Section \ref{subsubsec:q-spl}; for example, $\bfT$ denotes the centralizer of $\bfA$ in $\G$.
Since the action of $I_F$ on $\hat{\G}$ is trivial, we obtain the action of $\langle\Frob\rangle$ on $\hat{\G}$.
There exists $l_0\in\Z_{>0}$ such that the action of $\Frob^{l_0}$ on $\hat{\G}$ is trivial.
Let $(r_0,V_0)$ be a finite-dimensional algebraic representation of $\hat{\G}$. 
Via the quotient homomorphism ${}^{L}\G\to\hat{\G}\rtimes\lrFrob\to\hat{\G}\rtimes(\Z/l_0\Z)$, we regard the induced representation 
\begin{align}\label{eq:induced}
(r=\Ind^{\hat{\G}\rtimes(\Z/l_0\Z)}_{\hat{\G}}r_0,\quad V=\bigoplus_{i\in\Z/l_0\Z}V_0)
\end{align}
as a representation of ${}^{L}\G$, where $\Frob$ permutes each component of $V$. 

Write $W=W(\G,\bfA)$.
We define a $(W\times\lrFrob)$-equivalent map $N_0$ by
\[
N_0\colon X^{\ast}(\hat{\bfT})\to \Lambda_{T}=X^{\ast}(\hat{\bfT})^{\Frob};\quad \mu\mapsto \sum_{i=0}^{l_0-1}\Frob^i(\mu),
\]
and put $\mcI_0$ (resp.\ $\mcI_0^+$) to be $N_0(\mcP(r_0))$ (resp.\ the set of dominant elements in $N_0(\mcP(r_0))$). 
Then the canonical map $\mcI_0^+\to\mcI_0/W$ is bijective, as discussed for $\mcI^+$ before Theorem \ref{thm:L}. 

For $\mu\in\mcP(r_0)$, we write $V_{0,\mu}$ for the $\mu$-eigenspace in $V_0$. 
For $\mu\in\mcP(r_0)$, we define 
\[
V^{\mu}:=\bigoplus_{i=0}^{l_0-1}r(\Frob)^i(V_{0,\mu}).
\]
Then we see 
$V=\bigoplus_{\mu\in\mcP(r_0)}V^{\mu}.$
For $\lambda\in\mcI_0$, we set 
$m_{0,\lambda}:=\sum_{\mu\in N_0^{-1}(\lambda)}\dim V_{0,\mu}$. 
Since $w(V_{0,\mu})=V_{0,w(\mu)}$, we have $m_{0,\lambda}=m_{0,w(\lambda)}$ for any $w\in W$.


\begin{thm}\label{thm:main-the-split}
Assume that $\G$ is unramified and $r$ is given by \eqref{eq:induced}. 
Let $\pi$ be an $I$-spherical representation of $G$.
Let $\chi\in X^{\w}(T)$ be a weakly unramified character of $T$ such that $\pi$ is a subquotient of the normalized parabolic induction of $\chi$.
Then we have
\[
L_{\mathrm{ss}}(s,\pi,r)
=
\prod_{\lambda\in\mcI_0^+}
\det\bigl(1-q^{-(l_0s+\langle\rho_{\bfB},\lambda\rangle)}I_{\chi}(\mathbbm{1}_{\lambda})\,\big\vert\, V_{\chi}^{J_{\lambda}}\bigr)^{-m_{0,\lambda}}.
\]
\end{thm}
\begin{proof}
The proof is similar to that of Theorem \ref{thm:L}. 
Since $r(s(\pi))$ preserves $V^{\mu}$ for each $\mu\in\mcP(r_0)$, we have
\[
L_{\mathrm{ss}}(s,\pi,r)
=
\det\bigl(1- q^{-s}\cdot r(s(\pi)) \,\big\vert\, V\bigr)^{-1}
=
\prod_{\mu\in\mcP(r_0)}
\det\bigl(1- q^{-s}\cdot r(s(\pi)) \,\big\vert\, V^{\mu}\bigr)^{-1}.
\]
Since $q^{-s}\cdot r(s(\pi))$ maps $r(\Frob)^i(V_{0,\mu})$ to $r(\Frob)^{i+1}(V_{0,\mu})$ for each $i$, Lemma \ref{lem:det} shows
\[
\det\bigl(1- q^{-s}\cdot r(s(\pi)) \,\big\vert\, V^{\mu}\bigr)
=
\det\bigl(1- q^{-l_0s}\cdot r(s(\pi))^{l_0} \,\big\vert\, V_{0,\mu}\bigr).
\]

By $s(\pi)=\hat{\chi}\rtimes\Frob$, we have 
\[
r(s(\pi))^{l_0}
=
r(\mathcal{N}_0(\hat{\chi}))\cdot r(\Frob)^{l_0}
=r(\mathcal{N}_0(\hat{\chi})),
\]
where we put $\mathcal{N}_0(\hat{\chi}):=\prod_{i=0}^{l_0-1}\Frob^i(\hat{\chi})\in \hat{\bfT}^{\Frob}$.
For $\mu\in\mcP(r_0)$, we have
\[
\mu(\mathcal{N}_0(\hat{\chi}))
=\prod_{i=0}^{l_0-1}\Frob^{i}(\mu)(\hat{\chi})
= N_0(\mu)(\hat{\chi})
=\chi\circ\kappa_{T}^{-1}(N_0(\mu)).
\]
From the above argument, we obtain 
\begin{align*}
L_{\mathrm{ss}}(s,\pi,r)&=
\prod_{\mu\in\mcP(r_0)}
(1-q^{-l_0s}\cdot\chi\circ\kappa_{T}^{-1}(N_0(\mu)))^{-\dim V_{0,\mu}}\\
&=\prod_{\lambda\in\mcI_0}
(1-q^{-l_0s}\cdot\chi\circ\kappa_{T}^{-1}(\lambda))^{-m_{0,\lambda}}\\
&=\prod_{\lambda\in\mcI_0^+}
\prod_{w\in W/W_{\lambda}}
(1-q^{-l_0s}\cdot \chi\circ\kappa_{T}^{-1}(w(\lambda)))^{-m_{0,\lambda}}\\
&=\prod_{\lambda\in\mcI_0^+}
\det\bigl(1-q^{-(l_0s+\langle\rho_{\bfB},\lambda\rangle)} I_{\chi}(\mathbbm{1}_{\lambda})\,\big\vert\, V_{\chi}^{J_{\lambda}}\bigr)^{-m_{0,\lambda}},
\end{align*}
where we used $m_{0,\lambda}=m_{0,w(\lambda)}$ at the third equality, and Corollary \ref{cor:OST4.3} and Remark \ref{rem:q-spl} at the last equality. 
Hence we get the assertion.
\end{proof}
\begin{rem}\label{rem:main-the-split}
When $\G$ is split and the finite-dimensional continuous representation $r$ of ${}^{L}\G=\hat{\G}\times W_{F}$ is trivial on $W_{F}$, we can apply Theorem \ref{thm:main-the-split} to $l_0=1$ and $r_0=r$. 
In this case, there is no difference between $\mcP(r_0)$ and $\mcI$. 
Hence the formula in Theorem \ref{thm:main-the-split} is simplified as follows:
\[
L_{\mathrm{ss}}(s,\pi,r)
=
\prod_{\mu\in\mcP^+(r)}
\det\bigl(1-q^{-(s+\langle\rho_{\bfB},\mu\rangle)}I_{\chi}(\mathbbm{1}_{\mu})\,\big\vert\, V_{\chi}^{J_{\mu}}\bigr)^{-m_{\mu}}.
\]
Here $m_{\mu}$ denotes the multiplicity of the weight $\mu$ in $r$.
\end{rem}
%
%

\subsection{The case of quasi-minuscule representations}\label{subsec:q-minuscule}
In this section, we focus on the case where $\G$ is split.
Let us investigate simpler cases where the right-hand side of the formula of Theorem \ref{thm:L} consists of essentially one nontrivial factor.

\begin{defn}\label{def:minuscule}
We say that an irreducible finite-dimensional representation of $\hat{\G}$ is \textit{minuscule} (resp.\ \textit{quasi-minuscule}) if 
the Weyl group $W$ acts transitively on the set of weights (resp.\ the set of weights not fixed by $W$). 
\end{defn}

\begin{rem}\label{rem:q-minuscule}
Let $r$ be an irreducible representation of $\hat{\G}$ with highest weight $\mu$.
Recall that the map $\mathcal{P}^{+}(r)\rightarrow\mathcal{P}(r)/W$ is bijective as discussed for $\mcI^+$ before Theorem \ref{thm:L}.
Hence, we have $\#\mcP^{+}(r)=1$ if $r$ is minuscule.
Moreover we can check that if $r$ is quasi-minuscule and not minuscule, then we have $\#\mcP^{+}(r)=2$ as follows:
Let us suppose that $\mu_{1}$ and $\mu_{2}$ are dominant weights of $r$ fixed by $W$.
Then it suffices to show that $\mu_{1}=\mu_{2}$, which is equivalent to  
\begin{align}\label{eq:q-minuscule}
\langle\alpha,\mu_{1}\rangle
=
\langle\alpha,\mu_{2}\rangle
\end{align}
 for any $\alpha\in X_{\ast}(\mcT)$.
Let $\hat{\G}_{\der}$ denote the derived group of $\hat{\G}$.
As we have $\mcT=\mcT_{\der}Z_{\hat{\G}}$, where $\mcT_{\der} := \mcT\cap\hat{\G}_{\der}$ and $Z_{\hat{\G}}$ is the center of $\hat{\G}$, it is enough to check the equality (\ref{eq:q-minuscule}) for every $\alpha\in X_{\ast}(\mcT_{\der})$ and $\alpha\in X_{\ast}(Z_{\hat{\G}})$.
We first check the former case.
For every coroot $\alpha\in X_{\ast}(\mcT_{\der})$, since $\mu_{1}$ is $W$-invariant, we have
\[
\langle \alpha,\mu_{1}\rangle
=
\langle \alpha,s_{\alpha}^{-1}\mu_{1}\rangle
=
\langle s_{\alpha}\alpha,\mu_{1}\rangle
=
-\langle \alpha,\mu_{1}\rangle,
\]
where $s_{\alpha}$ is the reflection with respect to $\alpha$.
Thus we have $\langle \alpha,\mu_{1}\rangle=0$.
As the space $X_{\ast}(\mcT_{\der})_{\R}$ is spanned by the set of coroots of $\mcT_{\der}$ in $\hat{\G}_{\der}$, the equality $\langle \alpha,\mu_{1}\rangle=0$ holds for any element $\alpha$ of $X_{\ast}(\mcT_{\der})$.
Similarly, we have $\langle \alpha,\mu_{2}\rangle=0$ for any $\alpha\in X_{\ast}(\mcT_{\der})$.
Second, as the representation $r$ is irreducible, it has a central character by Schur's lemma.
In other words, all weights of $r$ has the same value on the center $Z_{\hat{\G}}$.
Thus the equality (\ref{eq:q-minuscule}) holds for any $\alpha\in X_{\ast}(Z_{\hat{\G}})$.
\end{rem}

\begin{cor}\label{cor:L2}
Let  $r$ be a quasi-minuscule representation of the Langlands dual group $\hat{\G}$ with highest weight $\mu$. 
\begin{enumerate}
\item Assume that $r$ is minuscule. Then we have 
\begin{align*}
L_{\mathrm{ss}}(s,\pi,r) = \det\bigl(1- q^{-(s+\langle\rho_{\bfB}, \mu\rangle)} I_{\chi}(\mathbbm{1}_{\mu}) \,\big\vert\, V_{\chi}^{J_{\mu}}\bigr)^{-1}.  \end{align*}
\item Assume that $r$ is not minuscule. Then the set $\mcP^{+}(r)$ of dominant weights in $r$ consists of $\mu$ and a dominant weight $\mu'$ fixed by $W$, and we have 
\[
L_{\mathrm{ss}}(s,\pi,r) = \bigl(1-q^{-s}\chi\circ\kappa_{T}^{-1}(\mu')\bigr)^{-m_{\mu'}}
\det\bigl(1 - q^{-(s+\langle\rho_{\bfB}, \mu\rangle)} I_{\chi}(\mathbbm{1}_{\mu}) \,\big\vert\, V_{\chi}^{J_{\mu}}\bigr)^{-1}. 
\]
\end{enumerate}
\end{cor}

\begin{proof}
Assertion (1) is a direct consequence of Theorem~\ref{thm:L} and Remark~\ref{rem:q-minuscule} (recall that the multiplicity of the highest weight of $r$ is one). 

Let us show assertion (2). 
Again by Theorem \ref{thm:L} and Remark \ref{rem:q-minuscule}, we get
\[
L_{\mathrm{ss}}(s,\pi,r)
=
\det\bigl(1 - q^{-(s+\langle\rho_{\bfB},\mu'\rangle)} I_{\chi}(\mathbbm{1}_{\mu'}) \,\big\vert\,  V_{\chi}^{J_{\mu'}}\bigr)
^{-m_{\mu'}}
\det\bigl(1 - q^{-(s+\langle\rho_{\bfB}, \mu\rangle)} I_{\chi}(\mathbbm{1}_{\mu}) \,\big\vert\, V_{\chi}^{J_{\mu}}\bigr)^{-1}. 
\]
Since $\mu'$ is a $W$-invariant weight, by the same argument as in Remark \ref{rem:q-minuscule}, we have $\langle\alpha, \mu'\rangle=0$ for any $\alpha\in\Phi$. 
Hence $W_{\mu'} = W$ and $\langle\rho_{\bfB},\mu'\rangle$ vanishes. 
Then Corollary \ref{cor:OST4.3} shows that 
\[
\det\bigl(1 - q^{-(s+\langle\rho_{\bfB},\mu'\rangle)} I_{\chi}(\mathbbm{1}_{\mu'}) \,\big\vert\,  V_{\chi}^{J_{\mu'}}\bigr)
=1-q^{-s}\chi\circ\kappa_{T}^{-1}(\mu').
\]
\end{proof}

\begin{rem}\label{rem:q-min}
Assume that $\G$ is a split connected simple group with trivial center. 
In Table \ref{Table} in the end of this paper, we list all isomorphism classes of nontrivial quasi-minuscule representations of $\hat{\G}$ (cf. \cite[221 page, Fig.\ A.\ 1]{MR2388163}). 
Note that since we are assuming that $\G$ is simple, a nontrivial quasi-minuscule representation $r$ is minuscule exactly when $m_0=0$. 
We also remark that the Langlands dual group of the adjoint group  is simply-connected, and that there is a natural one-to-one correspondence between finite-dimensional representations of a connected simply-connected simple complex Lie group and finite-dimensional representations of its Lie algebra. 

For a split connected simple group $\G'$ whose center is not necessarily trivial, we remark that quasi-minuscule representations of $\widehat{\G'}$ are exactly those of the Langlands dual group $\widehat{\G'/\mathbf{Z}'}$ of $\G'/\mathbf{Z}'$ factoring $\widehat{\G'}$, where $\mathbf{Z}'$ denotes the center of $\G'$. 

Let $\Delta_{\mathcal{B}}$ be the set of simple (with respect to the fixed Borel subgroup $\mcB$) roots of $\mcT$ in $\hat{\G}$.
Let $I$ denote the subset of $\Delta_{\mathcal{B}}$ consisting of the boxed simple roots in the Dynkin diagram on Table \ref{Table}. Then the highest weight $\mu$ of a quasi-minuscule representation $r$ of $\hat{\G}$ is characterized as the unique character satisfying
\begin{align*}
\langle\alpha, \mu\rangle=
\begin{cases}
1&\text{ if $\alpha\in I$,}\\
0&\text{ otherwise.}
\end{cases}
\end{align*}
\end{rem}

\section{Examples in the unramified case}
In this section, we present some examples in the cases where $\G$ is $\GL_{n}, \Res_{E/F}\GL_n$ and $\GSp_{2n}$.

\subsection{The case of $\GL_{n}$}\label{subsec:GL}
Let $\G=\GL_n$ $(n\ge 2)$. We take the split maximal torus $\bfT$ consisting of diagonal matrices, and the Borel subgroup $\bfB$ consisting of upper-triangular matrices. 
We take $\Z$-bases for the character group $X^{\ast}(\bfT)$ and the cocharacter group $X_{\ast}(\bfT)$ to be $\{e_i\}_{i=1}^n$ and $\{e_i^{\vee}\}_{i=1}^n$, where $e_i$ and $e_i^{\vee}$ are given by 
\begin{align*}
e_i\bigl(\diag(t_1,\ldots, t_n)\bigr)=t_i
\quad\text{and}\quad
e_i^{\vee}(s)=\diag(\underbrace{1,\ldots,1}_{i-1},s,\underbrace{1,\ldots,1}_{n-i})
\end{align*}
for $t_1,\ldots, t_n, s\in \Gm$. 
Then we see 
\begin{align*}
\Phi&=\{\pm(e_i-e_j)\mid 1\le i< j\le n\}, \quad
\Delta_{\mathbf{B}}=\{e_1-e_2,\ldots,e_{n-1}-e_n\}, 
\\
\Phi^{\vee}&=\{\pm(e_i^{\vee}-e_j^{\vee})\mid 1\le i< j\le n\}, \quad 
\Delta_{\mathbf{B}}^{\vee}=\{e_1^{\vee}-e_2^{\vee},\ldots,e_{n-1}^{\vee}-e_n^{\vee}\}.
\end{align*} 
From these expressions, it follows that the Langlands dual group $\widehat{\GL_n}$ is $\GL_n(\C)$. 
Since the set of positive roots is given by $\{e_i-e_j\mid 1\le i<j\le n\}$, we have 
\[
\rho_{\bfB}=\sum_{i=1}^n\frac{n+1-2i}{2} e_i.
\] 

For $i\neq j$, we define homomorphisms $x_{e_i-e_j}\colon \Ga\to \bfU_{\alpha}\subset\G$ by $x_{e_i-e_j}(a):=I_n+aE_{i,j}$ for each $a\in\Ga$. 
Here $I_n$ denotes the $n \times n$ unit matrix and $E_{i,j}$ denotes the $n \times n$ matrix where the $(i,j)$-entry is 1 and the other entries are 0. 
Then $\{x_{\alpha}\colon \Ga\to \bfU_{\alpha}\}_{\alpha\in\Phi}$ forms a Chevalley basis of $\G$.
We take the special point $\bfo\in\mcB(\GL_{n},F)$ corresponding to this Chevalley basis.
In other words, as explained in Remark \ref{rem:Chevalley}, for $\alpha\in\Phi$, the filtration $\{U_{\alpha,r}\}_{r\in\R}$ of the root subgroup $U_{\alpha}=\bfU_{\alpha}(F)$ is given by $U_{\alpha,r}=x_{\alpha}(\{a\in F \mid \val_{F}(a)\geq r\})$.
The corresponding special parahoric subgroup $K$ is simply given by $\GL_{n}(\mcO)$.

\subsubsection{Exterior $L$-functions}
Consider the $l$-th exterior power $r = \wedge^l$ of the standard representation 
of $\hat{\G}=\GL_n(\C)$ for $1\le l\le n-1$. 
It has the unique dominant weight $\mu=\sum_{i=1}^le_i^{\vee}$. 
Hence $\wedge^l$ is minuscule. 
We have 
$\langle\rho_{\bfB},\mu\rangle=\sum_{i=1}^l(n+1-2i)/2=l(n-l)/2$. 
Therefore Corollary \ref{cor:L2} gives 
\begin{align*}
L(s,\pi,\wedge^l)=\det\bigl(1-q^{-(s+l(n-l)/2)}I_{\chi}(\mathbbm{1}_{\mu})\mid V_{\chi}^{J_{\mu}}\bigr)^{-1},
\end{align*}
where we have 
\begin{align*}
J_{\mu}=\left\{
\begin{pmatrix}
A&B\\
C&D
\end{pmatrix} \ \middle| \  A\in \GL_l(\mcO), B\in M_{l,n-l}(\mcO), C\in M_{n-l,l}(\mathfrak{p}), D\in \GL_{n-l}(\mcO)\right\}
\end{align*}
and the element $\ul{\mu}\in T/T_{1}$ is represented by $\diag(\underbrace{\varpi,\ldots, \varpi}_{l},\underbrace{1,\ldots,1}_{n-l})$.

Note that when $n=2$ and $l=1$, this formula recovers the classical formula for $L(s,\pi,\Std)$ explained in Section \ref{sec:Intro}.

\subsubsection{Adjoint $L$-function}
Consider the adjoint representation $r = {\rm Ad}$. 
Its highest weight is given by $\mu=e_1^{\vee}-e_n^{\vee}$. 
The other dominant weight is $\mu'=0$, whose multiplicity is $n$. 
We remark that the adjoint representation is the direct sum of a quasi-minuscule representation and the trivial representation. 
We have 
$\langle\rho_{\bfB},\mu\rangle=(n-1)/2-(-n+1)/2=n-1$. 
Therefore Corollary \ref{cor:L2} gives 
\begin{align*}
L(s,\pi,{\rm Ad})=(1-q^{-s})^{-n}\det\bigl(1-q^{-(s+n-1)}I_{\chi}(\mathbbm{1}_{\mu})\mid V_{\chi}^{J_{\mu}}\bigr)^{-1},
\end{align*}
where we have 
\begin{align*}
J_{\mu}=
\Set{
\begin{pmatrix}
a&b&c\\
{}^t\!d&E&{}^t\!f\\
g&h&i
\end{pmatrix} | \begin{array}{l}
a, i\in \mcO^{\times}, b,f\in M_{1,n-2}(\mcO), c\in\mcO,\\
d,h\in M_{1,n-2}(\mathfrak{p}), E\in \GL_{n-2}(\mcO), g\in\mathfrak{p}
\end{array}}
\end{align*}
and the element $\ul{\mu}\in T/T_{1}$ is represented by $\diag(\varpi, \underbrace{1,\ldots, 1}_{n-2},\varpi^{-1})$.

\subsubsection{Symmetric $L$-functions} 
Consider the $l$-th symmetric power $r = {\rm Sym}^{l}$ of the standard representation of $\hat{\G}=\GL_n(\C)$ for non-negative integer $l$.  
We can check the irreducibility of ${\rm Sym}^{l}$ by the Weyl dimension formula, for example. 
Let 
\[
T^+_l:=\left\{\mathbf{a}=(a_1,\ldots,a_n) \in \Z^{n} \ \middle| \ a_1\ge a_2\ge\cdots\ge a_n\ge 0, \sum_{i=1}^na_i=l\right\}
\]
and
\[
\mu_{\mathbf{a}}:=\sum_{i=1}^na_ie_n^{\vee}.
\] 
Given $\mathbf{a}\in T^+_l$, define $m\ge 1$ and $r_1(\mathbf{a}),\ldots,r_m(\mathbf{a})$ so that $r_1(\mathbf{a})+\cdots+r_m(\mathbf{a})=n$ and 
\[
a_{1}=a_{r_1(\mathbf{a})}>a_{r_1(\mathbf{a})+1}=a_{r_1(\mathbf{a})+r_2(\mathbf{a})}>\cdots >a_{
r_1(\mathbf{a})+\cdots+r_{m-1}(\mathbf{a})
+1}=a_{r_1(\mathbf{a})+\cdots+r_m(\mathbf{a})}. 
\]

The set $\mathcal{P}^{+}({\rm Sym}^{l})$ of dominant weights is given by $\{\mu_{\mathbf{a}}\mid \mathbf{a}\in T^+_l\}$, and their multiplicities are one. 
Therefore Theorem \ref{thm:L} gives 
\begin{align*}
L(s,\pi,{\rm Sym}^{l})
= \prod_{\mathbf{a} \in T^{+}_l}
\det\bigl(1-q^{-\left(s+
\sum_{i=1}^na_i(n+1-2i)/2\right)
}I_{\chi}(\mathbbm{1}_{\mu_{\mathbf{a}}}) \,\big\vert\, V_{\chi}^{J_{\mu_{\mathbf{a}}}}\bigr)^{-1},
\end{align*}
where we have 
\begin{align*}
J_{\mu_{\mathbf{a}}}=\Set{\begin{pmatrix}A_{11}&A_{12}&\cdots&A_{1m}\\A_{21}&A_{22}&\cdots&A_{2m}\\\vdots&\vdots&\ddots&\vdots\\A_{m1}&A_{m2}&\cdots&A_{mm}\end{pmatrix} | 
\begin{array}{l}
A_{ii}\in \GL_{r_i(\mathbf{a})}(\mcO)\text{ for $1\le i\le m$},\\
A_{ij}\in M_{r_i(\mathbf{a}),  r_j(\mathbf{a})}(\mcO)\\ \text{ and }A_{ji}\in M_{r_j(\mathbf{a}), r_i(\mathbf{a})}(\mathfrak{p})\\\text{ for $1\le i< j\le m$}
\end{array}}
\end{align*}
and the element $\ul{\mu_{\mathbf{a}}}\in T/T_{1}$ is represented by $\diag(\varpi^{a_1},\ldots,\varpi^{a_n})$.


\subsection{The case of $\Res_{E/F}\GL_n$}\label{subsec:ResGL}
Let $E$ be the unramified quadratic extension of $F$. 
Let us take $\G$ to be the Weil restriction $\Res_{E/F}\GL_{n,E}$ of the general linear group $\GL_{n,E}$ over $E$ with respect to $E/F$ (note that $\G$ is unramified).
We take $\bfA$ to be the maximal $F$-split torus of $\G$ whose $F$-valued points consists of diagonal matrices of $\GL_n(F)$, $\bfT$ to be the $F$-rational $E$-split torus of $\G$ consisting of diagonal matrices, and $\bfB$ to be the $F$-rational Borel subgroup of $\G$ consisting of upper triangular matrices. 
The Langlands dual group $\hat{\G}$ of $\G$ is given by $\GL_n(\C)\times\GL_n(\C)$ and the Weil group $W_{F}$ acts on $\hat{\G}$ by
\[
\sigma(g_{1},g_{2})
=
\begin{cases}
(g_{1},g_{2}) & \text{if $\sigma\in I_{F}$}, \\
(g_{2},g_{1}) & \text{if $\sigma=\Frob$}.
\end{cases}
\]
Hence ${}^L\G$ has
\[
{}^L\bar{\G}:=\hat{\G}\rtimes\Gal(E/F)=(\GL_n(\C)\times\GL_n(\C))\rtimes\Z/2\Z
\]
as its quotient.

We write $\bfT_n$ for the $E$-split maximal torus of $\GL_{n,E}$ in Section \ref{subsec:GL}, and use notations $e_i, e_i^{\vee}$ therein. 
Then we have $X^{\ast}(\bfT)=X^{\ast}(\bfT_n)\oplus X^{\ast}(\bfT_n)$. 
Since the set of positive roots is given by $\{(e_i-e_j,0),(0,e_i-e_j)\mid 1\le i<j\le n\}$, we have 
\[
\rho_{\bfB}=\left(\sum_{i=1}^n\frac{n+1-2i}{2} e_i,\sum_{i=1}^n\frac{n+1-2i}{2} e_i\right).
\] 

We take a special point $\bfo\in\mcB(\G,F)$ in the apartment attached to $\bfA$ so that the corresponding special parahoric subgroup $K$ is simply given by $\GL_{n}(\mcO_E)$, where $\mcO_E$ denotes the ring of integers of $E$.

\subsubsection{Asai $L$-function}
Let $\epsilon\in\{\pm 1\}$. 
Consider the Asai representation $\As^{\epsilon}$ of ${}^{L}\G$, which is characterized by the following properties:
\begin{itemize}
\item
The restriction of $\As^{\epsilon}$ to $\hat{\G}=\GL_n(\C)\times\GL_n(\C)$ is given by the tensor product $\C^n\boxtimes \C^n$ of the standard representations of $\GL_n(\C)$. 
\item
The representation $\As^{\epsilon}$ factors through ${}^L\bar{\G}$, and 
$\As^{\epsilon}(\Frob)(v\otimes w)=\epsilon\cdot w\otimes v$ for any $v,w\in\C^n$. 
\end{itemize}
We see
\[
\mcP(\As^{\epsilon})=\{(e_i^{\vee},e_j^{\vee})\mid 1\le i,j\le n\},\quad
\mcI^+=\{(\lambda_1,1), (\lambda_2,2)\},
\] 
where we put $\lambda_1:=(e_1^{\vee},e_1^{\vee})$, $\lambda_2:=(e_1^{\vee}+e_2^{\vee},e_1^{\vee}+e_2^{\vee})\in\Lambda_{T}$. 
We have $\langle\rho_{\bfB},\lambda_1\rangle=n-1$ and $\langle\rho_{\bfB},\lambda_2\rangle=(n-1)+(n-3)=2(n-2)$.
Moreover, 
$C_{(\lambda_1,1)}=\{\epsilon\}$, $C_{(\lambda_2,2)}=\{1\}$
as multisets. Then Theorem \ref{thm:L} gives 
\[
L(s,\pi,\As^{\epsilon})=\det(1-q^{-(s+n-1)}\epsilon\cdot I_{\chi}(\mathbbm{1}_{\lambda_1}) \,\big\vert\, V_{\chi}^{J_{\lambda_1}}\bigr)^{-1}
\det(1-q^{-2(s+n-2)}I_{\chi}(\mathbbm{1}_{\lambda_2}) \,\big\vert\, V_{\chi}^{J_{\lambda_2}}\bigr)^{-1}.
\]

As representations of ${}^L\bar{\G}$, we have an isomorphism 
\[
\As^+\oplus\As^-\cong\Ind^{{}^L\bar{\G}}_{\hat{\G}}(\C^n\boxtimes\C^n).
\] 
Let us apply Theorem \ref{thm:main-the-split} to $l_0=2$ and $r_0=\C^n\boxtimes\C^n$. 
We have $\mcP(r_0)=\mcP(\As^{\epsilon})$, $\mcI_0^+=\{2\lambda_1,\lambda_2\}$ and $m_{0,2\lambda_1}=1$, $m_{0,\lambda_2}=2$. 
Therefore, we obtain
\[
L(s,\pi,\As^+\oplus\As^-)=\det(1-q^{-2(s+n-1)}I_{\chi}(\mathbbm{1}_{2\lambda_1}) \,\big\vert\, V_{\chi}^{J_{2\lambda_1}}\bigr)^{-1}
\det(1-q^{-2(s+n-2)}I_{\chi}(\mathbbm{1}_{\lambda_2}) \,\big\vert\, V_{\chi}^{J_{\lambda_2}}\bigr)^{-2}.
\]
In the above expressions, we have 
\begin{align*}
J_{\lambda_1}&=J_{2\lambda_1}=\Set{
\begin{pmatrix}
a&b\\
c&D
\end{pmatrix} | \begin{array}{l}
a\in\mcO_E^{\times}, b\in M_{1,n-1}(\mcO_E),\\
c\in M_{n-1,1}(\mathfrak{p}_E), D\in \GL_{n-1}(\mcO_E)
\end{array}},\\
J_{\lambda_2}&=\Set{
\begin{pmatrix}
A&B\\
C&D
\end{pmatrix} | \begin{array}{l}
A\in \GL_2(\mcO_E), B\in M_{2,n-2}(\mcO_E),\\
C\in M_{n-2,2}(\mathfrak{p}_E), D\in \GL_{n-2}(\mcO_E)
\end{array}}
\end{align*}
($\mathfrak{p}_E$ denotes the maximal ideal of $\mcO_E$) and the elements $\ul{\lambda_{1}}$ and $\ul{\lambda_{2}}$ of $T/T_{1}$ are represented by 
\[
\diag(\varpi,\underbrace{1,\ldots,1}_{n-1})
\quad\text{and}\quad
\diag(\varpi,\varpi,\underbrace{1,\ldots,1}_{n-2}),
\]
respectively.

\subsection{The case of $\GSp_{2n}$}\label{subsec:GSp}
Let us take $\G$ to be 
\[\GSp_{2n}=
\Set{g\in\GL_{2n} | {}^{t}\!g\begin{pmatrix}&-J_n\\J_n&\end{pmatrix}g=x\begin{pmatrix}&-J_n\\J_n&\end{pmatrix}\text{ for some $x\in \Gm$}}\]
for $n\ge 1$. 
Here $J_{n}$ denotes the anti-diagonal $n\times n$ matrix whose all anti-diagonal entries are $1$.  
We take the split maximal torus $\bfT$ consisting of the diagonal matrices and the Borel subgroup $\bfB$ consisting of the upper-triangular matrices. 
We take $\Z$-bases for the character group $X^{\ast}(\bfT)$ and the cocharacter group $X_{\ast}(\bfT)$ to be $\{e_i\}_{i=0}^{n}$ and $\{e_i^{\vee}\}_{i=0}^n$, where $e_i$ and $e_i^{\vee}$ are given by 
\begin{align*}
&e_i(\diag(t_0t_1,\ldots,t_0t_n, t_n^{-1},\ldots,t_1^{-1}))=t_i,\\
&e_i^{\vee}(s)=
\begin{cases}
\diag(\underbrace{1,\ldots,1}_{i-1},s,\underbrace{1,\ldots,1}_{2n-2i},s^{-1},\underbrace{1,\ldots,1}_{i-1}) & \text{ if $1\le i\le n$,}
\\
\diag(\underbrace{s,\ldots,s}_{n},\underbrace{1,\dots,1}_{n}) & \text{ if $i=0$,}
\end{cases}
\end{align*}
for $t_0, \ldots, t_n, s\in \Gm$. 
Then we see 
\begin{align*}
\Phi &=\{\pm(e_i-e_j)\mid 1\le i< j\le n\}\cup \{\pm(e_i+e_j+e_0)\mid 1\le i\le j\le n\}, 
\\ 
\Delta_{\mathbf{B}}&=\{e_1-e_2,\ldots,e_{n-1}-e_n, 2e_n+e_0\}, 
\\
\Phi^{\vee} &=\{\pm e_i^{\vee}\pm e_j^{\vee} \mid 1\le i<j\le n\}\cup\{\pm e_i^{\vee}\mid 1\le i\le n\}, 
\\
\Delta_{\mathbf{B}}^{\vee}&=\{e_1^{\vee}-e_2^{\vee},\ldots,e_{n-1}^{\vee}-e_n^{\vee}, e_n^{\vee}\}.
\end{align*} 
This root datum is the dual root datum of $\GSpin_{2n+1}$ given in \cite[Proposition 2.4]{MR1913914}. 
Hence the Langlands dual group $\widehat{\GSp_{2n}}$ is $\GSpin_{2n+1}(\C)$. 

Here we fix an isomorphism between root data $\Psi(\GSp_{2n})^{\vee}$ and $\Psi(\GSpin_{2n+1})$ in the following way.
Let $\mathrm{sim}_{\GSpin_{2n+1}}$ be the similitude character of $\GSpin_{2n+1}(\C)$ defined by composing the covering map $\GSpin_{2n+1}(\C)\twoheadrightarrow\mathrm{GSO}_{2n+1}(\C)$ with that $\mathrm{sim}_{\GSO_{2n+1}}$ of $\GSO_{2n+1}(\C)$, which is given by
\[
\GSO_{2n+1}(\C)
=
\{g\in\GL_{2n+1}(\C)
\mid
{}^{t}gJ_{2n+1}g=\mathrm{sim}_{\GSO_{2n+1}}(g)J_{2n+1}
\}^{0}.
\]
Then we choose a unique isomorphism between root data $\Psi(\GSp_{2n})^{\vee}$ and $\Psi(\GSpin_{2n+1})$ such that $2e_0^{\vee}-\sum_{i=1}^ne_i^{\vee}$ corresponds to $\mathrm{sim}_{\GSpin_{2n+1}}$.



Since the set of positive roots $\Phi^+$ is given by $\{e_i-e_j\mid 1\le i< j\le n\}\cup \{e_i+e_j+e_0\mid 1\le i\le j\le n\}$, we have $\rho_{\bfB}=\sum_{i=1}^n (n+1-i)e_i+n(n+1)/4 \cdot e_0$. 

Similarly to the case of $\GL_{n}$, we choose a special point $\bfo$ of the Bruhat--Tits building $\mcB(\GSp_{2n},F)$ associated with the following Chevalley basis $\{x_{\alpha}\colon\Ga\rightarrow\bfU_{\alpha}\}_{\alpha\in\Phi}$:
For $\alpha\in\Phi$, we define a homomorphism $x_{\alpha}\colon\Ga\rightarrow\bfU_{\alpha}\subset\G$ by
\begin{align*}
x_{e_i-e_j}(a)&=I_{2n}+a(E_{i,j}-E_{2n+1-j,2n+1-i})\text{\quad ($1\le i<j\le n$)},\\ 
x_{e_i+e_j+e_0}(a)&=I_{2n}+a(E_{i,2n+1-j}+E_{j,2n+1-i})\text{\quad ($1\le i<j\le n$)},\\
x_{2e_i+e_0}(a)&=I_{2n}+aE_{i,2n+1-i}\text{\quad ($1\le i\le n$)},\\
x_{-\alpha}(a)&={}^t\!x_{\alpha}(a)\text{\quad ($\alpha\in\Phi^+$)}.
\end{align*}
Then the filtration $\{U_{\alpha,r}\}_{r\in\R}$ of the root subgroup $U_{\alpha}=\bfU_{\alpha}(F)$ is given by $U_{\alpha,r}=x_{\alpha}(\{a\in F \mid \val_{F}(a)\geq r\})$.
The corresponding special parahoric subgroup $K$ is simply given by $\GSp_{2n}(\mcO)$.

\subsubsection{Spin $L$-function}
Consider the spin representation $r = {\rm Spin}$ of $\hat{\G}=\GSpin_{2n+1}$. 
By checking weights in the spin representation of the derived group $\Spin_{2n+1}$ (see \cite[Chapter V.9.27]{MR1920389}), 
we see that the spin representation of $\GSpin_{2n+1}$ is minuscule and that the highest weight $\mu\in X_{\ast}(\bfT)$ satisfies $\langle e_i-e_{i+1},\mu\rangle=0$ for $1\le i\le n-1$ and $\langle 2e_n+e_0,\mu\rangle=1$. 
Since the restriction of the similitude character of $\GSpin_{2n+1}$ to its center is the twice of the character defined by the spin representation, we have $\langle e_0,\mu\rangle=\langle e_0, 2e_0^{\vee}-\sum_{i=1}^ne_i^{\vee}\rangle/2= 1$. 
Therefore we obtain $\mu=e_0^{\vee}$. 
We have $\langle\rho_{\bfB},\mu\rangle=n(n+1)/4$. 
Therefore Corollary \ref{cor:L2} gives 
\begin{align*}
L(s,\pi,{\rm Spin})=\det\bigl(1-q^{-(s+n(n+1)/4)}I_{\chi}(\mathbbm{1}_{\mu})\mid V_{\chi}^{J_{\mu}}\bigr)^{-1},
\end{align*}
where we have
\begin{align*}
J_{\mu}=\Set{\begin{pmatrix}A&B\\C&D\end{pmatrix}\in\GSp_{2n}(F) | A,D\in\GL_n(\mcO), B\in M_{n,n}(\mcO), C\in M_{n,n}(\mathfrak{p})}
\end{align*}
and the element $\ul{\mu}\in T/T_{1}$ is represented by $\diag(\underbrace{\varpi,\ldots,\varpi}_{n},\underbrace{1,\ldots,1}_{n})$.

Note that when $n=2$, this formula recovers Taylor's formula for $L(s,\pi,\Spin)$ explained in Section \ref{sec:Intro} (see \cite[Section~2.4]{MR2636500}).

\subsubsection{Standard $L$-function}
Composing the quotient $\GSpin_{2n+1}\to \SO_{2n+1}$ with the standard representation $\Std$ of $\SO_{2n+1}$, we obtain an irreducible $(2n+1)$-dimensional representation $r = \widetilde{\Std}$ of $\hat{\G}=\GSpin_{2n+1}$.
Its highest weight is given by $\mu=e_1^{\vee}$. 
The other dominant weight is $\mu'=0$, whose multiplicity is one. 
Hence the representation $\widetilde{\Std}$ is quasi-minuscule. 

We have $\langle\rho_{\bfB},\mu\rangle=n$. 
Therefore Corollary \ref{cor:L2} gives 
\begin{align*}
L(s,\pi, \widetilde{\Std})=(1-q^{-s})^{-1}\det\bigl(1-q^{-(s+n)}I_{\chi}(\mathbbm{1}_{\mu})\mid V_{\chi}^{J_{\mu}}\bigr)^{-1},
\end{align*}
where we have
\begin{align*}
J_{\mu}=\Set{\begin{pmatrix}a&b&c\\ {}^t\!d&E&{}^t\!f\\g&h&i\end{pmatrix}\in\GSp_{2n}(F) | \begin{array}{l}
a, i\in \mcO^{\times}, b,f\in M_{1,2n-2}(\mcO), c\in\mcO,\\
d,h\in M_{1,2n-2}(\mathfrak{p}), \\E\in \GL_{2n-2}(\mcO), g\in\mathfrak{p}
\end{array}}
\end{align*}
and the element $\ul{\mu}\in T/T_{1}$ is represented by $\diag(\varpi,\underbrace{1,\ldots,1}_{2n-2},\varpi^{-1})$.


\objectmargin={0.5pt}

\begin{table}[ht]
\begin{center}
\begin{threeparttable}
\caption{All nontrivial quasi-minuscule representations of simple Lie algebras}
\label{Table}
\begin{tabular}{ccccc}\toprule
$\mathfrak{g}$&$\widehat{\mathfrak{g}}$&$r$&$m_0$&$I$\\\hline\hline
$\mathfrak{sl}_n$&$\mathfrak{sl}_n(\C)$&$\wedge^l\C^n$\tnote{*a}&$0$&$\xygraph{\bullet ([]!{+(0,.3)} {1}) - [r] \cdots - [r] *+[F]{\bullet} ([]!{+(0,.3)} {l}) - [r] \cdots - [r] \bullet ([]!{+(0,.3)} {n-1})}$\\
$(n\ge 2)$&&$(1\le l\le n-1)$&&\\
&&$\adj$&$n-1$&$\xygraph{*+[F]{\bullet} - [r] \bullet - [r] \cdots - [r] \bullet - [r] *+[F]{\bullet}}$\\\hline
$\mathfrak{so}_{2n+1}$&$\mathfrak{sp}_{2n}(\C)$&$\C^{2n}$\tnote{*a}&$0$&$\xygraph{!~:{@{=}|@{<}} *+[F]{\bullet} - [r] \bullet - [r] \cdots - [r] \bullet : [r] \bullet}$\\
$(n\ge 2)$&&$(\wedge^2\C^{2n})_0$\tnote{*a}&$n-1$&$\xygraph{!~:{@{=}|@{<}} \bullet - [r] *+[F]{\bullet} - [r] \cdots - [r] \bullet : [r] \bullet}$\\\hline
$\mathfrak{sp}_{2n}$&$\mathfrak{so}_{2n+1}(\C)$&$\spin$&$0$&$\xygraph{!~:{@{=}|@{>}} \bullet - [r] \bullet - [r] \cdots - [r] \bullet : [r] *+[F]{\bullet}}$\\
$(n\ge 2)$&&$\C^{2n+1}$\tnote{*a}&$1$&$\xygraph{!~:{@{=}|@{>}} *+[F]{\bullet} - [r] \bullet - [r] \cdots - [r] \bullet : [r] \bullet}$\\\hline
$\mathfrak{so}_{2n}$&$\mathfrak{so}_{2n}(\C)$&$\C^{2n}$\tnote{*a}&$0$&$\xygraph{*+[F]{\bullet} - [r] \bullet - [r] \cdots - [r] \bullet ( - []!{+(1,.1)} \bullet, - []!{+(1,-.1)} \bullet)}$\\
$(n\ge 4)$&&$\half \spin\times 2$\tnote{*b}&$0$&$\xygraph{\bullet - [r] \bullet - [r] \cdots - [r] \bullet ( - []!{+(1,.1)} *+[F]{\bullet}, - []!{+(1,-.1)} \bullet)}$\\
&&&$0$&$\xygraph{\bullet - [r] \bullet - [r] \cdots - [r] \bullet ( - []!{+(1,.1)} \bullet, - []!{+(1,-.1)} *+[F]{\bullet})}$\\
&&$\adj$&$n$&$\xygraph{\bullet - [r] *+[F]{\bullet} - [r] \cdots - [r] \bullet ( - []!{+(1,.1)} \bullet, - []!{+(1,-.1)} \bullet)}$\\\hline
$\mathfrak{e}_6$&$\mathfrak{e}_6(\C)$&$\C^{27}\times 2$\tnote{*c}&$0$&$\xygraph{*+[F]{\bullet} - [r] \bullet - [r] \bullet ( - []!{+(0,-.5)} \bullet, - [r] \bullet - [r] \bullet)}$\\
&&&$0$&$\xygraph{\bullet - [r] \bullet - [r] \bullet ( - []!{+(0,-.5)} \bullet, - [r] \bullet - [r] *+[F]{\bullet})}$\\
&&$\adj$&$6$&$\xygraph{\bullet - [r] \bullet - [r] \bullet ( - []!{+(0,-.5)} *+[F]{\bullet}, - [r] \bullet - [r] \bullet)}$\\\hline
$\mathfrak{e}_7$&$\mathfrak{e}_7(\C)$&$\C^{56}$\tnote{*d}&$0$&$\xygraph{\bullet - []!{+(.8,0)} \bullet - []!{+(.8,0)} \bullet (- []!{+(0,-.5)} \bullet, - []!{+(.8,0)} \bullet - []!{+(.8,0)} \bullet - []!{+(.8,0)} *+[F]{\bullet})}$\\
&&$\adj$&$7$&$\xygraph{*+[F]{\bullet} - []!{+(.8,0)} \bullet - []!{+(.8,0)} \bullet (- []!{+(0,-.5)} \bullet, - []!{+(.8,0)} \bullet - []!{+(.8,0)} \bullet - []!{+(.8,0)} \bullet)}$\\\hline
$\mathfrak{e}_8$&$\mathfrak{e}_8(\C)$&$\adj$&$8$&$\xygraph{\bullet - []!{+(.7,0)} \bullet - []!{+(.7,0)} \bullet (- []!{+(0,-.5)} \bullet, - []!{+(.7,0)} \bullet - []!{+(.7,0)} \bullet - []!{+(.7,0)} \bullet - []!{+(.7,0)} *+[F]{\bullet})}$\\\hline
$\mathfrak{f}_4$&$\mathfrak{f}_4(\C)$&$\C^{26}$\tnote{*d}&$2$&$\xygraph{!~:{@{=}|@{>}} \bullet - [r] \bullet : [r] \bullet - [r] *+[F]{\bullet}}$\\\hline
$\mathfrak{g}_2$&$\mathfrak{g}_2(\C)$&$\C^7$\tnote{*d}&$1$&$\xygraph{!~:{@3{-}|@{<}} *+[F]{\bullet} : [r] \bullet}$
\\
\bottomrule
\end{tabular}
\small
\begin{tablenotes}
\item[*a]
$\C^n$ denotes the $n$-dimensional representation defining $\widehat{\mathfrak{g}}$, and $(\wedge^2\C^{2n})_0$ denotes the unique nontrivial irreducible component of the $\mathfrak{sp}_{2n}(\C)$-module $\wedge^2\C^{2n}$. 
\item[*b] 
The spin representation of $\mathfrak{so}_{2n}(\C)$ decomposes into the direct sum of two inequivalent irreducible submodules, which are called $\half \spin$. 
\item[*c]
$\C^{27}\times 2$ denotes the two $27$-dimensional irreducible $\mathfrak{e}_6(\C)$-modules which are inequivalent.
\item[*d]
$\C^{56}, \C^{26}, \C^7$ denote the irreducible $56, 26, 7$-dimensional $\widehat{\mathfrak{g}}$-modules, respectively. 
\end{tablenotes}
\end{threeparttable}
\end{center}
\end{table}

\end{document}